\documentclass[makeidx]{amsart}
\usepackage{amsmath}
\usepackage{amssymb}
\usepackage{epsfig}
\usepackage{psfrag}
\usepackage{subfigure}
\usepackage{hyperref}
\newtheorem{Proposition}{Proposition}[section]
  \newtheorem{Remark}[Proposition]{Remark}
  \newtheorem{Corollary}[Proposition]{Corollary}
  \newtheorem{Lemma}[Proposition]{Lemma}
  
  \newtheorem{Theorem}{Theorem}[section]
 
 \newtheorem{Definition}[Proposition]{Definition}
 \newtheorem{Note}[Proposition]{Note}
 
 \newtheorem{Assumptions}{Assumptions}
\newcommand {\z}{{\noindent}}
\def\blackslug{\hbox{\hskip 1pt \vrule width 4pt height 8pt depth 1.5pt
\hskip 1pt}}
\def\qed{\quad\blackslug\lower 8.5pt\null\par}

\def\CC{\mathbb{C}}
\def\RR{\mathbb{R}}
\def\TT{\mathbb{T}}
\def\NN{\mathbb{N}}

\def\Re{\mathrm{Re}}
\def\Im{\mathrm{Im}}

\def\ds{\text{\raisebox{-2pt}{$\stackrel{_{\displaystyle *}}{*}$}}}

\makeindex
\begin{document}

\author{O. Costin$^1$, G. Luo$^{1}$ and S.  Tanveer$^1$ }\title{
Integral formulation of 3-D Navier-Stokes and longer time existence of
smooth solutions.}
\gdef\shortauthors{O.  Costin, G. Luo \& S. Tanveer}
\gdef\shorttitle{Navier-Stokes Equation}
\thanks{$1$.  Department of Mathematics, Ohio State University.}
\maketitle

\today

\bigskip

\begin{abstract}

We consider the 3-D Navier-Stokes  initial value problem,
$$
v_t - \nu \Delta v = -\mathcal{P} \left [ v \cdot \nabla v \right ] + f
~~,~~v(x, 0) = v_0 (x), ~~x \in \mathbb{T}^3\ \ \ (*)
$$
where $\mathcal{P}$ is the Hodge projection to divergence-free vector
fields. We assume that the Fourier transform norms $ \| {\hat f} \|_{l^1
(\mathbb{Z}^3 )}$ and $\| {\hat v}_0 \|_{l^{1} (\mathbb{Z}^3)}$ are
finite.  Using an inverse Laplace transform approach, similar to the
earlier work \cite{NS16}, we prove that an integral equation equivalent to
(*) has a unique solution ${\hat U} (k, q)$, exponentially bounded for $q$
in a sector centered on $\RR^+$, where $q$ is the inverse Laplace dual to
$1/t^n$ for $n \ge 1$.

This implies in particular local existence of a classical solution to (*)
for $t \in (0, T)$, where $T$ depends on  $\| {\hat v}_0 \|_{l^{1}}$ and
$\| {\hat f} \|_{l^1}$. Global existence of the solution to NS follows if
$\| {\hat U} (\cdot, q) \|_{l^1}$ has subexponential bounds as
$q\to\infty$.

If $f=0$, then the converse is also true: if NS has global solution, then
there exists $n \ge 1 $ for which $\| {\hat U} (\cdot, q) \|$ necessarily
decays.  More generally, if the exponential growth rate in $q$ of $ {\hat
U}$ is $\alpha$, then a classical solution to NS exists for $t\in\left (0,
\alpha^{-1/n} \right )$.

We show that $\alpha$ can be better estimated based on the values of
${\hat U}$ on a finite interval $[0,q_0]$. We also show how the integral
equation  can be solved numerically with controlled errors.

Preliminary numerical calculations of the integral equation over a modest
$[0, 10]$ $q$-interval for $n=2$ corresponding to Kida (\cite{Kida})
initial conditions, though far from being optimized or rigorously
controlled, suggest that this approach gives an existence time for 3-D
Navier-Stokes that substantially exceeds classical estimate.

\end{abstract}

\section{Introduction}

We consider the 3-D Navier-Stokes (NS)  initial value problem
\begin{equation}\label{nseq0}
v_t - \nu \Delta v = -\mathcal{P} [ v \cdot \nabla v ] + f(x),\ \ v(x, 0)
= v_0 (x), \ \ x\in\TT^3 [0,2\pi] ,\ \ t\in\RR^+
\end{equation}
where $v$ is the fluid velocity and $\mathcal{P} = I -\nabla \Delta^{-1}
(\nabla \cdot )$ is the Hodge projection operator to the space of
divergence-free vector fields. For simplicity we assume that the forcing
$f$ is time-independent.

In Fourier space, (\ref{nseq0}) can be written as
\begin{equation}\label{nseq}
{\hat v}_t + \nu |k|^2 {\hat v} = - i k_j P_k \left [ {\hat v}_j {\hat *}
{\hat v}  \right ] + {\hat f} ~~,~~{\hat v} (k, 0)= {\hat v}_0,
\end{equation}
where ${\hat v} (k, t) = \mathcal{F} \left [ v (\cdot, t) \right ] (k)$ is
the Fourier transform of the  velocity, ${\hat *}$ denotes Fourier
convolution, a repeated index $j$ indicates summation over $j =1,2,3$ and
$P_k=\mathcal{F}(\mathcal{P})$ is the Fourier space representation of the
Hodge projection operator on the space of divergence-free vector fields,
given explicitly by
\begin{equation}\label{8.0}
P_k \equiv 1 - \frac{k ( k \cdot )}{|k|^2}.
\end{equation}
We assume that $ {\hat v}_0$ and $ {\hat f} \in l^1 (\mathbb{Z}^3)$ and,
without loss of generality, that the average velocity and force in the
periodic box are zero, and hence ${\hat v} (0, t) =0={\hat f} (0)$.

Global existence of smooth solutions to the 3-D Navier-Stokes problem
remains a formidable open mathematical problem, even for zero forcing,
despite extensive research in this area.  The problem is important not
only in mathematics but it has wider impact, particularly if singular
solutions exist.  It is known \cite{Beale} that the singularities can only
occur if $\nabla v$ blows up.  This means that near a potential blow-up
time, the relevance of NS to model fluid flow becomes questionable, since
the linear approximation in the constitutive stress-strain relationship,
the assumption of incompressibility and even the continuum hypothesis
implicit in derivation of NS become doubtful.  In some physical problems
(such as inviscid Burger's equation) the singularity of an idealized
approximation is mollified by inclusion of regularizing effects. It may be
expected that if 3-D NS solutions exhibited blow up, then actual fluid
flow, on very small time and space scales, has to involve parameters other
than those considered in NS. This could profoundly affect our
understanding of small scale in turbulence.  In fact, some 75 years back,
Leray \cite{Leray1}, \cite{Leray2}, \cite{Leray3} was motivated to study
weak solutions of 3-D NS, conjecturing that turbulence was related to
blow-up of smooth solutions.

The typical method used in the mathematical analysis of NS, and of more
general PDEs, is the so-called energy method. For NS, the energy method
involves {\it a priori} estimates on the Sobolev $\mathbb{H}^m$ norms of
$v$. It is known that if $ \| v (\cdot, t) \|_{\mathbb{H}^1}$ is bounded,
then so are all the higher order energy norms $\| v (\cdot, t)
\|_{\mathbb{H}^m}$ if they are bounded initially.  The condition on $v$
has been further weakened \cite{Beale} to $\int_0^{t}\| \nabla \times v
(\cdot, t) \|_{L^\infty}dt <\infty$. Prodi \cite{Prodi} and Serrin
\cite{Serrin} have found a family of other controlling norms for classical
solutions \cite{Lady}. In particular, no singularity is possible if $ \| v
(\cdot, t) \|_{L^\infty}$ is bounded.  The ${L^{3}} $ norm is also
controlling, as has been recently shown in \cite{Sverak}. For classical
solutions, global existence proofs exist only for small initial data and
forcing or for large viscosity ({\it i.e.}  when the non-dimensional
Reynolds number is small). On a sufficiently small initial interval the
solution is classical and unique.  Global weak solutions (possibly
non-unique) are only known to exist \cite{Leray1}, \cite{Leray2},
\cite{Leray3} in a space of functions for which $\nabla v$ can blow-up on
a small set in space-time\footnote{The 1-D Hausdorff measure of the set of
blow-up points in space-time is known to be zero \cite{Caffarelli}}.
However, when $f=0$ (no forcing), a time $T_c$ may be estimated in terms
of the $\| v_0 \|_{H^1} $ beyond which any weak Leray solution becomes
smooth again. Such an estimate, which also follows directly from Leray's
observation on the cumulative dissipation being bounded, is worked out in
the Appendix. \footnote{We are grateful to Alexey Cheskidov for pointing
out the fact that classical estimates are easily obtainable.}

Classical energy methods have so far failed to give global existence
because of failure to obtain conservation laws involving any of the
controlling norms \cite{Tao}.

Numerical solutions to (\ref{nseq0})  are physically revealing but do not
shed enough light into the existence issue.  Indeed, the numerical errors
in Galerkin/finite-difference/finite-element approximations depend on
derivatives of $v$ that are not known to exist {\it a priori} beyond an
initial time interval.

This paper introduces a new method in approaching these issues. In our
formulation, the velocity $v (x, t)$ is obtained as a Laplace transform:
\begin{equation}\label{intro.1}
v (x, t) = v_0 (x) + \int_0^\infty U (x, q) e^{-q/t^n} dq ,\qquad n\ge 1
\end{equation}
where $U$ satisfies an integral equation (IE) which always has a unique
acceptable smooth solution.  Looking for $v$ in this form is motivated by
our earlier work \cite{NS16} showing that small $t$ formal series
solutions, which exist for analytic initial conditions and forcing, are
Borel summable to actual solutions. In that case, the actual solution is
indeed in the form (\ref{intro.1}), with $n=1$. However, the
representation (\ref{intro.1}), the IE for ${\hat U} (k, q) \equiv
\mathcal{F} \left [ U (\cdot, q) \right ] (k)$,(\ref{IntUeqn}), and its
properties important to (\ref{intro.1}) are valid even when $f$, $v_0 $
and the corresponding solutions are not analytic in $x$.  An overview of
our approach and nature of results is given in \cite{JTS}.

\begin{Note}\label{fn1}
For general initial data and forcing, $U$ is in $L_1 \left (\mathbb{R}^+,
e^{-\alpha q} dq \right ) $, as defined in (\ref{eq:eqL}). If $n>1$, then
$U$ is analytic in $q$ in an open sector. For $n = 1$, the solution is
$q$-analytic in a neighborhood of $\mathbb{R}^+ \cup \{0 \}$
\footnote{This together with the $L^1$ estimate proves Borel summability
of the small $t$ series.} iff ${\hat v} (x,0)$ and ${\hat f} (x)$ are
analytic in $x$.
\end{Note}

As it will be seen later, using $n>1$ is advantageous for some initial
data.

In Fourier space, (\ref{intro.1}) implies
\begin{equation}\label{intro.1.1}
{\hat v} (k,t) = {\hat v}_0 (k) + \int_0^\infty {\hat U} (k, q) e^{-q/t^n}
dq.
\end{equation}
{\bf Notation.} Variables in the Fourier domain are marked with a hat
$\hat{}$, Laplace convolution is denoted by $*$, Fourier-convolution by
${\hat *}$, while $\ds$ denotes Fourier followed by Laplace convolution
(their order is unimportant).

As seen in \S4, $\hat{U}$ satisfies  the following IE:
\begin{equation}\label{IntUeqn}
{\hat U} (k, q) = -i k_j \int_0^q \mathcal{G} (q, q'; k) {\hat H}_j (k,
q') dq' + {\hat U}^{(0)} (k, q) =:\mathcal{N} \left [ {\hat U} \right ]
(k, q),
\end{equation}
where
\begin{equation}\label{7.2}
{\hat H}_j (k, q) = P_k \left [ {\hat v}_{0,j} {\hat *} {\hat U} + {\hat
U}_j {\hat *} {\hat v}_0 + {\hat U}_j  \ds {\hat U} \right ] (k, q).
\end{equation}
The kernel $\mathcal{G}$, the inhomogeneous term ${\hat U}^{(0)} (k, q)$
and their essential properties are given in (\ref{kernelG}) and
(\ref{eqU0}) in \S4.

\begin{Note}{
\rm The solutions of (\ref{IntUeqn}), needed on $\RR^+$, are very regular,
see Note \ref{fn1}. The existence time of $\hat{v}$ is determined by the
behavior of $\hat{U}$ for large $q$. In this formulation, global existence
of $\hat{v}$ is equivalent to subexponential behavior of $\hat{U}$.}
\end{Note}

The IE formulation was first introduced in \cite{NS16} in a narrower
context, and provides a new approach towards solving IVPs.

\section{Main results}
We define
\begin{equation}\label{eq:eqL}
L_{1}(\RR^+,e^{-\alpha q}dq)=\left\{g:\RR^+\mapsto \CC ~\Big|
\int_0^{\infty} e^{-\alpha q} |g(q)| dq < \infty\right\}.
\end{equation}
\begin{Assumptions}\label{A11}
In the following, unless otherwise specified, we assume that ${\hat v}_0$
and ${\hat f}$ are in $l^1 \left (\mathbb{Z}^3 \right )$, ${\hat v}_0 (0)
= 0 = {\hat f} (0)$,   $n\ge 1$, $\nu>0$ and $\alpha$ in (\ref{eq:eqL}) is
large enough (see  Proposition \ref{propthm01}).
\end{Assumptions}

\smallskip

\begin{Theorem}
\label{Thm01}

(i) Eq. (\ref{IntUeqn}) has a unique solution ${\hat U} (\cdot, q)\in
L_{1}(\RR^+,e^{-\alpha q}dq)$. For $n > 1$ this solution is analytic in an
open sector, cf. Note \ref{fn1}. We let $U(x,q)=\mathcal{F}^{-1} \left [
{\hat U} (\cdot, q) \right ] (x)$.

(ii) With this ${\hat U} $, $\hat{v}$ in (\ref{intro.1.1}) ($v(x,t)$ in
(\ref{intro.1}) respectively) is a classical solution of (\ref{nseq}) (
(\ref{nseq0}), resp.) for $t \in \left ( 0, \alpha^{-1/n} \right )$.

(iii) Conversely, any classical solution of (\ref{nseq0}), $v (x, t)$,
$t\in (0,T_0)$ has a Laplace representation of the form (\ref{intro.1})
with $ U$ as in (i) and with
$$
{\hat U} (k, q) := \mathcal{L}^{-1} \left [ \mathcal{F} [ v (\cdot,
\tau^{-1/n}) ] (k) - \mathcal{F} [v_0 ] (k) \right ] (q)
$$
a solution of (\ref{IntUeqn}) in $L_{1}(\RR^+,e^{-\alpha q}dq)$, $\alpha>
T_0^{-n}$.
\end{Theorem}

\z The proof is given at the end of \S5.

\begin{Remark}
{\rm Proposition \ref{propthm01} below provides (relatively rough)
estimates on $\alpha$.  Theorem \ref{Thm03} gives sharper bounds in terms
of the values of ${\hat U}$ on a finite interval $[0, q_0]$.  Smaller
bounds on $\alpha$ entail smooth solutions of (\ref{nseq0}) over a longer
time.}

\smallskip

We have the following result which, in a sense, is a converse of Theorem
(\ref{Thm01}).
\begin{Theorem}\label{Thm02}
For $f = 0$, if (\ref{nseq0}) has a global classical solution, then for
all sufficiently large $n$, $ U (x, q)=O(e^{-C_nq^{1/(n+1)}})$ as $q \to
+\infty$,  for some $C_n>0$.
\end{Theorem}

\end{Remark}

\z The proof is given in \S7.
\begin{Corollary}
Theorems \ref{Thm01} and \ref{Thm02} imply that global existence is
equivalent to an asymptotic problem: $\hat{v}$ exists for all time iff
$\hat{U}$ decays in $q$ for some $n \in \mathbb{Z}^+$.
\end{Corollary}

The existence interval $\left ( 0, \alpha^{-1/n} \right )$ guaranteed by
Theorem~\ref{Thm01} is suboptimal. It does not take into account the fact
that the initial data $v_0 $ and forcing $f$ are real valued. (Blow up of
Navier-Stokes solution for complex initial conditions is known to occur
\cite{Sinai}).  Also, the estimate ignores possible cancellations in the
integrals.

In the following we address the issue of sharpening the estimates, in
principle arbitrarily well, based on more detailed knowledge of the
solution of the IE on an interval $[0,q_0]$.  This knowledge may come, for
instance, from  computer assisted estimates or from rigorous bounds based
on optimal truncation of asymptotic series. If this information shows that
the solution is sufficiently small for $q$ near the right end of the
interval, then $\alpha$ can be shown to be small. This in turn results in
longer times of guaranteed existence, possibly global existence for $f=0$
if this time exceeds $T_c$, the time after which it is known that a weak
solution becomes classical.

\section{Sharpening the estimates; rigorous numerical analysis}
Let ${\hat U} (k, q)$ be the solution of (\ref{IntUeqn}), provided by
Theorem~\ref{Thm01}. Define
\begin{equation}\label{eq:eq456}
\hat {U}^{(a)} (k, q) =
\begin{cases}
  \hat{U} (k, q) & \mbox{for $q \in (0, q_0] \subset \RR^+$} \\
  0 & \mbox{otherwise}
\end{cases},
\end{equation}
\begin{multline*}
\hat{U}^{(s)}(k,q) = -ik_{j} \int_{0}^{\min \{ q, 2 q_0 \}}
\mathcal{G}(q,q';k) \hat{H}_{j}^{(a)}(k,q')\,dq' + \hat{U}^{(0)}(k,q), \\
\hat{H}_{j}^{(a)}(k,q)  = P_{k} \Bigl[ \hat{v}_{0,j} \hat{*} \hat{U}^{(a)}
+ \hat{U}_{j}^{(a)} \hat{*} \hat{v}_{0} + \hat{U}_{j}^{(a)} \ds
\hat{U}^{(a)} \Bigr](k,q).
\end{multline*}

Using (\ref{eq:eq456}) we introduce the following functionals of ${\hat
U}^{(a)} (k, q)$, ${\hat v}_0$ and ${\hat f}$:
\begin{align}\label{ext3.1.0}
& b: = \alpha^{1/2+1/(2n)} \int_{q_{0}}^{\infty} e^{-\alpha q}
\|\hat{U}^{(s)}(\cdot,q)\|_{l^1}\,dq, \\
& \epsilon_1 = \Gamma \biggl( \frac{1}{2}+\frac{1}{2n} \biggr) \biggl[
B_{1} + \int_{0}^{q_{0}} e^{-\alpha q'} B_{2}(q')\,dq' \biggr], \\
& \epsilon = \Gamma \biggl( \frac{1}{2}+\frac{1}{2n} \biggr)\, B_{3},
\end{align}
where
\begin{multline*}
B_{1} = 4\sup_{k \in \mathbb{Z}^{3}} \big\{ |k|B_{0}(k) \big\}
\|\hat{v}_{0}\|_{l^1}, \qquad B_{0}(k)  = \sup_{q_{0} \leq q' \leq q}
\Big\{ (q-q')^{1/2-1/(2n)} |\mathcal{G}(q,q';k)| \Big\}, \\
B_{2}(q) = 4\sup_{k \in \mathbb{Z}^{3}} \big\{ |k|B_{0}(k) \big\}
\|\hat{U}^{(a)}(\cdot,q)\|_{l^1},  \qquad B_{3}  = 2\sup_{k \in
\mathbb{Z}^{3}} \big\{ |k|B_{0}(k) \big\}.
\end{multline*}

\bigskip

\begin{Theorem}\label{Thm03}
The exponential growth rate $\alpha$ of $\hat{U}$ is estimated in terms of
the restriction of  $\hat{U}$ to $[0,q_0]$ as follows.
\begin{equation}\label{ext9}
\text{\rm If \ \ \ }\alpha^{1/2+1/(2n)} >\epsilon_{1} + 2 \sqrt{\epsilon
b}\ \ \text{\rm then\ \ }\int_0^{\infty} \| \hat{U} (\cdot,q)\|_{l^1}
e^{-\alpha q}dq <\infty.
\end{equation}
\end{Theorem}
\z The proof of Theorem \ref{Thm03} is given in \S\ref{S6}.
\begin{Remark}{
\rm In \S \ref{furtherest}, it is shown that for a given {\em global
classical} solution to (\ref{nseq0}), in adapted variables, the quantity
$\epsilon_{1} + 2 \sqrt{\epsilon b}$ is small for large $q_0$.}
\end{Remark}

\bigskip

\begin{Remark}{
\rm In the proof it is also seen that if $\| {\hat U}^{(a)} (\cdot, q)
\|_{l^1}$ is small enough in a sufficiently large subinterval $[q_d,
q_0]$, then the right side of (\ref{ext9}) is small, implying a large
existence time $\left (0, \alpha^{-1/n} \right )$ of a classical solution
$v$. The guaranteed existence time  is larger if  $q_0$ is larger. If for
$f=0$, the estimated $ \alpha^{-1/n}$ exceeds $T_c$, the time for Leray's
weak solution to become  classical again (see Appendix), then global
existence of a classical solution $v$ follows.}
\end{Remark}

Since the improved estimates in Theorem \ref{Thm03} rely on the values of
$\hat{U}$ on a sufficiently large initial interval, we analyze the
properties of a discretized scheme for numerical computation of $\hat{U}$
with {\em controlled errors}.
\begin{Definition}\label{defDNorm}
We introduce the following norm on functions defined on a $\delta $-grid
in $q$
$$
\| {\hat W} \|^{(\alpha, \delta)} = \sup_{m_s \le m \in \mathbb{Z}^{+}}
m^{1-1/n} \delta^{1-1/n} (1+m^2 \delta^2) e^{-\alpha m \delta} \| {\hat W
} (\cdot, m\delta) \|_{l^1}.
$$
\end{Definition}
\begin{Theorem}\label{Thm04}
Consider a discretized integral equation consistent with (\ref{IntUeqn}) (
cf. definition \ref{defconsis}) based on Galerkin truncation to $[-N,N]^3$
Fourier modes and uniform discretization in $q$,
$$
{\hat U}_{\delta}^{(N)} = \mathcal{N}_{\delta}^{(N)} \left [ {\hat
U}_\delta^{(N)} \right ],
$$
see (\ref{discret2}) below. Then, the error ${\hat U} - {\hat
U}^{(N)}_\delta$ at the points $q = m \delta$, satisfies
\begin{multline}\label{eq:eu1}
\| {\hat U} (\cdot, m \delta) - {\hat U}^{(N)}_\delta (\cdot, m\delta )
\|_{l^1} \\
\le \left [ 2 \| T_{E,N} \|^{(\alpha, \delta)} + 2 \| T_{E,\delta}^{(N)}
\|^{(\alpha, \delta)} + \| (I-\mathcal{P}_N) {\hat U} \|^{(\alpha,
\delta)} \right ] \frac{ e^{\alpha m \delta} }{ m^{1-1/n} \delta^{1-1/n}
(1+m^2 \delta^2)}
\end{multline}
for $m \ge m_s \in \mathbb{Z}^+$, where $m_s \delta=:q_m > 0$ is
independent of $\delta$. In (\ref{eq:eu1}), $T_{E,N}$ is the truncation
error due to Galerkin projection $\mathcal{P}_N$ and $T_{E,\delta}^{(N)}$
is the truncation error due to the $\delta$-discretization in $q$ for a
given $N$. We have $\| T_{E,N} \|^{(\alpha, \delta)}, \| (I-\mathcal{P}_N)
{\hat U} \|^{(\alpha, \delta)} \to 0$ as $N \to \infty$ for any $\delta$
and $ \| T^{(N)}_{E,\delta} \|^{(\alpha, \delta)} \to 0 $ as $\delta \to
0$, uniformly in $N$.
\end{Theorem}

\begin{Remark}{
\rm For small $q$, independent of $\delta$, an asymptotic expansion of
${\hat U}$ exists, and solving the equation numerically for $q \in [0,
q_m]$ can be avoided. For this reason we start with $q=q_m$.}
\end{Remark}

\section{Integral Equation (\ref{IntUeqn}) and its properties}

We define ${\hat u}$ through the decomposition
\begin{equation}\label{hatudef}
{\hat v} (k, t) = {\hat v}_0 (k) + {\hat u} (k, t).
\end{equation}
Then, (\ref{nseq}) implies
\begin{multline}\label{hatueq}
{\hat u}_t + \nu |k|^2 {\hat u} = - i k_j P_k \left [ {\hat v}_{0,j} {\hat
*} {\hat u} + {\hat u}_j {\hat *} {\hat v}_0 + {\hat u}_j {\hat *} {\hat
u} \right ] + {\hat v}_1 (k) =:-i k_j {\hat h}_j (k, t) + {\hat v}_1 (k),
\end{multline}
where ${\hat v}_1$ is given by (\ref{4}). Using ${\hat v} (k, 0) = {\hat
v}_0$, we have ${\hat u} (k, 0) = 0$ and we obtain from (\ref{hatueq}),
\begin{equation}\label{5}
{\hat u} (k, t) = -i k_j \int_0^t e^{-\nu |k|^2 (t-s)} {\hat h}_j (k, s)
ds + {\hat v}_1 (k) \left ( \frac{1-e^{-\nu |k|^2 t} }{\nu |k|^2}  \right
).
\end{equation}
We look for $\hat{u}$ in the form of a  Laplace transform
\begin{equation}\label{65}
{\hat u} (k, t) = \int_0^\infty {\hat U} (k, q) e^{-q/t^n} dq;\ \ n\ge 1
\end{equation}
We apply the inverse Laplace transform of (\ref{5}) with respect to $\tau
= 1/t^n$ (justified at the end of the proof of Lemma \ref{connection},
with more details in the Appendix) to obtain (\ref{IntUeqn}). The inverse
Laplace transform of $f$ is given, as usual, by
\begin{equation}\label{eq:deflaplace}
\left [ \mathcal{L}^{-1}f \right ] (p) = \frac{1}{2\pi
i}\int_{c-i\infty}^{c+i\infty} f(s)e^{ps}ds,
\end{equation}
where $c$ is chosen so that $f$ is analytic and has suitable asymptotic
decay for  $\Re ~s \ge  c$.

For $n=1$ the kernel $\mathcal{G}$ is given by, see \cite{NS16},
\begin{multline}\label{kernelG1}
\mathcal{G} (q, q'; k) = \frac{\pi z'}{z} \left ( J_1 (z) Y_1 (z') - J_1
(z') Y_1 (z) \right )\\{\rm where} ~z = 2 |k| \sqrt{\nu q} ~~,~~z'= 2 |k|
\sqrt{\nu q'},   \ \ (n=1)
\end{multline}
$J_1$ and $Y_1$ are Bessel functions of order 1, and
\begin{equation}\label{eqU01}
\hat{U}^{(0)} (k,q) = 2 {\hat{v}_{1}(k)} \frac{J_1 (z)}{z}  ~~,~~{\rm
where} ~z= 2 |k| \sqrt{\nu q}
\end{equation}
For $n \ge 2$ the kernel has the form (derived in the Appendix, see
(\ref{Gdefineac}))
\begin{multline}\label{kernelG}
\mathcal{G}(q,q';k) = \int_{(q'/q)^{1/n}}^1 \left \{ \frac{1}{2 \pi i}
\int_{c-i \infty}^{c+i \infty} \tau^{-1/n} \exp \left [ - \nu |k|^2
\tau^{-1/n} (1-s) + (q-q' s^{-n} ) \tau \right ] d\tau \right \} ds \\
= \frac{\gamma^{1/n}}{\nu^{1/2} |k| q^{1-1/(2n)}} \int_{1}^{\gamma^{-1/n}}
(1-s^{-n})^{1/(2n)-1} (1-s\gamma^{1/n})^{-1/2} \mu^{1/2} F(\mu)\,ds,
\end{multline}
where
$$
\gamma = \frac{q'}{q}, \qquad \mu = \nu |k|^{2} q^{1/n} (1-s\gamma^{1/n})
(1-s^{-n})^{1/n},
$$
\begin{equation}\label{eqF}
F(\mu) = \frac{1}{2\pi i} \int_{C} \zeta^{-1/n} e^{\zeta-\mu
\zeta^{-1/n}}\,d\zeta,
\end{equation}
and $C$ is a contour starting at $\infty e^{-i \pi}$ and ending at $\infty
e^{i \pi}$ turning around the origin counterclockwise. The function
$\hat{U}^{(0)}(k,q)$ in (\ref{IntUeqn}) is defined by
\begin{equation}\label{eqU0}
\hat{U}^{(0)}(k,q) = \frac{{\hat v}_1 (k)}{\nu |k|^2} \mathcal{L}^{-1}
\left \{ 1 - \exp \left [ - \nu |k|^2 \tau^{-1/n} \right ] \right \} (q) =
\frac{\hat{v}_{1}(k)}{\nu |k|^{2} q} G(\nu |k|^{2} q^{1/n}),
\end{equation}
where
\begin{equation}\label{intG}
G(\tilde{\mu}) = -\frac{1}{2\pi i} \int_{C} e^{\zeta-{\tilde \mu}
\zeta^{-1/n}}\,d\zeta.
\end{equation}

\subsection{Properties of $F$, $G$, $\mathcal{G}$, ${\hat U}^{(0)} $ and
the relation between the IE and NS}

\begin{Lemma}\label{lemG}
The functions $F$, $G$ in (\ref{eqF}) and (\ref{intG}) are entire and
$G^\prime (\mu) = F(\mu)$. Furthermore $F(0) = \frac{1}{\Gamma (1/n)}$,
$G(0) = 0$ and, for $n\ge 2$, their asymptotic behavior for large positive
$\mu$ is given by
\begin{equation}\label{ar1}
F(\mu) \sim
\begin{cases}
\displaystyle \sqrt{\frac{2}{\pi (n+1)}} n^{\frac{3}{2 (n+1)}}
\mu^{\frac{n-2}{2 (n+1)} } \Im \left \{ \exp \left [ \frac{3 i \pi}{2
(n+1)} \right ] e^{-z} \right \} & {\rm if } \arg\mu=0 \\
\displaystyle -i \sqrt{\frac{1}{2 \pi (n+1)} } n^{\frac{3}{2 (n+1)}}
\mu^{\frac{n-2}{2 (n+1)} } \exp \left [ \frac{3 i \pi}{2 (n+1)} \right ]
e^{-z} & {\rm if } \arg \mu \in (0, \frac{n+3}{2 n} \pi ) \\
\displaystyle  i \sqrt{\frac{1}{2 \pi (n+1)} } n^{\frac{3}{2 (n+1)}}
\mu^{\frac{n-2}{2 (n+1)} } \exp \left [ \frac{-3 i \pi}{2 (n+1)} \right ]
e^{-{\hat z}} & {\rm if } \arg\mu \in (-\frac{n+3}{2 n} \pi, 0 )
\end{cases}
\end{equation}
\begin{equation}\label{ar2}
G(\mu) \sim
\begin{cases}
\displaystyle -\sqrt{\frac{2}{\pi (n+1)} } n^{\frac{1}{2 (n+1)}}
\mu^{\frac{n}{2 (n+1)} } \Im \left \{ \exp \left [ \frac{i \pi}{2 (n+1)}
\right ] e^{-z} \right \} & {\rm if } \arg\mu=0 \\
\displaystyle i \sqrt{\frac{1}{2 \pi (n+1)} } n^{\frac{1}{2 (n+1)}}
\mu^{\frac{n}{2 (n+1)} } \exp \left [ \frac{i \pi}{2 (n+1)} \right ]
e^{-z} & {\rm if }  \arg \mu \in (0, \frac{n+3}{2 n} \pi ) \\
\displaystyle -i \sqrt{\frac{1}{2 \pi (n+1)} } n^{\frac{1}{2 (n+1)}}
\mu^{\frac{n}{2 (n+1)} } \exp \left [ \frac{-i \pi}{2 (n+1)} \right ]
e^{-{\hat z}} & {\rm if } \arg\mu \in (-\frac{n+3}{2 n} \pi, 0 )
\end{cases}
\end{equation}
where
\begin{equation}\label{eq:z}
z =\xi_0 \mu^{n/(n+1)} e^{i \pi/(n+1)} ~,~ ~~\xi_0 = n^{-n/(n+1)} ~(n+1);\
{\hat z} = \xi_0 \mu^{n/(n+1)} e^{-i \pi/(n+1)}.
\end{equation}
\end{Lemma}
\begin{proof}
These results follow  from  standard steepest descent analysis and from
the ordinary differential equation that $F$ and $G$ satisfy, see \S
\ref{A12}.
\end{proof}

\begin{Remark}
We see that $F(\mu)$ and $G(\mu)$ are exponentially small for large $\mu$
when $\arg \mu \in \left (-\frac{(n-1)\pi}{2n} , \frac{(n-1) \pi}{2n}
\right ) $, that is, when $\arg q \in \left (-\frac{(n-1) \pi}{2},
\frac{(n-1) \pi}{2} \right )$.
\end{Remark}

\begin{Definition}
For $\delta > 0$ and $n \ge 2$ we define the sector
$$
\mathcal{S}_\delta := \left \{ q: \arg q \in \left ( - \frac{(n-1) \pi}{2}
+\delta, \frac{(n-1) \pi}{2} -\delta \right ) \right \}.
$$
\end{Definition}

\begin{Lemma}\label{L2.4}
For $n \ge 2$, $q, q' \in e^{i \phi} \mathbb{R}^+ \subset
\mathcal{S}_\delta $, with $0 < |q'| \le |q| < \infty$ and $k \in
\mathbb{Z}^3 $ we have
$$
|\mathcal{G}(q,q';k)| \leq \frac{C_{2}
|q-q'|^{\frac{1}{2n}-\frac{1}{2}}}{\nu^{1/2} |k| |q|^{1/2}},
$$
where $C_{2}$ only depends on $\delta$. For $n=1$, the same inequality
holds for $q, q' \in \mathbb{R}^+$ with $0 < q' \le q$.
\end{Lemma}
\begin{proof}
The case $n=1$ follows from the behavior of $J_1$ and $Y_1$, see
\cite{NS16}\footnote{In that paper the viscosity $\nu$ was scaled to 1.}.
For $n \ge 2$, it follows from Lemma \ref{lemG} that $|\mu^{1/2} F(\mu)|$
is bounded, with a bound dependent on $\delta$. Below, $C$ is a generic
constant, possibly $\delta$ and $n$ dependent. {F}rom (\ref{kernelG}) we
get
\begin{multline}
|\mathcal{G}(q,q';k)| \leq \frac{C \gamma^{1/n}}{\nu^{1/2} |k|
|q|^{1-1/(2n)}} \Biggl[ \int_{1}^{\frac{1}{2}(1+\gamma^{-1/n})} +
\int_{\frac{1}{2}(1+\gamma^{-1/n})}^{\gamma^{-1/n}} \Biggr] \\
\times (1-s^{-n})^{1/(2n)-1} (1-s\gamma^{1/n})^{-1/2}\,ds =: \frac{C
\gamma^{1/n}}{\nu^{1/2} |k| |q|^{1-1/(2n)}} \Bigl( I_{1} + I_{2} \Bigr) ;\
\ \text{where } \gamma = \frac{q'}{q}
\end{multline}
For $s \in \left (1,\frac{1}{2} (1+\gamma^{-1/n}) \right ]$ we have
\begin{multline}
(1-s^{-n})^{1/(2n)-1} (1-s\gamma^{1/n})^{-1/2} \leq \biggl(
\frac{s-1}{1-s^{-n}} \biggr)^{1-1/(2n)} (s-1)^{1/(2n)-1} \biggl(
\frac{1}{2} - \frac{1}{2} \gamma^{1/n} \biggr)^{-1/2} \\
\leq C \Bigl[ 1 + (s-1)^{1-1/(2n)} \Bigr] (s-1)^{1/(2n)-1} \biggl(
\frac{1}{2} - \frac{1}{2} \gamma^{1/n} \biggr)^{-1/2} \\
\leq C (1-\gamma^{1/n})^{-1/2} \Bigl[ 1 + (s-1)^{1/(2n)-1} \Bigr],
\end{multline}
and for $s \in [\frac{1}{2} (1+\gamma^{-1/n}),\gamma^{-1/n})$,
\begin{multline*}
(1-s^{-n})^{1/(2n)-1} (1-s\gamma^{1/n})^{-1/2} \leq C \Bigl[ 1 +
(s-1)^{1/(2n)-1} \Bigr] (1-s\gamma^{1/n})^{-1/2} \\
\leq C (1-\gamma^{1/n})^{1/(2n)-1} (1-s\gamma^{1/n})^{-1/2}.
\end{multline*}
Thus
\begin{multline*}
I_{1} \leq C (1-\gamma^{1/n})^{-1/2}
\int_{1}^{\frac{1}{2}(1+\gamma^{-1/n})} \Bigl[ 1 + (s-1)^{1/(2n)-1}
\Bigr]\,ds \\
\leq C \gamma^{-1/n} (1-\gamma^{1/n})^{-1/2} \Bigl[ (1-\gamma^{1/n}) +
(1-\gamma^{1/n})^{1/(2n)} \Bigr] \\
\leq C \gamma^{-1/n} (1-\gamma^{1/n})^{1/(2n)-1/2}, \\
I_{2} \leq C (1-\gamma^{1/n})^{1/(2n)-1} \int_{\frac{1}{2}(1
+\gamma^{-1/n})}^{\gamma^{-1/n}} (1-s\gamma^{1/n})^{-1/2}\,ds \\
\leq C \gamma^{-1/n} (1-\gamma^{1/n})^{1/(2n)-1/2}.
\end{multline*}
\end{proof}

\begin{Lemma}\label{lemU0}
(i) For $n \ge 2$ and $0\ne q \in \mathcal{S}_\delta$, we have for $\alpha
\ge 1$,
$$
\| {\hat U}^{(0)} (\cdot, q) \|_{l^1} \le c_1 \| {\hat v}_1 \|_{l^1}
|q|^{-1+1/n} \exp \left [ - c_2 \nu^{n/(n+1)} |q|^{1/(n+1)} \right ],
$$
$$
\| k {\hat U}^{(0)} (\cdot, q) \|_{l^1} \le c_1 \| |k| {\hat v}_1 \|_{l^1}
|q|^{-1+1/n} \exp \left [ - c_2 \nu^{n/(n+1)} |q|^{1/(n+1)} \right ],
$$
where $c_1$ and $c_2$ depend on $\delta$ and $n$. Thus, we have
\begin{equation}\label{eqlemU01}
\int_0^{\infty} e^{-\alpha |q|} \| {\hat U}^{(0)} (\cdot, q) \|_{l^1} d|q|
\le c_1 \| {\hat v}_1 \|_{l^1} \alpha^{-1/n} \Gamma \biggl( \frac{1}{n}
\biggr).
\end{equation}
With $c_1 = 1$ and $q\in\RR^+$, the bound in (\ref{eqlemU01}) holds for $n
=1$ as well. For $n \ge 2$, noting that ${\hat v}_1 (0) =0={\hat f} (0)$,
we have
\begin{equation}\label{eqCGU0}
\int_0^\infty \| {\hat U}^{(0)} (\cdot, q) \|_{l^1} d|q| \le C_G \biggl\|
\frac{{\hat v}_1}{\nu |k|^2} \biggr\|_{l^1} \le C_G \left \{ \| {\hat v}_0
\|_{l^1} \biggl( 1 + \frac{2}{\nu} \| {\hat v}_0 \|_{l^1} \biggr) +
\frac{1}{\nu} \biggl\| \frac{{\hat f}}{|k|^2} \biggr\|_{l^1} \right \},
\end{equation}
where
$$
C_G = \sup_{\phi \in \left [-\frac{n-1}{2 n} \pi + \frac{\delta}{n},
\frac{n-1}{2n} \pi - \frac{\delta}{n} \right ]} n \int_0^\infty s^{-1}|G(s
e^{i \phi} ) | ds.
$$
(ii) If moreover $|k|^{j+2} {\hat v}_0, |k|^{j}{\hat f} \in l^1$ ($j=
0,1$), then
\begin{multline*}
\sup |q|^{1-1/n} (1+|q|^2) e^{-\alpha |q|} \| k^j {\hat U}^{(0)} (\cdot,
q) \|_{l^1} \\
\le 2 c_1 \| |k|^{j} {\hat v}_1 \|_{l^1} \le 2 c_1 \left [ \nu \|
|k|^{j+2} {\hat v}_0 \|_{l^1} + 2 \|{|k|^j \hat v}_0 \|_{l^1} \| |k| {\hat
v}_0 \|_{l^{1}}+ \|  |k|^{j}{\hat f} \|_{l^1} \right ]
\end{multline*}
where the sup is taken over $\mathbb{R}^+$ if $n=1$ and over
$\mathcal{S}_\delta$ if $n>1$.
\end{Lemma}
\begin{proof}
The result follows from (\ref{eqU0}) and (\ref{4}) using the asymptotics
of $G$, cf. (\ref{ar2}) and the behavior $G({\tilde \mu}) \sim C {\tilde
\mu}$ near ${\tilde \mu} =0$. For $n=1$, the bound (\ref{eqlemU01})
follows from the fact that $ \left |2 z^{-1}{J_1(z)} \right | \le 1$.
\end{proof}

The following lemma proves that a suitable solution to the integral
equation (\ref{IntUeqn}) gives rise to a solution of NS.
\begin{Lemma}\label{connection}
For any solution ${\hat U} $ of (\ref{IntUeqn}) such that $ \| {\hat U}
(\cdot, q) \|_{l^1} \in L_{1}(\RR^+,e^{-\alpha q}dq)$, the Laplace
transform
$$
{\hat v} (k, t) = {\hat v}_0 (k) + \int_0^\infty {\hat U} (k, q)
e^{-q/t^n} dq
$$
solves (\ref{nseq}) for $t \in \left ( 0, \alpha^{-1/n} \right )$. For
$n=1$, ${\hat v} (k, t)$  is analytic in $t$ for $\Re \frac{1}{t}>
\alpha$.

It will turn out, cf. Lemma \ref{instsmooth} in the appendix, that $|k|^2
{\hat v} (\cdot, t) \in l^1 $ for $t \in \left ( 0, \alpha^{-1/n} \right
)$. Therefore, $v (x, t) = \mathcal{F}^{-1} \left [ {\hat v} (\cdot, t)
\right ] (x)$ is the classical solution of (\ref{nseq0}).
\end{Lemma}

\begin{proof}
{F}rom (\ref{eqU0}), we obtain
\begin{multline*}
\int_0^\infty e^{-q t^{-n} } {\hat U}^{(0)} (k, q) dq = {\hat v}_1 (k)
\int_0^\infty e^{-q t^{-n}} \frac{1}{2 \pi i} \int_{c-i\infty}^{c+i\infty}
\frac{1-e^{-\nu |k|^2 \tau^{-1/n} } }{\nu |k|^2}e^{q \tau} d\tau  dq \\
= {\hat v}_1 (k) \left ( \frac{1-e^{-\nu |k|^2 t}}{\nu |k|^2} \right ).
\end{multline*}
Furthermore , we may rewrite (\ref{kernelG}) as
\begin{multline}\label{eq124}
\mathcal{G} (q, q'; k) \\
= \frac{1}{2 \pi i} \int_{0}^1 \int_{c-i \infty}^{c+i \infty} \tau^{-1/n}
\left \{ \exp \left [ -\nu |k|^2 \tau^{-1/n} (1-s) + (q-q'/s^n) \tau
\right ] d\tau \right \} ~ds
\end{multline}
since the integral with respect to $\tau$ is identically zero when $ s \in
\left (0, (q'/q)^{1/n} \right )$ (the $\tau$ contour can be pushed to
$+\infty$), we can replace the lower limit in the outer integral in
(\ref{eq124}) by $(q'/q)^{1/n}$. Note that $\| \hat{H}_j (\cdot ,q)
\|_{l^{1}} \in {L}_1 \left ( e^{-\alpha |q|} d|q| \right )$, since
\begin{equation}\label{eq:conv2}
\|F*G\|_{\alpha}\le \|F\|_{\alpha}\|G\|_{\alpha}
\end{equation}
(see\cite{Duke} and also Lemma~\ref{lemBanach} below). Changing variable
$q'/s^n \to q'$ and applying Fubini's theorem we get
\begin{equation}\label{eqHG}
-i k_j \int_0^q {\hat H}_j (k, q') \mathcal{G} (q, q'; k) dq' = \int_0^1
s^n \left \{ \int_0^q \left [-i k_j {\hat H}_j \right ] (k, q' s^n)
\mathcal{Q} (q-q', s; k) dq' \right \} ds
\end{equation}
where for $q > 0$ we have
\begin{equation}
\mathcal{Q} (q,s;k) = \frac{1}{2 \pi i} \int_{c-i \infty}^{c+i \infty}
\exp \left [ -\nu |k|^2 \tau^{-1/n} (1-s) + q \tau) \right ] \tau^{-1/n}
d\tau.
\end{equation}
Laplace transforming (\ref{eqHG}) with respect to $q$, again by Fubini we
have
\begin{multline}
\int_0^\infty e^{-q t^{-n}} \left \{ \int_0^1 \int_0^q \left \{ -ik_j
{\hat H}_j \right \} (k, q's^n) \mathcal{Q} (q-q'; s, k) s^n dq' ds \right
\}~ dq \\
= -i k_j \int_0^1 ~ds ~g(t, s; k) {\hat h}_j (k, st ),
\end{multline}
where ${\hat h}_j (k, t) = \mathcal{L} \left [ {\hat H}_j (k, \cdot)
\right ] (t^{-n})$, ~~$g(t, s; k) = \mathcal{L} \left [ \mathcal{Q}
(\cdot, s; k) \right ] (t^{-n}) $. By assumption, $\| {\hat U} (\cdot, q)
\|_{l^1} \in L_{1} \left (\mathbb{R}^+, e^{-\alpha q} dq \right ) $ and
${\hat v}_0 (k) \in l^1$. {F}rom (\ref{7.2}) and (\ref{eq:conv2}) it
follows that ${\hat H}_j$ is Laplace transformable in $q$ and
$$
{\hat h}_j (k, t) = P_k \left \{ {\hat v}_{0,j} {\hat *} {\hat u} + {\hat
u}_j {\hat *} {\hat v}_0 + {\hat u}_j {\hat *} {\hat u}  \right \} (k, t),
$$
while
$$
g(t, s; k) = t \exp \left [ -\nu |k|^2 t (1-s) \right ].
$$
This leads to
\begin{multline*}
{\hat u} (k, t) = t \int_0^1 e^{-\nu |k|^2 t (1-s)} \left [ -i k_j {\hat
h}_j  \right ] (k, s t) ds + {\hat v}_1 (k) \left ( \frac{1-e^{-\nu |k|^2
t}}{\nu |k|^2} \right ) \\
= \int_0^t e^{-\nu |k|^2 (t -\tau) } \left [ -i k_j {\hat h}_j \right ]
(k, \tau) d\tau + {\hat v}_1 (k) \left ( \frac{1-e^{-\nu |k|^2 t}}{\nu
|k|^2} \right )
\end{multline*}
and thus
$$
{\hat u}_t + \nu |k|^2 {\hat u} = -i k_j {\hat h}_j (k, t) ~+~{\hat
v}_1~~, ~~{\rm with} ~{\hat u} (k, 0) = 0.
$$
Therefore, using expression (\ref{hatueq}) for ${\hat h}_j$, we see that
${\hat v} (k, t) = {\hat u} (k, t) + {\hat v}_0 (k) $ (\ref{nseq}), with
${\hat v} (k, 0) = {\hat v}_0 (k) $. Analyticity in $t$ of this solution
in region $\Re \frac{1}{t} > \alpha$ follows from the representation
(\ref{intro.1.1}). It is clear that $|k|^2 {\hat v} (\cdot, t), {\hat f}
\in l^1 $, ensures that $\mathcal{F}^{-1} \left [ {\hat v} (\cdot, t)
\right ] (x) $ is a classical solution to (\ref{nseq0}).
\end{proof}

\section{Existence of a solution to (\ref{IntUeqn})}
\z First, we prove some preliminary lemmas.

\begin{Lemma}\label{lem0.1}
By standard Fourier theory, if ${\hat v}, {\hat w} \in l^1 \left (
\mathbb{Z}^3 \right )$, then so is ${\hat v}{\hat *} {\hat w}$, and $ \|
{\hat v} {\hat *} {\hat w} \|_{l^1} \le \| {\hat v} \|_{l^1} \| {\hat w}
\|_{l^1}$. \qed
\end{Lemma}

\begin{Lemma}\label{lem0.2}
$$
\| P_k \left [ {\hat w}_j {\hat *} {\hat v} \right ] \|_{l^1} \le 2 \|
{\hat w}_j \|_{l^1} \|{\hat v} \|_{l^1}.
$$
\end{Lemma}

\begin{proof}
It is easily seen from the representation of $P_k$ in (\ref{8.0}) that
\begin{equation}\label{Pbound}
| P_k  {\hat g} (k) | \le 2 |{\hat g} (k)|.
\end{equation}
The rest follows from Lemma (\ref{lem0.1}).
\end{proof}

\begin{Lemma}\label{lemC2}
Let $C_2=C_2(\delta,n)$ be given by
\begin{multline*}
C_2 = 2 \sup_{\substack{q, q' \in e^{i \phi} \mathbb{R}^+ \subset
\mathcal{S}_\delta,\, 0 \le |q'| \le |q| \\ k \in \mathbb{Z}^3}} \nu^{1/2}
|k| |q|^{1/2} |q-q'|^{1/2-1/(2 n)} | \mathcal{G} (q, q'; k) |\quad
\mbox{{\rm for} $n \ge 2$}, \\
C_2 = 2 \sup_{\substack{q, q' \in \mathbb{R}^+,\, 0 \le q' \le q \\ k \in
\mathbb{Z}^3}} \nu^{1/2} |k| q^{1/2} | \mathcal{G} (q, q'; k) |\quad
\mbox{{\rm for} $n = 1$}.
\end{multline*}
Then, for $n \ge 2$, we have
\begin{multline}\label{N1}
\| \mathcal{N} [{\hat U}] (\cdot , q) \|_{l^1} \le \frac{C_2}{\nu^{1/2}
|q|^{1/2}} \int_0^{|q|}  (|q|-s)^{-1/2+1/(2n)} \left \{ \| {\hat U} (\cdot
, s e^{i \phi} ) \|_{l^1} \right. \\
\left. * \| {\hat U} (\cdot , s e^{i \phi}) \|_{l^1} + 2 \| {\hat v}_0
\|_{l^1} \| {\hat U}(\cdot , s e^{i \phi}) \|_{l^{1}} \right \} ds + \|
{\hat U}^{(0)} (\cdot, q) \|_{l^1},
\end{multline}
\begin{multline}\label{N2}
\| \mathcal{N} [{\hat U}^{[1]}] (\cdot , q) - \mathcal{N} [{\hat U}^{[2]}]
(\cdot , q) \|_{l^1} \\
\le \frac{C_2}{\nu^{1/2} |q|^{1/2}} \int_0^{|q|} (|q|-s)^{-1/2+1/(2n)}
\left \{ \left ( \| {\hat U}^{[1]} (\cdot , s e^{i \phi}) \|_{l^1} + \|
{\hat U}^{[2]} (\cdot , s e^{i \phi}) \|_{l^1} \right ) \right. \\
\left. * \| {\hat U}^{[1]} (\cdot , s e^{i \phi}) - {\hat U}^{[2]} (\cdot
, s e^{i \phi} )\|_{l^1} + 2 \| {\hat v}_0 \|_{l^1} \| {\hat U}^{[1]}
(\cdot , s e^{i \phi}) -{\hat U}^{[2]} ( \cdot , s e^{i \phi}) \|_{l^1}
\right \} ds.
\end{multline}
For $n=1$, (\ref{N1}) and (\ref{N2}) hold for $q \in \mathbb{R}^+$, {\it
i.e.} when $\phi=0$.
\end{Lemma}

\begin{proof}
{F}rom Lemma \ref{lem0.2}, we have, for any $q$
$$
\| P_k \left \{ {\hat U}_j \ds  {\hat U} \right \} (k,q ) \|_{l^{1}} \le 2
\| {\hat U} (\cdot, q) \|_{l^1} * \| {\hat U} (\cdot, q)\|_{l^{1}},
$$
and similarly
$$
\| P_k \left \{ {\hat v}_{0,j}{\hat *}{\hat U} (\cdot, q) + {\hat U}_j
(\cdot, q) {\hat *} {\hat v}_{0} \right \} \|_{l^{1}} \le 4 \| {\hat v}_0
\|_{l^1} \| {\hat U} (\cdot, q)\|_{l^1},
$$
and (\ref{N1}) follows.

The second part of the lemma follows by noting that
\begin{equation}\label{Udiff}
{\hat U}_j^{[1]} \ds  {\hat U}^{[1]} -{\hat U}_j^{[2]} \ds {\hat U}^{[2]}
= {\hat U}^{[1]}_j \ds  \left ({\hat U}^{[1]} - {\hat U}^{[2]} \right ) +
\left( {\hat U}^{[1]}_j - {\hat U}^{[2]}_j \right ) \ds {\hat U}^{[2]}.
\end{equation}
Applying Lemma \ref{lem0.2} to (\ref{Udiff}), we obtain
\begin{multline*}
\| P_k \left \{ {\hat U}_j^{[1]} \ds  {\hat U}^{[1]} (\cdot, q) -{\hat
U}_j^{[2]} \ds  {\hat U}^{[2]} (\cdot, q) \right \} \|_{l^{1}} \le  2 \|
{\hat U}^{[1]} (\cdot, q) \|_{l^{1}} * \| {\hat U}^{[1]} (\cdot,
q) - {\hat U}^{[2]} (\cdot, q) \|_{l^1} \\
+ 2 \| {\hat U}^{[2]} (\cdot, q) \|_{l^{1}} * \| {\hat U}^{[1]} (\cdot, q)
- {\hat U}^{[2]} (\cdot, q) \|_{l^{1}},
\end{multline*}
from which (\ref{N2}) follows easily.
\end{proof}

It is convenient to define a number of different $q$-norms, $q \in e^{i
\phi} \mathbb{R}^+ \cup \{0\} \subset \mathcal{S}_\delta$.

\begin{Definition}\label{DefA}
(i) For $\alpha > 0$, $n\ge 2$, we let $\mathcal{A}^{(\alpha)}$ be the set
of analytic functions in $\mathcal{S}_\delta$ with the norm
\begin{equation}\label{8.0.1}
\| {\hat f} \|^{(\alpha)} = \sup_{q \in \mathcal{S}_\delta} |q|^{1-1/n}
(1+|q|^2) e^{-\alpha |q|} \|{\hat f} (\cdot, q) \|_{l^{1}} < \infty,
\end{equation}
while for $n=1$, $\mathcal{A}^{(\alpha)}$ will denote the set of
continuous functions on $[0, \infty)$ with norm $\| \cdot \|^{(\alpha)}$.

(ii) Let $\alpha > 0$, $n \ge 2$, $\delta > 0$. We define a Banach space
$\mathcal{A}_1^{\alpha;\phi} $ of functions along the ray
$|q|e^{i\phi}\in\mathcal{S}_\delta$ with the norm
\begin{equation}\label{8.0.0}
\| {\hat f} \|_1^{\alpha;\phi} = \int_0^\infty e^{-\alpha |q|} \| {\hat f}
(\cdot, |q|e^{i\phi}) \|_{l^{1}} d|q| < \infty.
\end{equation}
We agree to omit the superscript $\phi$ when $\phi=0$ (which is always the
case if $n=1$).
\end{Definition}

\begin{Lemma}\label{lemBanach}
We have the following Banach algebra properties:
\begin{equation}\label{eq:norm1}
\| {\hat f}~ \ds ~ {\hat g} \|_1^{\alpha; \phi} \le \| {\hat f}
\|_1^{\alpha; \phi} \| {\hat g} \|_1^{\alpha; \phi},
\end{equation}
\begin{equation}\label{eq:norm2}
\| {\hat f} ~\ds~ {\hat g} \|^{(\alpha)} \le M_0 \| {\hat f} \|^{(\alpha)}
\| {\hat g} \|^{(\alpha)},
\end{equation}
where
$$
M_0 = 2^{4-1/n} \int_0^\infty \frac{ds}{s^{1-1/n} (1+s^2)}.
$$
\end{Lemma}

\begin{proof}
In the following, we take $u(s) = \| {\hat f}(\cdot, s e^{i \phi})
\|_{l^1}$ and $v (s) = \| {\hat g} (\cdot, s e^{i\phi}) \|_{l^1} $. For
(\ref{eq:norm1}) we note that for any $L > 0$,
\begin{multline}
\int_0^L e^{-\alpha |q|} \int_0^{|q|} u(s) v (|q|-s) ds ~d|q| \\
= \int_0^L \int_0^{|q|} e^{-\alpha s} e^{-\alpha (|q|-s)} u(s) v (|q|-s)
ds ~d|q| \le \int_0^L e^{-\alpha s} u(s) ds  \int_0^L e^{-\alpha \tau}
v(\tau) d\tau.
\end{multline}
{F}rom (\ref{8.0.1}), we note that
$$
\int_0^{|q|} u(s) v(|q|-s) ds \le \| {\hat f} \|^{(\alpha)} \| {\hat g}
\|^{(\alpha)} e^{\alpha |q|} \int_0^{|q|} \frac{ds}{s^{1-1/n}
(|q|-s)^{1-1/n} [1+s^2] [1+(|q|-s)^2]}.
$$
Finally,
\begin{multline*}
\int_0^{|q|} \frac{ds}{s^{1-1/n} (|q|-s)^{1-1/n} [1+s^2] [1+(|q|-s)^2]} \\
= 2 \int_0^{|q|/2} \frac{ds}{s^{1-1/n} (|q|-s)^{1-1/n} [1+s^2]
[1+(|q|-s)^2]} \\
\le \frac{2^{2-1/n}}{|q|^{1-1/n} (1+|q|^2/4)} \int_0^{|q|/2}
\frac{ds}{s^{1-1/n} [1+s^2]} \le \frac{2^{4-1/n}}{|q|^{1-1/n} (1+|q|^2)}
\int_0^\infty \frac{ds}{s^{1-1/n} [1+s^2]},
\end{multline*}
where we used $\sup \frac{1+|q|^2}{1+|q|^2/4} = 4$.
\end{proof}

\begin{Lemma}\label{NUnormbound}
Let $C_2$ be as in Lemma \ref{lemC2} and $\alpha \ge 1$. The operator
$\mathcal{N} $ in (\ref{IntUeqn}) is well defined on:

(i) $\mathcal{A}_1^{\alpha; \phi}$, where it satisfies the following
inequalities
\begin{equation}
\| \mathcal{N} [{\hat U}] \|_1^{\alpha; \phi} \le C_2 \nu^{-1/2} \Gamma
\biggl( \frac{1}{2n} \biggr) \alpha^{-1/(2n)} \left \{ \left ( \| {\hat U}
\|_1^{\alpha; \phi} \right )^2 + 2 \| {\hat v}_0 \|_{l^1} \| {\hat U}
\|_1^{\alpha; \phi} \right \} + \| {\hat U}^{(0)} \|_1^{\alpha; \phi},
\end{equation}
\begin{multline}
\| \mathcal{N} [{\hat U}^{[1]} ] - \mathcal{N} [ {\hat U}^{[2]}]
\|_1^{\alpha; \phi} \le C_2 \nu^{-1/2} \Gamma \biggl( \frac{1}{2n} \biggr)
\alpha^{-1/(2n)} \\
\times \left \{ \left ( \| {\hat U}^{[1]} \|_1^{\alpha; \phi} + \| {\hat
U}^{[2]} \|_1^{\alpha; \phi} \right ) \| {\hat U}^{[1]} - {\hat U}^{[2]}
\|_1^{\alpha; \phi} + 2 \| {\hat v}_0 \|_{l^{1}} \| {\hat U}^{[1]}-{\hat
U}^{[2]} \|_1^{\alpha; \phi} \right \}.
\end{multline}

(ii) $\mathcal{A}^{(\alpha)}$, where it satisfies the inequalities:
\begin{equation}\label{NUnormsup}
\| \mathcal{N} [{\hat U}] \|^{(\alpha)} \le C_2 C_3 \nu^{-1/2}
\alpha^{-1/(2n)} \left \{ M_0 \left ( \| {\hat U} \|^{(\alpha)} \right )^2
+ 2 \| {\hat v}_0 \|_{l^{1}} \| {\hat U} \|^{(\alpha)} \right \} + \|
{\hat U}^{(0)} \|^{(\alpha)},
\end{equation}
\begin{multline}\label{3.50}
\| \mathcal{N} [{\hat U}^{[1]} ] - \mathcal{N} [ {\hat U}^{[2]}]
\|^{(\alpha)} \le C_2 C_3 \nu^{-1/2} \alpha^{-1/(2n)} \\
\times \left \{ M_0 \left ( \| {\hat U}^{[1]} \|^{(\alpha)} + \| {\hat
U}^{[2]} \|^{(\alpha)} \right ) \| {\hat U}^{[1]} - {\hat U}^{[2]}
\|^{(\alpha)} + 2 \| {\hat v}_0 \|_{l^{1}} \| {\hat U}^{[1]}-{\hat
U}^{[2]} \|^{(\alpha)} \right \},
\end{multline}
where $C_3$ is defined in (\ref{C3def}) and depends on $n$ alone.
\end{Lemma}

\begin{proof}
(i) For any $0 < L \le \infty$ and $u \ge 0$ we have
\begin{multline*}
\int_0^L e^{-\alpha |q|} |q|^{-1/2} \left ( \int_0^{|q|}
(|q|-s)^{-1/2+1/(2n)} u(s e^{i\phi} ) ds \right ) d|q| \\
= \int_0^L u(s e^{i \phi} ) e^{-\alpha s} \left ( \int_{s}^L |q|^{-1/2}
(|q|-s)^{-1/2+1/(2n)} e^{-\alpha (|q|-s)} d|q| \right ) ds \\
\le \int_0^L e^{-\alpha s} u(s e^{i\phi})  \left \{ \int_0^L
{s'}^{-1/2+1/(2n)} (s'+s)^{-1/2} e^{-\alpha s'} ds' \right \} ds.
\end{multline*}
Using (\ref{N1}) it follows that
\begin{multline*}
\int_0^\infty e^{-\alpha |q|} \| \mathcal{N} [{\hat U} ] (\cdot, |q| e^{i
\phi}) \|_{l^1} d|q| \\
\le C_2 \nu^{-1/2} \Gamma \biggl( \frac{1}{2n} \biggr) \alpha^{-1/(2n)}
\left ( \left [ \| {\hat U} \|_1^{\alpha; \phi} \right ]^2 + 2 \| v_0
\|_{l^{1}} \| {\hat U} \|_{1}^{\alpha; \phi} \right ) + \| {\hat U}^{(0)}
\|_1^{\alpha; \phi}.
\end{multline*}
{F}rom (\ref{N2}), it now follows that
\begin{multline*}
\int_0^\infty \| \mathcal{N} [{\hat U}^{[1]} ] - \mathcal{N} [ {\hat
U}^{[2]}] \|_{l^1} e^{-\alpha |q|} d|q| \\
\le C_2 \nu^{-1/2} \Gamma \biggl( \frac{1}{2n} \biggr) \alpha^{-1/(2n)}
\left \{ \left ( \| {\hat U}^{[1]} \|_1^{\alpha; \phi} + \| {\hat U}^{[2]}
\|_1^{\alpha; \phi} \right ) \| {\hat U}^{[1]} - {\hat U}^{[2]}
\|_1^{\alpha; \phi} \right . \\
\left. + 2 \| {\hat v}_0 \|_{l^{1}} \| {\hat U}^{[1]}-{\hat U}^{[2]}
\|_1^{\alpha; \phi} \right \}.
\end{multline*}

(ii) We first note that
\begin{multline*}
|q|^{1/2-1/n} \int_0^{|q|} e^{-\alpha (|q|-s)} (|q|-s)^{-1/2+1/(2n)}
s^{-1+1/n} (1+s^2)^{-1} ds \\
= |q|^{1/(2n)} \int_0^1 e^{-\alpha |q| (1-t)} t^{-1+1/n}
(1-t)^{-1/2+1/(2n)} (1+t^2 |q|^2 )^{-1} dt \\
= |q|^{1/(2n)} \left \{ \int_0^{1/2} e^{-\alpha |q| (1-t)}
\frac{t^{-1+1/n} (1-t)^{-1/2+1/(2n)}}{ (1+t^2 |q|^2)} dt + \right. \\
\left . \int_{1/2}^1 e^{-\alpha |q| (1-t)} \frac{t^{-1+1/n}
(1-t)^{-1/2+1/(2n)}}{ (1+t^2 |q|^2) } dt  \right \}
\end{multline*}
\begin{multline}\label{r1}
\le |q|^{1/(2n)} e^{-\alpha |q|/2} \int_0^{1/2} t^{-1+1/n}
(1-t)^{-1/2+1/(2n)} dt \\
+ \frac{2^{1-1/n} |q|^{1/(2n)}}{1+|q|^2/4} \int_{1/2}^1 e^{-\alpha |q|
(1-t)} (1-t)^{-1/2+1/(2n)} dt.
\end{multline}
The first term on the right of (\ref{r1}) is bounded by $n 2^{1/2-3/(2n)}
|q|^{1/(2n)} e^{-\alpha |q|/2}$. For the second term we separate two
cases. Let first $\alpha |q| \le 1$. It is then clear that
\begin{multline*}
|q|^{1/(2n)} \int_{1/2}^1 e^{-\alpha |q| (1-t)} (1-t)^{-1/2+1/(2n)} dt \\
\le |q|^{1/(2n)} \int_{1/2}^1 (1-t)^{-1/2+1/(2n)} dt \le
\frac{2n}{(n+1)\alpha^{1/(2n)}}.
\end{multline*}
Now, if $\alpha |q| > 1$, we have
\begin{multline*}
|q|^{1/(2n)} \int_{1/2}^1 e^{-\alpha |q| (1-t)} (1-t)^{-1/2+1/(2n)} dt =
|q|^{1/(2n)} \int_{0}^{1/2} e^{-\alpha |q| t} t^{-1/2+1/(2n)} dt \\
\le |q|^{1/(2n)} \Gamma \biggl( \frac{1}{2}+\frac{1}{2n} \biggr) \left
[\alpha |q| \right ]^{-1/2-1/(2n)} \le \alpha^{-1/(2n)} \Gamma \biggl(
\frac{1}{2}+\frac{1}{2n} \biggr).
\end{multline*}
Combining these results we get
$$
|q|^{1/(2n)} \int_{1/2}^1 e^{-\alpha |q| (1-t)} (1-t)^{-1/2+1/(2n)} dt \le
\alpha^{-1/(2n)} C_1,
$$
where
$$
C_1 = \max \left \{ \Gamma \biggl( \frac{1}{2}+\frac{1}{2n} \biggr),
\frac{2n}{n+1} \right \}.
$$
Therefore,
\begin{multline}\label{C3def}
\sup_{|q| > 0} \left \{ |q|^{1-1/n} (1+|q|^2) e^{-\alpha |q|} |q|^{-1/2}
\int_0^{|q|} e^{\alpha s} (|q|-s)^{-1/2+1/(2n)} s^{-1+1/n} (1+s^2)^{-1} ds
\right \} \\
\le (C_0 + 2^{3-1/n} C_1) \alpha^{-1/(2n)} \equiv C_3 \alpha^{-1/(2n)},
\end{multline}
where
$$
C_0 = n 2^{1/2-1/n} \left [ \sup_{\gamma > 0} \gamma^{1/(2n)} e^{-\gamma}
+ 4 \sup_{\gamma > 0 } \gamma^{2+1/(2n)} e^{-\gamma} \right ].
$$
From (\ref{N1}) and the definition of $\| \cdot \|^{(\alpha)}$, it follows
that
$$
\| \mathcal{N} [{\hat U} ] \|^{(\alpha)} \le C_2 C_3 \nu^{-1/2}
\alpha^{-1/(2n)} \left [ M_0 \left ( \| {\hat U} \|^{(\alpha)} \right )^2
+ 2 \| {\hat v}_0 \|_{l^{1}} \| {\hat U} \|^{(\alpha)} \right ] + \| {\hat
U}^{(0)} \|^{(\alpha)}.
$$
Inequality (\ref{3.50}) follows similarly.
\end{proof}

\begin{Lemma}\label{inteqn}
The integral equation (\ref{IntUeqn}) has a unique solution in:

(i) the ball of radius $2 \| {\hat U}^{(0)} \|_1^{\alpha; \phi}$ in
$\mathcal{A}_1^{\alpha; \phi}$, if $\alpha$  is large enough so that
\begin{equation}\label{ensure2}
C_2 \nu^{-1/2} \Gamma \biggl( \frac{1}{2n} \biggr) \alpha^{-1/(2n)} \left
( 4 \| {\hat v}_0 \|_{l^{1}} + 4 \| {\hat U}^{(0)} \|_1^{\alpha; \phi}
\right ) < 1.
\end{equation}
Here $C_2$ is the same as in Lemma \ref{lemC2} and depends on $\delta$ and
$n$ for $n \ge 2$. For $n=1$ we have $\phi=0$.

(ii) the ball of radius $2 \| {\hat U}^{(0)} \|^{(\alpha)}$ in
$\mathcal{A}^{(\alpha)}$ if $\alpha$ is large enough so that
\begin{equation}\label{ensure3}
C_2 C_3 \nu^{-1/2} {\alpha}^{-1/(2n)} \left ( 4 \| {\hat v}_0 \|_{l^{1}} +
4 M_0 \| {\hat U}^{(0)} \|^{(\alpha)} \right ) < 1,
\end{equation}
where $C_2$ (defined in Lemma \ref{lemC2}) and  $C_3$ (defined in
(\ref{C3def})) depend on $\delta$ and $n$ for $n \ge 2$.
\end{Lemma}

\begin{proof}
The estimates in Lemma \ref{NUnormbound} imply that $\mathcal{N}$ maps a
ball of size $ 2 \| {\hat U}^{(0)} \|_1^{\alpha; \phi} $ in
$\mathcal{A}_1^{\alpha; \phi}$ back to itself and that $\mathcal{N}$ is
contractive in that ball when $\alpha$ satisfies (\ref{ensure2}). In
$\mathcal{A}^{(\alpha)}$, the estimates of Lemma \ref{NUnormbound} imply
that $\mathcal{N}$ maps a ball of size $ 2 \| {\hat U}^{(0)} \|^{(\alpha)}
$ to itself and that $\mathcal{N}$ is contractive in that ball when
$\alpha$ satisfies (\ref{ensure3}).
\end{proof}

\begin{Remark}
If $\alpha$ satisfies both (\ref{ensure2}) and (\ref{ensure3}), then it
follows from Lemma \ref{connection} and the uniqueness of classical
solution of \ref{nseq} that the solutions ${\hat U}$ in
$\mathcal{A}^{\alpha; \phi}_1 $ and $\mathcal{A}^{(\alpha)} $ are one and
the same.
\end{Remark}

\begin{Lemma}\label{lemqder}
The $q$-derivatives of ${\hat U} (k, q)$ in $\mathcal{A}^{(\alpha)}$ for
$q > 0$ are estimated by:
\begin{multline}\label{Uderbound}
\left \| \frac{\partial^m}{\partial q^m} {\hat U} (\cdot, q) \right
\|_{l^{1}} \le C_m \| {\hat v}_1 \|_{l^1} \frac{q^{-1+1/n}
\omega^{-m}}{1+q^2} e^{\alpha q +\omega \alpha}, \\
{\rm where}~\omega = q/2~{\rm for}~q \le 2,\ \omega = 1~{\rm for}~q > 2.
\end{multline}
\end{Lemma}

\begin{proof}
For $q \le 2$, we use Cauchy's integral formula on a circle of radius
$q/2$ around $q$ and Lemma \ref{lemU0} to bound ${\hat U} $ for $|q| > 0$,
$\arg q \in \left [ -(n-1) \frac{\pi}{2} + \delta, (n-1) \frac{\pi}{2} -
\delta \right ]$ (we may pick for instance $\delta = \frac{\pi}{4}$ to
obtain specific values of constants here). For $q > 2$, the  argument is
similar, now on a circle of radius 1.
\end{proof}

In the following we need bounds on $\| k {\hat U} \|^{(\alpha)}$. We
rewrite (\ref{IntUeqn}) using the divergence-free condition (note that $k
{\hat U}$ is a tensor of rank 2) as
\begin{multline}\label{IntUeqn1}
k {\hat U} (k, q) = -i k \int_0^q \mathcal{G} (q, q'; k) P_k \left \{
{\hat U}_j \ds [ k_j {\hat U} ] + {\hat v}_{0,j} {\hat *} [ k_j {\hat U} ]
\right \} (k, q') dq' + {\hat U}^{(0,1)} (k, p) \\
:= \mathcal{\tilde N} \left [ k {\hat U} \right ] \\
{\rm where }~~{\hat U}^{(0,1)} (k, p) := -i k \int_0^q \mathcal{G} (q, q';
k) P_k \left [ {\hat U}_j {\hat *} [k_j {\hat v}_0] \right ] (k, q') dq' +
k {\hat U}^{(0)} (k, j).
\end{multline}
We now think of  ${\hat U}$ in (\ref{IntUeqn1}) as known; then
$\mathcal{\tilde N}$ becomes linear in $k {\hat U}$.

\begin{Lemma}\label{lemkU}
If $|k|^3 {\hat v}_0 \in l^1 $ and $\alpha$ is large enough so that
(\ref{ensure3}) is satisfied, then
$$
\| |k| {\hat U} \|^{(\alpha)} \le 4 c_1 \left ( \nu \| |k|^3 {\hat v}_0
\|_{l^{1}}  + 2 \| |k| {\hat v}_0  \|_{l^1}^2 + \| |k| {\hat f} \|_{l^1}
\right ) + \| |k| {\hat v}_0 \|_{l^1}.
$$
\end{Lemma}

\begin{proof}
{F}rom (\ref{IntUeqn1}), we obtain
\begin{multline*}
\| |k| {\hat U} \|^{(\alpha)} = \| \mathcal{\tilde N} [ |k| {\hat U} ]
\|^{(\alpha)} \\
\le C_2 C_3 \nu^{-1/2} \alpha^{-1/(2n)} \left \{ M_0 \| {\hat U}
\|^{(\alpha)} \| |k| {\hat U} \|^{(\alpha)} + \| {\hat v}_0 \|_{l^{1}} \|
|k| {\hat U} \|^{(\alpha)} \right \} + \| {\hat U}^{(0,1)} \|^{(\alpha)}.
\end{multline*}
Lemma \ref{inteqn}, which applies when $\alpha$ satisfies (\ref{ensure3}),
implies that $\| {\hat U} \|^{(\alpha)} \le 2 \| {\hat U}^{(0)}
\|^{(\alpha)}$ and thus
\begin{multline*}
\| |k| {\hat U} \|^{(\alpha)} \le C_2 C_3 \nu^{-1/2} \alpha^{-1/(2n)} \|
|k| {\hat U} \|^{(\alpha)} \left \{ 2 M_0 \| {\hat U}^{(0)} \|^{(\alpha)}
+ \| {\hat v}_0 \|_{l^{1}} \right \} + \| {\hat U}^{(0,1)} \|^{(\alpha)}
\\
\le \frac{1}{2} \| |k| {\hat U} \|^{(\alpha)} + \| {\hat U}^{(0,1)}
\|^{(\alpha)}.
\end{multline*}
Thus,
\begin{multline*}
\| |k| {\hat U} \|^{(\alpha)} \\
\le 2 \| {\hat U}^{(0,1)} \|^{(\alpha)} \le 2 \| |k| {\hat U}^{(0)}
\|^{(\alpha) } + 4 M_0 C_2 C_3 \nu^{-1/2} \alpha^{-1/(2n)} \| |k| {\hat
v}_0 \|_{l^1} \| {\hat U}^{(0)} \|^{(\alpha)}.
\end{multline*}
Lemma follows from (\ref{ensure3}) and bounds on ${\hat U}^{(0)} $ given
in Lemma \ref{lemU0}.
\end{proof}

\begin{Proposition}\label{propthm01}
Assume ${\hat f} (0) =0 = {\hat v}_0 (0)$ and we define $\bigl\|
\frac{\hat f}{|k|^2} \bigr\|_{l_1} = \sum_{k \in \mathbb{Z}^3 \setminus
\{0\}} \frac{ |{\hat f} |^2}{|k|^2} $. If for $n \ge 2$, $\alpha$
satisfies the condition:
\begin{equation}\label{ensure4}
C_2 \nu^{-1/2} \Gamma \biggl( \frac{1}{2n} \biggr) \alpha^{-1/(2n)}
\biggl\{ 4 \| {\hat v}_0 \|_{l^{1}} + 4 C_G \biggl[ \| {\hat v}_0 \|_{l^1}
\Bigl( 1 + \frac{2}{\nu} \|{\hat v}_0 \|_{l^1} \Bigr) + \frac{1}{\nu}
\biggl\| \frac{{\hat f}}{|k|^2} \biggr\|_{l^1} \biggr] \biggr\} < 1,
\end{equation}
with constants $C_2$ and $C_G$ defined in Lemmas \ref{lemC2} and
\ref{lemU0}, then the integral equation (\ref{IntUeqn}) has a unique
solution in a ball of size $2 \| {\hat U}^{(0)} \|_1^{\alpha; \phi}$ in
$\mathcal{A}_1^{\alpha; \phi}$. If in addition $ |k|^2 {\hat v}_0 \in l^1
$, then for $n \ge 1$ and $\alpha=\alpha_1$ is such that
\begin{equation}\label{ensure5}
C_2 \nu^{-1/2} \Gamma \biggl( \frac{1}{2n} \biggr) \alpha_1^{-1/(2n)}
\biggl\{ 4 \| {\hat v}_0 \|_{l^{1}} + 4 c_1 \Gamma \biggl( \frac{1}{n}
\biggr) \alpha_1^{-1/n} \| {\hat v}_1 \|_{l^1} \biggr\} < 1,
\end{equation}
where
\begin{equation}\label{4}
{\hat v}_1 (k) = \left ( -\nu |k|^2 {\hat v}_0 - i k_j P_k \left [ {\hat
v}_{0,j}{\hat*}{\hat v}_0 \right ] \right ) + {\hat f} (k)
\end{equation}
with $c_1$ defined in Lemma \ref{lemU0}, then the integral equation
(\ref{IntUeqn}) has a unique solution in a ball of size $2 \| {\hat
U}^{(0)} \|_1^{\alpha_1; \phi}$ in $\mathcal{A}_1^{\alpha_1; \phi}$.
\end{Proposition}

\begin{proof}
The proof follows from Lemma \ref{lemU0} since (\ref{ensure4}) and
(\ref{ensure5}) imply (\ref{ensure2}), and thus Lemma \ref{inteqn}
applies.
\end{proof}

\noindent{\bf Proof of Theorem \ref{Thm01}}

Proposition \ref{propthm01} gives a unique solution to (\ref{IntUeqn}) in
some small ball in the Banach space $\mathcal{A}_1^{\alpha; \phi}$ for
sufficiently large $\alpha$. {F}rom Lemma \ref{connection}, we see that
${\hat U}$ generates via (\ref{intro.1.1}) a solution ${\hat v}$ to
(\ref{nseq}) for $t \in \left [ 0, \alpha^{-1/n} \right )$. Classical
arguments (presented for completeness in Lemma \ref{instsmooth} in the
Appendix), show that $|k|^2 {\hat v} (\cdot, t) \in l^1$ and hence
$\mathcal{F}^{-1} \left [ {\hat v} (\cdot, t) \right ] (x)$ is a smooth
solution to (\ref{nseq0}) for $t \in \left (0, \alpha^{-1/n} \right )$.
Analyticity in  $t$ for $\Re \frac{1}{t^n} > \alpha$ follows from the
Laplace representation. For optimal analyticity region in $t$, we choose
$n=1$.

It is well known  that (\ref{nseq0}) has locally a unique classical
solution \cite{Temam}, \cite{Doering}, \cite{ConstFoias}. Thus, given
${\hat v}_0, {\hat f} \in l^1$, all solutions obtained via the integral
equation coincide. Furthermore, ${\hat v} (k, t) - {\hat v}_0$ is
inverse-Laplace transformable in $1/t^n$ and the inverse Laplace transform
satisfies (\ref{IntUeqn}). Therefore, no restriction on the size of ball
in spaces $\mathcal{A}^{\alpha; \phi}_1$, $\mathcal{A}^{(\alpha)}$ is
necessary for uniqueness of the solution of (\ref{IntUeqn}).

\begin{Remark}{
\rm The arguments in the proof of Theorem~\ref{Thm01} show that $\| {\hat
v} (\cdot, t) \|_{l^1} < \infty$ over an interval of time implies that the
solution is classical.  This is not a new result.  Standard Fourier
arguments show that, in this case, we have $ \| v (\cdot, t)
\|_{{L}^\infty} < \infty$, {\em i.e.} one of the Prodi-Serrin criteria for
existence of classical solutions \cite{Prodi}, \cite{Serrin} is
satisfied.}
\end{Remark}

\section{Error bounds in a Galerkin approximation involving $[-N, N]^3$ Fourier modes}

\begin{Definition}
We define the operator $\mathcal{N}^{(N)} $ (associated to $\mathcal{N}$)
by
\begin{multline}
\mathcal{N}^{(N)} \left [ {\hat U} \right ] (k, q) \\
= -ik_j \int_0^q  \mathcal{G} (q, q'; k) \mathcal{P}_N P_k \left [ {\hat
U}_j \ds {\hat U} + {\hat v}_{0,j} {\hat *} {\hat U} + {\hat U}_j {\hat *}
{\hat v}_0 \right ] (k, q') dq' + \mathcal{P}_N {\hat U}^{(0)} (k, q),
\end{multline}
where $\mathcal{P}_N$, the Galerkin projection to $[-N, N]^3$ Fourier modes, is given by
$$
\left [ \mathcal{P}_N {\hat U} \right ] (k, q) = {\hat U} (k, q) ~~{\rm
for}~~k \in [-N, N]^3~~,~~~\left [\mathcal{P}_N {\hat U}\right] (k,
q)=~~0~~ {\rm otherwise}.
$$
\end{Definition}

\begin{Lemma}\label{lemNUN}
The integral equation
$$
{\hat U}^{(N)} = \mathcal{N}^{(N)} \left [ {\hat U}^{(N)} \right ]
$$
has a unique solution in $ \mathcal{A}^{\alpha}_1$\footnote{Recall this
means $\mathcal{A}^{\alpha;\phi}_1$ with $\phi=0$} as well as in
$\mathcal{A}^{(\alpha)}$, if $\alpha$ satisfies the conditions in Theorem
\ref{Thm01}.
\end{Lemma}

\begin{proof}
The proof is very similar to that of Theorem \ref{Thm01} part 1, noting
that the  Galerkin projection $\mathcal{P}_N$ does not increase $l^1$
norms and $\mathcal{N}^{(N)}$ and  $\mathcal{N}$ have similar properties.
\end{proof}

\begin{Lemma}\label{lemTEN}
Assume that $\alpha$ is large enough so that
\begin{equation}\label{ensure3.2}
C_2 C_3 \nu^{-1/2} {\alpha}^{-1/(2n)} \left ( 4 \| {\hat v}_0 \|_{l^{1}} +
4 M_0 \| {\hat U}^{(0)} \|^{(\alpha)} \right ) \leq \frac{1}{2},
\end{equation}
and that $|k|^3 {\hat v}_0,\ |k| {\hat f} \in l^1$. Define the Galerkin
truncation error:
\begin{multline}\label{eqTEN}
T_{E, N} = \mathcal{P}_N {\hat U} - \mathcal{N}^{(N)} \left [
\mathcal{P}_N {\hat U} \right ] = \mathcal{P}_N \mathcal{N} \left [ {\hat
U} \right ] - \mathcal{N}^{(N)} \left [ \mathcal{P}_N {\hat U} \right ] \\
= -i k_j \int_0^q \mathcal{G} (q, q'; k) \mathcal{P}_N P_k \Bigl[ {\hat
v}_{0,j} {\hat *} (I-\mathcal{P}_N) {\hat U} + (I-\mathcal{P}_N) {\hat
U}_j {\hat *} {\hat v}_0 \\
+ (I-\mathcal{P}_N) {\hat U}_j \ds \mathcal{P}_N {\hat U} + \mathcal{P}_N
{\hat U}_j \ds (I-\mathcal{P}_N) {\hat U} + (I-\mathcal{P}_N ) {\hat U}_j
\ds (I-\mathcal{P}_N) {\hat U} \Bigr] (k, q') dq'.
\end{multline}
Then,
\begin{equation*}
\| {\hat U} - {\hat U}^{(N)} \|^{(\alpha)} \le \| (I-\mathcal{P}_N ) {\hat
U} \|^{(\alpha)} + 2 \| T_{E,N} \|^{(\alpha)},
\end{equation*}
where
\begin{multline*}
\| (I-\mathcal{P}_N) {\hat U} \|^{(\alpha)} + 2 \| T_{E,N} \|^{(\alpha)}
\\
\le \frac{1}{N} \left [ 2 c_1 \left ( \nu \| |k|^3 {\hat v}_0 \|_{l^1} + 2
\| |k| {\hat v}_0 \|^2_{l^1} + \| |k| {\hat f} \|_{l^1} \right ) + \| |k|
{\hat v}_0 \|_{l^1} \right ] \\
\times \left \{ 1 + 4 c \| {\hat v}_0 \|_{l^1} + 12 c \left ( \nu \| |k|^2
{\hat v}_0 \|_{l^1} + 2 \| |k| {\hat v}_0 \|_{l^1} \| {\hat v}_0 \|_{l^1}
+ \| {\hat f} \|_{l^1} \right ) \right \}.
\end{multline*}
\end{Lemma}

\begin{proof}
Clearly,
$$
\| {\hat U} - {\hat U}^{(N)} \|^{(\alpha)} \le \| (I-\mathcal{P}_N) {\hat
U} \|^{(\alpha)} + \| \mathcal{P}_N {\hat U} - {\hat U}^{(N)}
\|^{(\alpha)}.
$$
By (\ref{ensure3.2}), (\ref{eqTEN}) and contractivity of
$\mathcal{N}^{(N)}$,
\begin{multline*}
\| \mathcal{P}_N {\hat U} - {\hat U}^{(N)} \|^{(\alpha)} \le \|
\mathcal{N}^{(N)} [ \mathcal{P}_N {\hat U} ] - \mathcal{N}^{(N)} [ {\hat
U}^{(N)} ] \|^{(\alpha)} + \| T_{E,N} \|^{(\alpha)} \\
\le \frac{1}{2} \| \mathcal{P}_N {\hat U} - {\hat U}^{(N)} \|^{(\alpha)} +
\| T_{E,N} \|^{(\alpha)},
\end{multline*}
so
$$
\| {\hat U} - {\hat U}^{(N)} \|^{(\alpha)} \le \| (I-\mathcal{P}_N) {\hat
U} \|^{(\alpha)} + 2 \| T_{E,N} \|^{(\alpha)}.
$$
Now estimates similar to (\ref{NUnormsup}) imply that
\begin{multline*}
\| T_{E,N} \|^{(\alpha)} \le c \| (I-\mathcal{P}_N) {\hat U} \|^{(\alpha)}
\left [ 2 \| {\hat v}_0 \|_{l^1} + 2 \| \mathcal{P}_N {\hat U}
\|^{(\alpha)} + \| (I - \mathcal{P}_N ) {\hat U} \|^{(\alpha)} \right ] \\
\le c \| (I - \mathcal{P}_N ) {\hat U} \|^{(\alpha)} \left [ 2 \| {\hat
v}_0 \|_{l^1} + 6 \| {\hat U}^{(0)} \|^{(\alpha)} \right ],
\end{multline*}
and Lemma \ref{lemkU} implies that
\begin{multline*}
\| (I - \mathcal{P}_N ) {\hat U} \|^{(\alpha)} \le \frac{1}{N} \| k {\hat
U} \|^{(\alpha)} \\
\le \frac{1}{N} \left [ 2 c_1 \left ( \nu \| |k|^3 {\hat v}_0 \|_{l^{1}} +
2 \| |k| {\hat v}_0  \|_{l^1}^2 + \| |k| {\hat f} \|_{l^1} \right ) + \|
|k| {\hat v}_0 \|_{l^1} \right ].
\end{multline*}
Hence the lemma follows.
\end{proof}

\section{The exponential rate  $\alpha$ and the singularities of $v$}

We have already established that at most subexponential growth of  $\|
{\hat U} (\cdot, q) \|_{l^1} $ implies global existence of a classical
solution to (\ref{nseq0}).

We now look for a converse: suppose (\ref{nseq0}) has a global solution,
is it true that ${\hat U} (\cdot, q)$  always is subexponential in $q$?
The answer is no. For $n=1$, any complex singularity $t_s$ in the
right-half complex $t$-plane of $v(x, t)$ produces exponential growth of
$\hat{U}$ with rate $\Re(1/t_s)$ (oscillatory with a frequency
$\Im(1/t_s)$).

However, if $f = 0$, we will see that for any {\em given global classical}
solution of (\ref{nseq0}), there is a $c>0$ so that for any $t_s$ we have
$|\arg t_s|>c$. This means that for sufficiently large $n$, the function
$v(x,\tau^{-1/n})$ has no singularity in the right-half $\tau$ plane. Then
the inverse Laplace transform
$$
U(x, q) = \frac{1}{2 \pi i} \int_{c-i\infty}^{c+i \infty} \left \{ v (x,
\tau^{-1/n} ) - v_0 (x) \right \} e^{q \tau} d\tau
$$
can be shown to decay for $q$ near $\RR^+$.

We now seek to find conditions for which there are no singularities of $v
(x, \tau^{-1/n})$ in $\left \{ \tau: \Re\, \tau \ge 0, ~\tau \not\in
\mathbb{R}^+ \cup \{0\} \right \}$.

\begin{Lemma}{(Special case of \cite{FoiasTem})}\label{LL1}
If $f=0$ and $v (\cdot, t_0) \in {H}^1 \left (\mathbb{T}^3\right [0, 2\pi]
) $, then $v(x,t)$ is analytic in $x$ and $t$ in the domain $ |\Im ~x_j |
< c \nu |t-t_0| $, $ 0 < |t-t_0| < C  $ for $\arg (t-t_0) \in \left [
-\frac{\pi}{4}, \frac{\pi}{4} \right ]$, where $c$ and $C$ are positive
constants ($C$ depends on $\| v_0 \|_{{H}^1}$ and $\nu$ and bounded away
from 0 when $\| v_0 \|_{{H}^1}$ is bounded).
\end{Lemma}
See page 71 of \cite{Foias}.

\begin{Lemma}\label{LL2}
(i) Assume $k {\hat v}_0 , {\hat f} \in l^1$ and $\alpha$ is large enough
so that (\ref{ensure2}) holds for $n=1$. The classical solution of
(\ref{nseq0}) has no singularity in $\Re \frac{1}{t} > \alpha$, $x \in
\mathbb{T}^3$.

(ii) Furthermore, for $f=0$ (no forcing), no singularity can exist for
$\arg (t - T_{c,a}) \in (-{\tilde \delta}, {\tilde \delta})$ for any $0 <
{\tilde \delta} < \frac{\pi}{2}$ and any $x \in \mathbb{T}^3$. ($T_{c,a}$
is estimated in terms of $\| v_0 \|_{H^1},\ \nu$, and ${\tilde \delta}$ in
Theorem \ref{Testimate} in the Appendix using standard arguments.)
\end{Lemma}

\begin{proof}
(i) The assumption implies $v_0 \in {H}^1 (\mathbb{T}^3)$. Since it is
well known (see for instance \cite{Foias}, \cite{Temam},
\cite{ConstFoias}, \cite{Doering}) that a classical solution to
(\ref{nseq0}) is unique, it follows that this solution equals the one
given in Theorem~\ref{Thm01} in the form (\ref{intro.1}). {F}rom standard
properties of Laplace transforms this solution is analytic for $\Re
\frac{1}{t} > \alpha$, where $\alpha$ is given in Theorem \ref{Thm01}.

(ii) We know that under these assumptions $\| v (\cdot, t)\|_{H^1}\to 0$
as $t\to \infty$. There is then a critical time $T_{c,a}$ so that standard
contraction mapping arguments show that $v (\cdot, t)$ is analytic for $t
- T_{c,a} \in {\tilde S}_{\tilde \delta}$ as seen in Theorem
\ref{Testimate} in the Appendix.
\end{proof}

\begin{Corollary}\label{Cr1}
If $f = 0$, for any $v_0$ there exists a $c>0$ so that any singularity
$t_s$ of the solution $v$ of (\ref{nseq0}) is either a positive real time
singularity, or else $|\arg t_s|>c$.
\end{Corollary}

\begin{proof}
If there exists a classical solution on $\RR^+$ then $\| v(\cdot, t)
\|_{H^1}$ is uniformly bounded and by the proof of Lemma~\ref{LL2} (ii)
there is a $T_{c,a}$ (as given in Theorem \ref{Testimate} in the Appendix)
such that $v(\cdot,t)$ is analytic for $\arg (t-T_{c,a}) \in \left [ -
\frac{\pi}{4}, \frac{\pi}{4} \right ]$. Let now $M_1 = \max_{t \in
[0,T_{c,a}+\epsilon]} \| v(\cdot, t) \|_{H^1}$. Then by Lemma \ref{LL1},
for any $t'\in [0,T_{c,a}+\epsilon]$ there exists a $c_2=c_2 (M_1)$ such
that $v$ is analytic in the region $|t-t'|<c_2,\ |\arg(t-t')|\le\pi/4$.
Thus $v$ is analytic in (see Fig.\ref{fig.ns.analytic})
$$
\Bigl\{ t: |t-t'|<c_2, |\arg(t-t')|\le\frac{\pi}{4}, 0 \le t' \le
T_{c,a}+\epsilon \Bigr\} \bigcup\, \Bigl\{ t: |\arg (t -
T_{c,a})|\le\frac{\pi}{4} \Bigr\}.
$$
\begin{figure}[h]
  \centering
  \psfrag{Tca}{\tiny $T_{c,a}$}
  \psfrag{c2}{\tiny $c_2$}
  \includegraphics[scale=0.5]{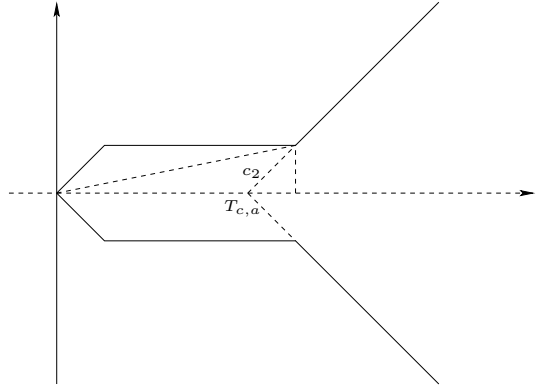}
  \caption{The region of analyticity of $v$.}
  \label{fig.ns.analytic}
\end{figure}%
Thus, if $t_s$ is a singular point of $v$, then $\tan |\arg(t_s)|>c$ where
$$
c = \frac{c_2/\sqrt{2}}{T_{c,a}+c_2/\sqrt{2}} =
\frac{c_2}{\sqrt{2}\,T_{c,a}+c_2}.
$$
\end{proof}

\noindent{\bf Proof of Theorem \ref{Thm02}}

By definition,
$$
\hat{U}(k,q) = \frac{1}{2\pi i} \int_{C} \hat{u}(k,\tau^{-1/n})
e^{q\tau}\,d\tau = \frac{1}{2\pi i} \int_{C} \Bigl[ \hat{v}(k,\tau^{-1/n})
- \hat{v}_0(k) \Bigr] e^{q\tau}\,d\tau,
$$
where the Bromwich contour $C$ lies to the right of all singularities of
$\hat{u}(k,\tau^{-1/n})$ in the complex $\tau$-plane. By
Corollary~\ref{Cr1}, $\hat{u}(k,t)$ has no singularities in the sector
$$
S_{t,\phi} := \Bigl\{ t: |\arg t\,| \le \phi := \tan^{-1} c \Bigr\},
$$
so $\hat{u}(k,\tau^{-1/n})$ has no singularities in the sector
$$
S_{\tau,\phi} := \Bigl\{ \tau: |\arg \tau| < n\phi \Bigr\}.
$$
Clearly, if $n\phi \in (\frac{\pi}{2},\pi)$, then $\hat{u}(k,\tau^{-1/n})$
is analytic in a sector of width between $\pi$ and $2\pi$, and in
particular the Bromwich contour $C$ can be chosen to be the imaginary
axis. With the suitable decay of $\hat{u}(k,\tau^{-1/n})$ at $\tau =
\infty$:
\begin{multline*}
\hat{u}(k,t) = \hat{v}(k,t) - \hat{v}_0(k) = O(t)\qquad \mbox{as } t \to
0,\qquad \mbox{which means that} \\
\hat{u}(k,\tau^{-1/n}) = O(\tau^{-1/n})\qquad \mbox{as } \tau \to \infty,
\end{multline*}
Jordan's lemma applies and $C$ can be deformed to the edges of
the sector $S_{\tau,\phi}$, i.e.
\begin{multline*}
\hat{U}(k,q) = \frac{1}{2\pi i} \biggl\{ \int_{\infty e^{-i n\phi}}^{0} +
\int_{0}^{\infty e^{i n\phi}} \biggr\} \Bigl[ \hat{v}(k,\tau^{-1/n}) -
\hat{v}_0(k) \Bigr] e^{q\tau}\,d\tau \\
= \frac{1}{2\pi i} \biggl\{ \int_{\infty e^{-i n\phi}}^{0} +
\int_{0}^{\infty e^{i n\phi}} \biggr\} \hat{v}(k,\tau^{-1/n})
e^{q\tau}\,d\tau
\end{multline*}
(carefully note that the integral of $\hat{v}_0(k)$ over the contour is
0). Further, as shown in Theorem \ref{Testimate} in the Appendix, there is
a sector ${\tilde S}_{\tilde \delta}$ in the right-half $t$-plane (with
$\phi < {\tilde \delta} < \frac{\pi}{2}$) so that
$$
\|\hat{v}(\cdot,t)\|_{l^1} \le C e^{-\frac{3}{4} \nu \Re\,t}\qquad
\mbox{as $t \to \infty$ in ${\tilde S}_{\tilde \delta}$}.
$$
So
$$
\|\hat{v}(\cdot,\tau^{-1/n})\|_{l^1} \le C e^{-\frac{3}{4} \nu
\Re(\tau^{-1/n})}\qquad \mbox{as $\tau \to 0$ along $e^{\pm i
n\phi}(0,\infty)$},
$$
and the boundedness of $\|\hat{v}(\cdot,\tau^{-1/n})\|_{l^1}$ for large
$|\tau|$ implies that
$$
\|\hat{v}(\cdot,\tau^{-1/n})\|_{l^1} \le C e^{-\frac{3}{4} \nu
\Re(\tau^{-1/n})}\qquad \mbox{for all $\tau \in e^{\pm i
n\phi}(0,\infty)$}.
$$
It follows that
\begin{multline*}
\|\hat{U}(\cdot,q)\|_{l^1} \le C \int_{0}^{\infty} e^{-\frac{3}{4} \nu
\Re(\tau^{-1/n}) + q \Re\,\tau}\,d|\tau| \le C \int_{0}^{\infty}
e^{-\frac{3}{4} \nu r^{-1/n} \cos\phi + q r\cos n\phi}\,dr,
\end{multline*}
and a standard application of the Laplace method (with the change of
variable $r = q^{-n/(n+1)} s$) shows that
$$
\|\hat{U}(\cdot,q)\|_{l^1} \le C_1 e^{-C_2 q^{1/(n+1)}}\qquad \mbox{as } q
\to +\infty.
$$

\section{Estimates of $\alpha$ based the solution of (\ref{IntUeqn}) in
$[0, q_0]$}\label{S6}

Define $\hat{U}^{(a)}$ as in (\ref{eq:eq456}) and $\hat{U}^{(b)} =
\hat{U}-\hat{U}^{(a)}$. Using (\ref{IntUeqn}), it is convenient to write
an integral equation for $\hat{U}^{(b)}$ for $q > q_0$:
\begin{equation}\label{e.1}
\hat{U}^{(b)}(k,q) = -ik_{j} \int_{q_{0}}^{q} \mathcal{G}(q,q';k)
\hat{H}_{j}^{(b)}(k,q')\,dq' + \hat{U}^{(s)}(k,q),
\end{equation}
where
\begin{equation}\label{e.2}
\hat{U}^{(s)}(k,q) = -ik_{j} \int_{0}^{\min\{q,2q_{0}\}}
\mathcal{G}(q,q';k) \hat{H}_{j}^{(a)}(k,q')\,dq' + \hat{U}^{(0)}(k,q),
\end{equation}
and
\begin{gather}
\label{e.3} \hat{H}_{j}^{(a)}(k,q) = P_{k} \Bigl[ \hat{v}_{0,j} \hat{*}
\hat{U}^{(a)} + \hat{U}_{j}^{(a)} \hat{*} \hat{v}_{0} + \hat{U}_{j}^{(a)}
\ds \hat{U}^{(a)} \Bigr](k,q), \\
\label{e.4} \hat{H}_{j}^{(b)}(k,q) = P_{k} \Bigl[ \hat{v}_{0,j} \hat{*}
\hat{U}^{(b)} + \hat{U}_{j}^{(b)} \hat{*} \hat{v}_{0} + \hat{U}_{j}^{(a)}
\ds \hat{U}^{(b)} + \hat{U}_{j}^{(b)} \ds \hat{U}^{(a)} +
\hat{U}_{j}^{(b)} \ds \hat{U}^{(b)} \Bigr](k,q).
\end{gather}
Also, we define ${\hat R}^{(b)} (k, q) = -i k_j {\hat H}_j^{(b)} (k,q)$.
It is to be noted that the support of $\hat{H}^{(a)}$ is $[0, 2q_0]$.
Thus, if $\hat{U}^{(a)}$ is known (computationally or otherwise), then
$\hat{H}^{(a)}$ and therefore $\hat{U}^{(s)} $ are known for all $q$.

\noindent{\bf Proof of Theorem \ref{Thm03}:}
Note that
$$
|\hat{R}^{(b)}(k,q)| \leq 2|k| \Bigl[ 2|\hat{v}_{0}| \hat{*}
|\hat{U}^{(b)}| + 2|\hat{U}^{(a)}| \ds |\hat{U}^{(b)}| + |\hat{U}^{(b)}|
\ds |\hat{U}^{(b)}| \Bigr](k,q),
$$
where $|\cdot|$ is the usual Euclidean norm in $\mathbb{R}^{3}$. By Lemma
\ref{L2.4} we can define a best constant
\begin{equation}\label{sizeG}
B_{0}(k) = \sup_{q_{0} \leq q' \leq q} \Big\{ (q-q')^{1/2-1/(2n)}
|\mathcal{G}(q,q';k)| \Big\}
\end{equation}
and conclude that
\begin{multline*}
|\mathcal{G}(q,q';k) \hat{R}^{(b)}(k,q')|  \leq 2|k|B_{0}(k)
(q-q')^{1/(2n)-1/2} \Bigl[ 2|\hat{v}_{0}| \hat{*} |\hat{U}^{(b)}| +
2|\hat{U}^{(a)}| \ds |\hat{U}^{(b)}| \\
+ |\hat{U}^{(b)}| \ds |\hat{U}^{(b)}| \Bigr](k,q').
\end{multline*}
It follows from Lemma \ref{lem0.1} that
$$
\|\mathcal{G}(q,q';\cdot) \hat{R}^{(b)}(\cdot,q')\|_{l^1} \leq \psi(q-q')
\Bigl[ B_{1} u + B_{2}*u + B_{3} u*u \Bigr](q'),
$$
where $\psi(q) = q^{1/(2n)-1/2}$ and
\begin{multline*}
u(q)  = \|\hat{U}^{(b)}(\cdot,q)\|_{l^1},\qquad B_{1}  = 4\sup_{k \in
\mathbb{Z}^{3}} \big\{ |k|B_{0}(k) \big\} \|\hat{v}_{0}\|_{l^1}, \\
B_{2}(q) = 4\sup_{k \in \mathbb{Z}^{3}} \big\{ |k|B_{0}(k) \big\}
\|\hat{U}^{(a)}(\cdot,q)\|_{l^1},\qquad  B_{3}  = 2\sup_{k \in
\mathbb{Z}^{3}} \big\{ |k|B_{0}(k) \big\}.
\end{multline*}

Taking the $l^1$-norm in $k$ on both sides of \eqref{e.1}, multiplying the
equation by $e^{-\alpha q}$ for some $\alpha \geq \alpha_{0} \geq 0$ and
integrating over the interval $[q_{0},M]$, we obtain
\begin{multline*}
L_{q_{0},M} \leq \int_{q_{0}}^{M} e^{-\alpha q} \int_{q_{0}}^{q}
\psi(q-q') \Bigl[ B_{1} u + B_{2}*u + B_{3} u*u \Bigr](q')\,dq'\,dq +
\int_{q_{0}}^{M} e^{-\alpha q} u^{(s)}(q)\,dq \\
\leq \int_{q_{0}}^{M} \Bigl[ B_{1} u + B_{2}*u + B_{3} u*u \Bigr](q')
\int_{q'}^{M} e^{-\alpha q} \psi(q-q') \,dq\,dq' + \int_{q_{0}}^{M}
e^{-\alpha q} u^{(s)}(q)\,dq \\
\leq \int_{0}^{\infty} e^{-\alpha q} \psi(q) \,dq \int_{q_{0}}^{M}
e^{-\alpha q'} \Bigl[ B_{1} u + B_{2}*u + B_{3} u*u \Bigr](q')\,dq' +
\int_{q_{0}}^{M} e^{-\alpha q} u^{(s)}(q)\,dq,
\end{multline*}
where
\begin{equation}\label{defus}
L_{q_{0},M} := \int_{q_{0}}^{M} e^{-\alpha q} u(q)\,dq,\qquad u^{(s)}(q) =
\|\hat{U}^{(s)}(\cdot,q)\|_{l^1}.
\end{equation}
If we use the fact that
\begin{multline*}
\int_{q_{0}}^{M} e^{-\alpha q'} u*v(q')\,dq'  = \int_{q_{0}}^{M}
e^{-\alpha q'} \int_{q_{0}}^{q'} u(s) v(q'-s)\,ds\,dq' \\
= \int_{q_{0}}^{M} u(s) \int_{s}^{M} e^{-\alpha q'} v(q'-s)\,dq'\,ds
\end{multline*}
for any function $v$ on $[0,M]$ (recall that $u = 0$ on $[0,q_{0}]$), then
\begin{multline}\label{ext3}
L_{q_{0},M}  \leq \int_{0}^{\infty} e^{-\alpha q} \psi(q) \,dq \Bigg\{
\biggl[ B_{1} + \int_{0}^{q_{0}} e^{-\alpha q'} B_{2}(q')\,dq' \biggr]
L_{q_{0},M} + B_{3} L_{q_{0},M}^{2} \Bigg\}  \\ + b\alpha^{-1/2-1/(2n)}
\leq \alpha^{-1/2-1/(2n)} \Bigl[ \epsilon_{1} L_{q_{0},M} + \epsilon
L_{q_{0},M}^{2} \Bigr] + b\alpha^{-1/2-1/(2n)},
\end{multline}
where
\begin{gather}
\label{ext3.1} b = \alpha^{1/2+1/(2n)} \int_{q_{0}}^{\infty} e^{-\alpha q}
u^{(s)}(q)\,dq, \\
\label{ext3.2} \epsilon_{1} = \Gamma \biggl( \frac{1}{2}+\frac{1}{2n}
\biggr) \biggl[ B_{1} + \int_{0}^{q_{0}} e^{-\alpha_{0} q'} B_{2}(q')\,dq'
\biggr],\qquad \epsilon = \Gamma \biggl( \frac{1}{2}+\frac{1}{2n}
\biggr)\, B_{3}.
\end{gather}
For
$$
\epsilon_{1} < \alpha^{1/2+1/(2n)}\qquad \mbox{and}\qquad \bigl(
\epsilon_{1}-\alpha^{1/2+1/(2n)} \bigr)^{2} > 4\epsilon b,
$$
this leads to an estimate for $L_{q_{0},M}$ independent of $M$:
\begin{equation}\label{ext11}
L_{q_{0},M} \leq \frac{1}{2\epsilon} \biggl[ \alpha^{1/2+1/(2n)} -
\epsilon_{1} - \sqrt{\bigl( \epsilon_{1}-\alpha^{1/2+1/(2n)} \bigr)^{2} -
4\epsilon b} \biggr].
\end{equation}
So $\|\hat{U}(\cdot,q)\|_{l^1} \in {L}^{1}(e^{-\alpha q}\,dq)$ and the
solution to (\ref{nseq0}) exists for $t \in (0,\alpha^{-1/n})$, if
$\alpha$ is sufficiently large so that
$$
\alpha \geq \alpha_{0},\qquad \alpha^{1/2+1/(2n)} > \epsilon_{1} +
2\sqrt{\epsilon b}.
$$
Alternatively, we may choose $\alpha_0 = \alpha$, in which case $\alpha$
has to be large enough to satisfy:
$$
\alpha^{1/2+1/(2n)} > \epsilon_{1} + 2\sqrt{\epsilon b}.
$$
This completes the proof of Theorem \ref{Thm03}

\subsection{Further estimates on $\epsilon_1$, $b$ and $\epsilon$}
\label{furtherest}

By Lemma \ref{L2.4},
\begin{equation}\label{ext9.0}
c_{g} = \sup_{\substack{k \in \mathbb{Z}^3 \\ q_0 \leq q' \leq q}} \Big\{
|k|\, q^{1/2} (q-q')^{1/2-1/(2n)} |\mathcal{G} (q,q';k)| \Big\} < \infty,
\end{equation}
and by (\ref{defus}), (\ref{e.2}), Lemma \ref{L2.4}, Lemma \ref{lemU0} and
the compact support of $\hat{H}^{(a)}$,
\begin{equation}\label{ext9.1}
c_{s} = \sup_{q_{0} \leq q} \Big\{ q^{1/2-1/(2n)} u^{(s)}(q) \Big\} <
\infty.
\end{equation}
It follows that
\begin{equation}\label{ext9.2}
b \leq c_{s} \Gamma \biggl( \frac{1}{2}+\frac{1}{2n}, \alpha_{0} q_{0}
\biggr),\qquad \epsilon \leq 2 \Gamma \biggl( \frac{1}{2}+\frac{1}{2n}
\biggr) c_{g} q_{0}^{-1/2},
\end{equation}
where
$$
\Gamma(a,x) = \int_{x}^{\infty} t^{a-1} e^{-t}\,dt
$$
is the incomplete Gamma function, and condition (\ref{ext9}) is satisfied if
\begin{equation}\label{ext9.4}
\alpha > \alpha_{0},\qquad \alpha^{1/2+1/(2n)} > \epsilon_{1} + 2 \biggl[
2\Gamma \biggl( \frac{1}{2}+\frac{1}{2n} \biggr) \Gamma \biggl(
\frac{1}{2}+\frac{1}{2n}, \alpha_{0} q_{0} \biggr) c_{g} c_{s}
\biggr]^{1/2} q_{0}^{-1/4}.
\end{equation}
If on a large subinterval $[q_{d}, q_{0}]$, ${\hat U}^{(a)} (\cdot, q)$,
and therefore ${\hat H}^{(a)}$, decays, cf. the exponential decay in
Theorem \ref{Thm02}, then the estimated $c_s$ is small. Also, $\epsilon_1$
in (\ref{ext3.2}) is small for large $q_0$, ultimately since $B_0(k)$ in
(\ref{sizeG}) is small. It is then clear that $\alpha$ in (\ref{ext9.4})
can be chosen small as well.

\section{Control of numerical errors in $[0, q_0]$ in a discretized
scheme}

The errors in a numerical discretization scheme for 3-D Navier-Stokes
cannot be readily controlled since these depend on derivatives of the
classical solution; and these are not known to exist beyond some initial
time interval. In contrast to physical space approaches, the $q$
derivatives of the solution ${\hat U}$ to (\ref{IntUeqn}), are {\it a
priori} bounded on any interval $[q_m, q_0] \subset \mathbb{R}^+$ for $q_m
> 0$, by Lemma \ref{lemqder}. Further, if $q_m$ is chosen appropriately
small, then the small $t$ expansion of NS solution provides an accurate
representation for ${\hat U} $ and therefore of ${\hat H}_j$ in $[0, q_m]$
to any desired accuracy. Calculating the numerical solution to
(\ref{IntUeqn}) with rigorous error control is relevant in more than one
way.

In \S\ref{S6}, we have shown that control of ${\hat U}$ on a finite
$q$-interval provides sharper estimates on the exponent $\alpha$, and
therefore an improved classical existence time estimate for $v$. If this
estimate exceeds $T_c$, the time beyond which Leray's weak solution
becomes classical again (see the Appendix for a bound on $T_c$) then, of
course, global existence of $v$ follows.

Furthermore, a numerical scheme to calculate (\ref{IntUeqn}), which is
analyzed in this section is interesting in its own right. It provides,
through Laplace transform, an alternative calculation method for
Navier-Stokes dynamics. Evidently, this method is not numerically
efficient to determine $v (x, t)$ for fixed time $t$; nonetheless it may
be advantageous in finding long time averages involving $v$ and $\nabla v$
needed for turbulent flow. These can sometimes be expressed as functionals
of ${\hat U}$.

\begin{Definition}\label{defalphadelta}
We introduce a discrete operator $\mathcal{N}_\delta^{(N)}$ by
\begin{multline}\label{discret2}
\left \{ \mathcal{N}_{\delta}^{(N)} [ \hat{V} ] \right \} \left (k,
m\delta \right ) = - i k_j \sum_{m'=m_s}^{m-1} w^{(1)} (m, m'; k, \delta)
\mathcal{P}_N {\hat H}_{j,\delta}^{(N)} (k, m'\delta)
\\
+ {\hat U}^{(0,N)} (k, m \delta) - i k_j w^{(1,1)} (m, k, \delta)
\mathcal{P}_N {\hat H}_{j, \delta}^{(N)} (k, m\delta),
\end{multline}
where $k \in [-N, N]^3 \setminus \{0 \}$, $\NN \ni m \ge m_s$, $q_m = m_s
\delta$ ($q_m$ is independent of $\delta$) and
\begin{equation}\label{eqU0N}
{\hat U}^{(0,N)} (k, m \delta) = {\hat U}^{(0)} (k, m\delta) - i k_j
\int_0^{q_m} \mathcal{G} (m\delta, q'; k) \mathcal{P}_N \hat{H}_j^{(N)}
(k, q') dq'
\end{equation}
is considered known, while for $m' \ge m_s$,
\begin{multline}\label{discret2.H}
{\hat H}_{j,\delta}^{(N)} (k, m'\delta) = \sum_{k' \in [-N, N]^3 \setminus
\{ 0, k \}} P_k \left [ {\hat v}_{0,j} (k')\hat{V}(k-k', m'\delta) + {\hat
v}_{0} (k') \hat{V}_j (k-k', m'\delta) \right ] \\
+ \sum_{\substack{k' \in [-N, N]^3 \setminus \{ 0, k \} \\
m^{\prime\prime}=m_s..,m'-m_s}} P_k \left [ {\hat V}_{j} (k',
m^{\prime\prime} \delta) {\hat V} (k-k', (m'-m^{\prime\prime}) \delta)
\right ] w^{(2)} (m', m^{\prime\prime}; k, \delta ) \\
+ 2\sum_{l=0}^{m_s-1} w^{(2,l)} (m', k, \delta) P_k \left [ {\hat E}^{(l)}
(k) {\hat *} \hat{V} (k, (m'-l) \delta) \right ].
\end{multline}
\end{Definition}
\z In (\ref{discret2.H}), ${\hat E}^{(l)} (k)$ involves ${\hat v}_0
(k)$--this representation encapsulates the singular contribution of ${\hat
U} (\cdot, q')$ and ${\hat U} (\cdot, q-q')$ when $q'$ and $q-q'$ are
small respectively. The precise form of these functions and of the weights
$w^{(1)} (m, m'; k, \delta)$, $w^{(1,1)} (m, k, \delta)$, $w^{(2)} (m',
m''; k, \delta)$ and $w^{(2,l)} (m', k, \delta)$ generally  depend on the
particular discretization scheme employed to calculate
$\mathcal{N}_\delta^{(N)} [ {\hat U} ]$. Also, note that in
(\ref{discret2.H}), the nonlinear terms in the summation are absent when
$m_s \le m' < 2 m_s$. To simplify the discussion, we do not specify the
weights, but only require that they ensure consistency, namely that in the
formal limit  $\delta \to 0$, the discrete operator ${\mathcal
N}_\delta^{(N)}$ becomes $\mathcal{N}^{(N)}$. Based on behavior of the
kernel $\mathcal{G}$, consistency implies that
\begin{multline}\label{wbounds}
|k| |w^{(1)} (m, m'; k, \delta) | \le \frac{C_1 \delta^{1/(2n)}}{m^{1/2}
(m-m')^{1/2-1/(2n)}} \\
|k w^{(1,1)} | \le C_{1,1} \delta^{1/2+1/(2n)} (m \delta)^{-1/2}~,~
|w^{(2)} | \le C_2 \delta ~,~ |w^{(2,l)}| \le C_3 \delta^{1/n}
(l+1)^{-1+1/n}.
\end{multline}
Consider the solution
\begin{equation}\label{discret3}
{\hat U}^{(N)}_\delta (k, m\delta) = \left \{ \mathcal{N}_{\delta}^{(N)}
\left [\hat {U}^{(N)}_\delta \right ] \right \} \left (k, m\delta \right )
~~{\rm for} ~~m_s \le m , ~~k \in [-N, N]^3,
\end{equation}
where as noted before, $q_m = m_s \delta$ is small enough so that the
known asymptotic series of  $\hat{U}$ at $q=0$ can be used to accurately
calculate ${\hat U}^{(N)}$ and ${\hat H}_j^{(N)}$ for $q < q_m$, and thus
of ${\hat U}^{(0, N)}$ and ${\hat E}^{(l)}$ in (\ref{eqU0N}) and
(\ref{discret2.H}).

\begin{Definition}\label{defconsis}
We let
$$
T_{E, \delta}^{(N)} = \mathcal{N}^{(N)} {\hat U}^{(N)} -
\mathcal{N}^{(N)}_\delta {\hat U}^{(N)}
$$
be the truncation error due to $q$-discretization for a fixed number of
Fourier modes, $[-N, N]^3$. The discretization is consistent (in the
numerical analysis sense) if $T_{E, \delta}^{(N)}$ scales with some
positive power of $\delta$ and involves a finite number of derivatives of
${\hat U}$.
\end{Definition}

\begin{Definition}
We define $\| \cdot \|^{(\alpha,\delta)}$, the discrete analog of $ \|
\cdot \|^{(\alpha)}$, as follows:
$$
\| {\hat f} \|^{(\alpha,\delta)} = \sup_{m \ge m_s} m^{1-1/n}
\delta^{1-1/n} (1+m^2 \delta^2) e^{-\alpha m\delta} \|{\hat f} (\cdot,
m\delta) \|_{l^{1}}.
$$
\end{Definition}

\begin{Remark}{\label{remweight}
\rm More specific bounds on the truncation error depend on the specific
numerical scheme. It is however standard for numerical quadratures to
choose the weights $w^{(j)}$ so that $q$-integration is exact on $q \in
[q_m, q_0]$ for a polynomial of some order $l$. For a general ${\hat V}
(\cdot, q)$, the interpolation errors involve $l+1$ $q$-derivatives. Lemma
\ref{lemqder} guarantees that the derivatives of $\hat{U}$ are
exponentially bounded for large $q$. It follows that $\|
T^{(N)}_{E,\delta} \|^{(\alpha, \delta)} \to 0$ as $\delta \to 0$.}
\end{Remark}

\begin{Remark}
In the rest of this section, with slight abuse of notation, we write $*$
for the discrete summation convolution in $q$-space ({\it i.e.} sum over
$m'$) and $\ds$ for the discrete double, Fourier-Laplace, convolution.
Since the rest of the paper deals with discrete systems, this should not
cause confusion.
\end{Remark}

\begin{Lemma}\label{discHj}
For $m \ge m_s$, ${\hat H}_{j,\delta}^{(N)}(\cdot,m\delta)$ satisfies the
following estimate:
\begin{multline}
\| {\hat H}_{j,\delta}^{(N)} (\cdot, m\delta ) \|_{l^{1}} \\
\le C \frac{ e^{\alpha m \delta}}{ (1+m^2 \delta^2) m^{1-1/n}
\delta^{1-1/n}} \|\hat{U}_{\delta}^{(N)}\|^{(\alpha,\delta)} \left \{
\|\hat{U}_{\delta}^{(N)}\|^{(\alpha,\delta)} + \| {\hat v}_0 \|_{l^1}  +
C_E \right \}
\end{multline}
\z for some known constant $C_E$.
\end{Lemma}
\begin{proof}
Using the properties of discrete convolution we see that
\begin{multline*}
\| P_k \left \{ {\hat v}_{0,j} {\hat *} {\hat U}_\delta ^{(N)} + {\hat
U}_{\delta, j}^{(N)} {\hat *} {\hat v}_{0} + {\hat U}_{\delta, j}^{(N)}
\ds {\hat U}_{\delta}^{(N)} \right \} \|_{l^{1}} \\
\le C \biggl \{ \| {\hat v}_0 \|_{l^1} \|{\hat U}_{\delta}^{(N)} (\cdot,
m\delta ) \|_{l^1} + \delta \sum_{m'=m_s}^{m-m_s} \|{\hat
U}_{\delta}^{(N)} (\cdot, m'\delta ) \|_{l^1} \|{\hat U}_{\delta}^{(N)}
(\cdot, (m-m')\delta ) \|_{l^1} \\
+ \delta^{1/n} \sum_{m'=0}^{m_s-1} (m'+1)^{-1+1/n} \|{\hat
E}^{(m')}\|_{l^1} \|{\hat U}_{\delta}^{(N)} (\cdot, (m-m')\delta )
\|_{l^1} \biggr \} \\
\le C \frac{e^{\alpha m \delta} }{m^{1-1/n} \delta^{1-1/n} (1+m^2
\delta^2)} \left \{ \left ( C_{E}+ \| {\hat v}_0 \|_{l^1} \right ) \|{\hat
U}_{\delta}^{(N)} \|^{(\alpha, \delta)} + \left ( \| {\hat
U}_{\delta}^{(N)} \|^{(\alpha,\delta)} \right )^2 \right \},
\end{multline*}
where, by a standard integral estimate,
\begin{multline*}
\delta^{1-1/n} m^{1-1/n} (1+m^2 \delta^2) \sum_{m'=1}^{m-1} \frac{\delta}{
[\delta m' \delta (m-m')]^{1-1/n} (1+\delta^2 {m'}^2) (1+ \delta^2
{(m-m')}^2 ) } < C, \\
\delta^{1-1/n} m^{1-1/n} (1+m^2 \delta^2) \sum_{m'=0}^{m_s-1}
\frac{\delta}{ [\delta (m'+1) \delta (m-m')]^{1-1/n} (1+\delta^2 (m-m')^2)
} < C,
\end{multline*}
for $C$ independent of $m$, $m'$ and $\delta$. In the above estimates we
have used
$$
\|{\hat E}^{(m')}\|_{l^1} \le C_E e^{\alpha_0 m'\delta}\qquad (\alpha_0
\le \alpha)
$$
which can be obtained from the definition of ${\hat E}^{(m')}$.
\end{proof}

\z Define ${\hat H}_{j,\delta}^{(N,1)}$ and ${\hat H}_{j,\delta}^{(N,2)} $
by substituting ${\hat U}^{(N)}_\delta = {\hat U}^{(N,1)}_\delta $ and
${\hat U}_\delta^{(N,2)}$, respectively, in ${\hat H}_{j,\delta}^{(N)}$.
\begin{Lemma}\label{discHjdiff}
For $m \ge m_s$, we have
\begin{multline*}
\| {\hat H}_{j,\delta}^{(N,1)} (\cdot, m\delta ) - {\hat H}_{j,
\delta}^{(N,2)} (\cdot, m\delta) \|_{l^{1}} \\
\le C \frac{ e^{\alpha m \delta}}{ (1+m^2 \delta^2) m^{1-1/n}
\delta^{1-1/n}} \| {\hat U}^{(N,1)}_\delta - {\hat U}^{(N,2)}_\delta
\|^{(\alpha,\delta)} \\
\times \left \{ \| {\hat U}^{(N,1)}_\delta \|^{(\alpha,\delta)} + \| {\hat
U}^{(N,2)}_\delta \|^{(\alpha,\delta)} + \| {\hat v}_0 \|_{l^1} + C_E
\right \}.
\end{multline*}
\end{Lemma}
\begin{proof}
The proof is similar to that of Lemma~\ref{discHj}
\end{proof}

\begin{Lemma}\label{lemTENd}
(i) For $C_4$ defined in (\ref{C4def}), assume $\alpha$ is large enough so
that
\begin{equation}\label{ensure6}
2 C_4 \alpha^{-1/2-1/(2n)} \left ( \left [C_E + \| {\hat v}_0 \|_{l^{1}}
\right ] + 2 \| {\hat U}^{(0,N)} \|^{(\alpha, \delta)} \right ) < 1.
\end{equation}
Then, for any $\alpha^{-1} \ge \delta_0 \ge \delta > 0$,
$\mathcal{N}_{\delta}^{(N)}$ is contractive and there is a unique solution
to ${\hat U}^{(N)}_\delta = \mathcal{N}_{\delta}^{(N)} \left [ {\hat
U}^{(N)}_\delta \right ]$, which satisfies the bounds
$$
\| {\hat U}^{(N)}_\delta  (\cdot, m \delta) \|_{l^1} \le \frac{2 e^{\alpha
m \delta}}{m^{1-1/n} \delta^{1-1/n} (1+m^2 \delta^2) } \| {\hat U}^{(0,N)}
\|^{(\alpha, \delta)}.
$$
(ii) If $\alpha$ is such that
\begin{equation}\label{ensure6.2}
2 C_4 \alpha^{-1/2-1/(2n)} \left ( \left [C_E + \| {\hat v}_0 \|_{l^{1}}
\right ] + 2 \| {\hat U}^{(0,N)} \|^{(\alpha, \delta)} \right ) \leq
\frac{1}{2},
\end{equation}
then
$$
\| {\hat U}^{(N)}_\delta  (\cdot, m \delta) - {\hat U}^{(N)} (\cdot,
m\delta) \|_{l^1} \le \frac{2 e^{\alpha m \delta}}{m^{1-1/n}
\delta^{1-1/n} (1+m^2 \delta^2) } \| T_{E,\delta}^{(N)} \|^{(\alpha,
\delta)}.
$$
\end{Lemma}
\begin{proof}
(i) We have
\begin{multline}\label{C4def}
\| \mathcal{N}_{\delta}^{(N)} [ {\hat U}^{(N)}_\delta ] (\cdot,m\delta)
\|_{l^1} \le \| {\hat U}^{(0,N)}(\cdot,m\delta) \|_{l^1}
\\
+ C \sum_{m'=m_s}^{m-1} \frac{\delta^{1/(2n)}}{m^{1/2}
(m-m')^{1/2-1/(2n)}} \|{\hat H}_{\delta}^{(N)} (\cdot, m'\delta)\|_{l^1} +
C \delta^{1/(2n)} m^{-1/2} \|{\hat H}_{\delta}^{(N)} (\cdot,
m\delta)\|_{l^1} \\
\le \frac{ e^{\alpha m \delta}}{ (1+m^2 \delta^2) m^{1-1/n}
\delta^{1-1/n}} \biggl\{ \|{\hat U}^{(0,N)}\|^{(\alpha,\delta)} \\
+ C_4 \alpha^{-1/2-1/(2n)} \| {\hat U}^{(N)}_\delta \|^{(\alpha,\delta)}
\Bigl( \| {\hat U}^{(N)}_\delta \|^{(\alpha,\delta)} + \| {\hat v}_0
\|_{l^1} + C_E \Bigr) \biggr\},
\end{multline}
where, by a standard integral estimate,
\begin{multline*}
\delta^{1-1/n} m^{1-1/n} (1+m^2 \delta^2) \sum_{m'=m_s}^{m-1} \frac{\delta
e^{\alpha (m'-m) \delta}}{[\delta m]^{1/2} [\delta (m-m')]^{1/2-1/(2n)}
[\delta m']^{1-1/n} (1+m'^2 \delta^2)} \\
\le C \alpha^{-1/2-1/(2n)},
\end{multline*}
and
$$
\frac{\delta^{1/2+1/(2n)}}{[\delta m]^{1/2}} \le C \delta^{1/2+1/(2n)} \le
C \alpha^{-1/2-1/(2n)}.
$$
Thus ${\hat U}_{\delta}^{(N)} = \mathcal{N}_{\delta}^{(N)} \left [ {\hat
U}^{(N)}_\delta \right ]$ has a unique solution such that
$$
\| {\hat U}_{\delta}^{(N)} \|^{(\alpha,\delta)} \le 2 \| {\hat U}^{(0,N)}
\|^{(\alpha,\delta)}.
$$
Hence the first part of the lemma follows.

(ii) Under the assumption,
\begin{multline*}
\| {\hat U}^{(N)} - {\hat U}_\delta^{(N)} \|^{(\alpha, \delta)} \le \|
\mathcal{N}_\delta^{(N)} [ {\hat U}^{(N)} ] - \mathcal{N}_\delta^{(N)} [
{\hat U}_\delta^{(N)} ] \|^{(\alpha, \delta)} + \| T_{E,\delta}^{(N)}
\|^{(\alpha, \delta)} \\
\le \frac{1}{2} \| {\hat U}^{(N)} - {\hat U}_\delta^{(N)} \|^{(\alpha,
\delta)} + \| T_{E,\delta}^{(N)} \|^{(\alpha, \delta)}.
\end{multline*}
So
$$
\| {\hat U}^{(N)} - {\hat U}_\delta^{(N)} \|^{(\alpha, \delta)} \le 2 \|
T_{E,\delta}^{(N)} \|^{(\alpha, \delta)}
$$
and the second part of the lemma follows.
\end{proof}

\noindent{\bf Proof of Theorem \ref{Thm04}:}
Note that
$$
{\hat U}_\delta^{(N)} - {\hat U} = {\hat U}_\delta^{(N)} - {\hat U}^{(N)}
+ {\hat U}^{(N)} - {\hat U}.
$$
From Lemmas \ref{lemTEN} and \ref{lemTENd}, it follows that
$$
\| {\hat U}_\delta^{(N)} - {\hat U} \|^{(\alpha, \delta)} \le 2 \| T_{E,
N} \|^{(\alpha,\delta)} + 2 \| T_{E,\delta}^{(N)} \|^{(\alpha, \delta)} +
\| (I-\mathcal{P}_N) {\hat U} \|^{(\alpha,\delta)},
$$
which tends to zero as $N \to \infty$, $\delta \to 0$, by Lemmas
\ref{lemTEN} and \ref{lemTENd}.

\section{Numerical Method}

In this section we describe a numerical scheme for calculating the
solution $\hat{U}_\delta^{(N)}$ over a fixed interval. The procedure can
be further optimized in a number of ways, such as adapting the quadrature
scheme to the features of the kernel.

\subsection{Outline of the Algorithm}
The main algorithm is summarized as follows:
\begin{verbatim}
  initialization;
  startup routine;
  for each time step
    advance the solution using second order Runge-Kutta integration;
  end
  estimate the error and output the results.
\end{verbatim}

\subsection{Startup Routine}
One difficulty in numerically solving \eqref{IntUeqn} is that the equation
is  singular at $q = 0$. To overcome it, we first compute $\hat{u}$ for
small $t$ by solving \eqref{hatueq} using Taylor expansion:
$$
\hat{u}(k,t) = \sum_{m=1}^{\infty} \hat{c}_{m}(k)\, t^{m},
$$
where
\begin{multline*}
\hat{c}_{1} = \hat{v}_{1}, \\
\hat{c}_{m+1} = \frac{1}{m+1} \Biggl[ -\nu |k|^{2} \hat{c}_{m} -
ik_{j}P_{k}\biggl( \hat{v}_{0,j} \hat{*} \hat{c}_{m} + \hat{c}_{m,j}
\hat{*} \hat{v}_{0} + \sum_{\ell=1}^{m-1} \hat{c}_{\ell,j} \hat{*}
\hat{c}_{m-\ell} \biggr) \Biggr],\ m \geq 1.
\end{multline*}
Then $\hat{U}$ is computed for small $q$ by
$$
\hat{U}(k,q) = \sum_{m=1}^{m_{0}} \hat{d}_{m}(k)\, q^{m/n-1},
$$
where\footnote{Note that
$$
\int_{0}^{\infty} \hat{d}_{m} q^{m/n-1} e^{-q/t^{n}}\,dq = \hat{d}_{m}
t^{m} \int_{0}^{\infty} q^{m/n-1} e^{-q}\,dq = \Gamma \biggl( \frac{m}{n}
\biggr) \cdot \hat{d}_{m} t^{m},
$$
so $\Gamma(m/n) \cdot \hat{d}_{m} = \hat{c}_{m}$.}
$$
\hat{d}_{m} = \frac{\hat{c}_{m}}{\Gamma(m/n)}.
$$

\subsection{Second Order Runge-Kutta Integration}
After computing the solution on $[0,q_{m}]$ for some $q_{m} > 0$ by using
Taylor expansions, we solve the integral equation \eqref{IntUeqn} on
$[q_{m},q_{0}]$ using second order Runge-Kutta (predictor-corrector)
method. Since this numerical scheme is preliminary and far from being
optimized, we do not include the details here.

What is worth  mentioning is the evaluation of the functions $F(\mu)$ and
$G(\mu)$. As shown in earlier sections, both $F$ and $G$ are entire
functions and have power series expansions at $\mu = 0$. For small $\mu$,
these expansions converge very rapidly (super-factorially) and provide an
efficient way to evaluate $F$ and $G$. For large $\mu$, however, the
alternating nature of the expansions raises the issue of catastrophic
cancellation, and it is no longer appropriate to use them for numerical
computation. In this regime we use the asymptotic expansions of $F$ and
$G$, which we derive below.

While the complete asymptotics of $F$ and $G$ can be derived using
Laplace's method, a faster and easier way is to use the differential
equations they satisfy. For example, recall that for $n = 2$,
$$
F(\mu) = \frac{1}{2\pi i} (I_{1} - \bar{I}_{1}) = \frac{1}{\pi} \Im I_{1},
$$
where
\begin{displaymath}
  I_{1} = i \int_{0}^{\infty} r^{-1/2} e^{-r-i\mu r^{-1/2}}\,dr.
\end{displaymath}
It is easy to check that $I_{1}$ satisfies the third-order ODE (the same
equation satisfied by $F$)
$$
\mu I_{1}''' + I_{1}'' - 2I_{1} = 0,
$$
and it has the leading order asymptotics
$$
I_{1} \sim 2\sqrt{\frac{\pi}{3}}\, e^{-z},
$$
where
$$
z = 3 \cdot 2^{-2/3} \mu^{2/3} e^{i\pi/3}.
$$
If we make the change of dependent variable
$$
I_{1} = 2\sqrt{\frac{\pi}{3}}\, e^{-z} J_{1}(z),
$$
then $J_{1}$ must have the form
$$
J_{1}(z) = 1 + \sum_{m=1}^{\infty} a_{m} z^{-m},
$$
and it solves the ODE
$$
J_{1}''' - 3J_{1}'' + \Bigl( 3 + \frac{1}{4z^{2}} \Bigr) J_{1}' -
\frac{1}{4z^{2}}\, J_{1} = 0.
$$
It follows that
$$
F(\mu) \sim \frac{2}{\sqrt{3\pi}} \Im \Bigg\{ e^{-z} \biggl( 1 +
\sum_{m=1}^{\infty} a_{m} z^{-m} \biggr) \Bigg\},
$$
where $a_{1},\ a_{2}$, etc. are determined by the recurrence
\begin{multline*}
a_{0} = 1,\qquad a_{1} = -\frac{1}{12}, \\
a_{m} = -\frac{1}{12m} \biggl[ \Bigl( 12m^{2} - 12m + 1 \Bigr) a_{m-1} +
\Bigl( 4m^{3} - 12m^{2} + 9m - 2 \Bigr) a_{m-2} \biggr],\qquad m \geq 2.
\end{multline*}
Similarly,
$$
G(\mu) \sim -\frac{(4\mu)^{1/3}}{\sqrt{3\pi}} \Im \Bigg\{ e^{-z+i\pi/6}
\biggl( 1 + \sum_{m=1}^{\infty} c_{m} z^{-m} \biggr) \Bigg\},
$$
where
\begin{multline*}
c_{0} = 1,\qquad c_{1} = \frac{5}{12},\qquad c_{2} = -\frac{35}{288}, \\
c_{m} = \frac{1}{24m} \biggl[ \Bigl( -48m^{2} + 60m - 2 \Bigr) c_{m-1} +
\Bigl( -32m^{3} + 108m^{2} - 80m + 9 \Bigr) c_{m-2} \\
+ \Bigl( -8m^{4} + 52m^{3} - 102m^{2} + 67m - 14 \Bigr) c_{m-3}
\biggr],\qquad m \geq 3.
\end{multline*}

\section{Preliminary Numerical Results}

For all computations in this section we take $n = 2$. The numerical
results and computation scheme are preliminary. The algorithm has not been
optimized for efficiency, and not all estimates have been rigorously
analyzed yet, and these will be published elsewhere. Nonetheless, the
partial results show some important features of the integral equation
approach.

\subsection{Test Case}
We first tested our code with the following test function:
\begin{multline*}
\mbox{(Kida flow)}:  v = (v^{(1)},v^{(2)},v^{(3)}), \\
v^{(1)}(x_{1},x_{2},x_{3},t) = \frac{\sin x_{1}}{1+t} (\cos 3x_{2} \cos
x_{3} - \cos x_{2} \cos 3x_{3}), \\
v^{(1)}(x_{1},x_{2},x_{3},t) = v^{(2)}(x_{3},x_{1},x_{2},t) =
v^{(3)}(x_{2},x_{3},x_{1},t).
\end{multline*}
The forcing $f$ corresponding to $v$ was generated with $\nu = 1$ and
equation \eqref{IntUeqn} was solved without the knowledge of $v$. The
computed solution was then compared to $v$.

For this test case, the startup routine computed the solution on $[0,q_m]
= [0,0.2]$ using $m_{0} = 8$ terms and the Runge-Kutta solver advanced the
solution to $q_0 = 1$. $2N = 16$ points (i.e. 8 Fourier modes) were used
in each dimension (excluding the extra points for anti-aliasing).

We computed the solution for different step size $\delta$ and the errors
at $q_0$
$$
e_{\delta} = \max_{x \in \mathbb{T}^{3}} |U_{\delta}^{(N)}(x,q_0) -
U(x,q_0)|
$$
are listed in Table \ref{table.err.acc.test}. To ensure the error decays
at the right order $O(\delta^{2})$, we also included in the table the
numerical order of convergence:
$$
\beta_{\delta} = \log_{2} \frac{e_{2\delta}}{e_{\delta}}.
$$
\begin{table}[h]
\centering \caption{Test case: errors at $q_0$.}
\label{table.err.acc.test}
\begin{tabular}{c||c|c}
  \hline\hline
  $\delta$ & $e_{\delta}$ & $\beta_{\delta}$ \\
  \hline
  $1/20$ & 1.3399e-04 & $-$ \\
  $1/40$ & 3.1987e-05 & 2.07 \\
  $1/80$ & 7.1462e-06 & 2.16 \\
  $1/160$ & 1.3620e-06 & 2.39 \\
  \hline\hline
\end{tabular}
\end{table}

\subsection{Kida Flow}
Now we consider the Kida flow with the initial condition
$$
v_{0}^{(1)}(x_{1},x_{2},x_{3},0) = \sin x_{1} (\cos 3x_{2} \cos x_{3} -
\cos x_{2} \cos 3x_{3}).
$$
We computed the solution for $\nu = 0.1$ with zero forcing to $q_0 = 10$
using $2N = 128$ points in each dimension, and step size $\delta = 0.05$.
The parameters for the startup procedure are the same as before: $q_m =
0.2$ and $m_{0} = 8$. To investigate the growth of the solution
$\hat{U}_{\delta}^{(N)}$ with $q$, we computed the $l^{1}$-norm
$$
\| \hat{U}_{\delta}^{(N)}(\cdot,q) \|_{l^{1}} = \sum_{k \in [-N,N]^{3}}
|\hat{U}_{\delta}^{(N)}(k,q)|
$$
and plotted $\| \hat{U}_{\delta}^{(64)}(\cdot,q) \|_{l^{1}}$ vs. $q$ in
Fig.\ref{fig.kn.kida.acc.0f}. For comparison we also included in
Fig.\ref{fig.kn.kida.acc.0f} a plot of the solution to the original
(unaccelerated) equation.
\begin{figure}[h]
\centering \subfigure[]{
  \psfrag{p}{\tiny $p$}
  \psfrag{kn1}{\tiny $\| \hat{U}_{\delta}^{(64)}(\cdot,p) \|_{l^{1}}$}
  \psfrag{Zero forcing}{\tiny Zero forcing, $\nu = 0.1$}
  \includegraphics[scale=0.7]{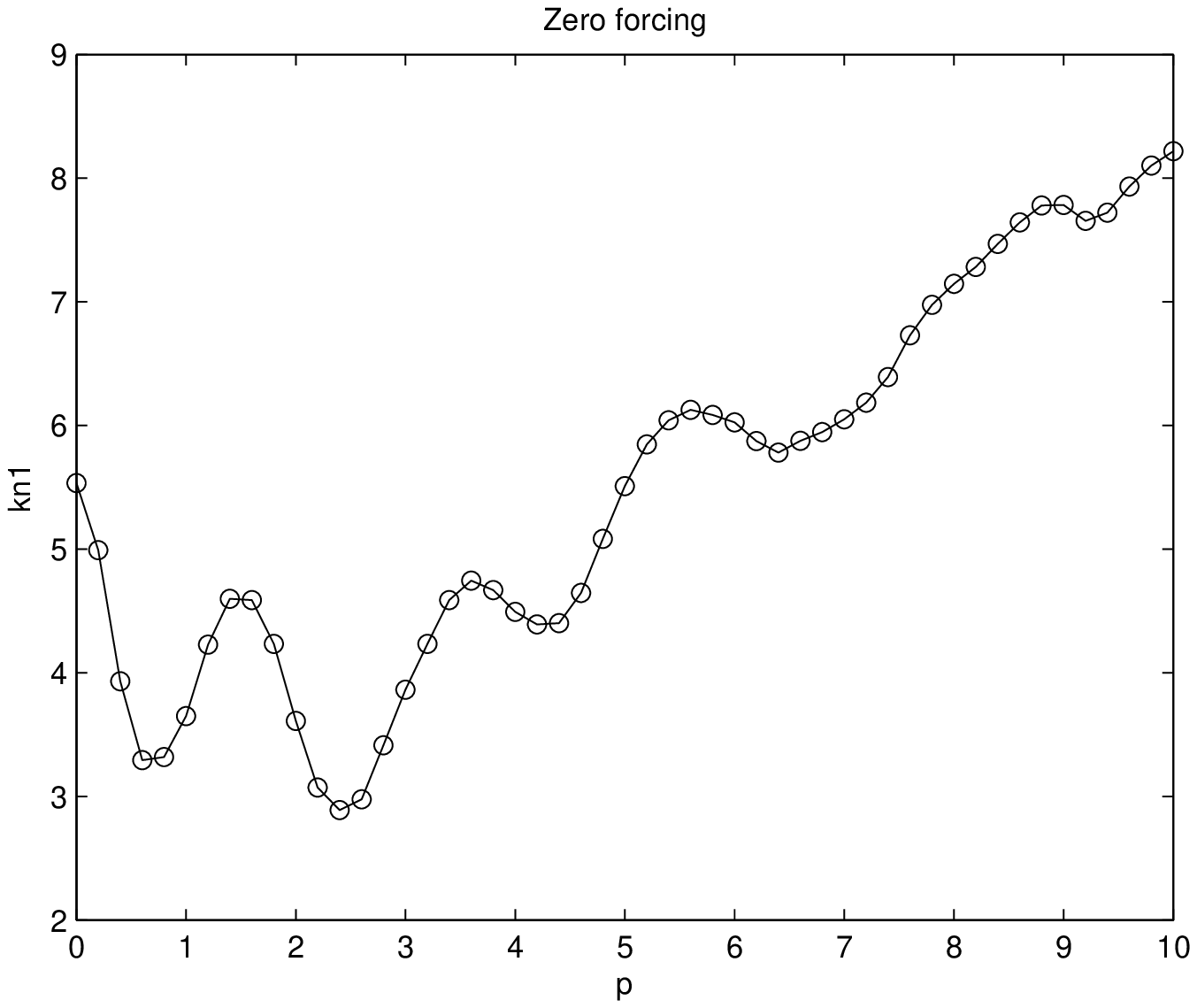}
} \subfigure[]{
  \psfrag{q}{\tiny $q$}
  \psfrag{kn1}{\tiny $\| \hat{U}_{\delta}^{(64)}(\cdot,q) \|_{l^{1}}$}
  \psfrag{Zero forcing}{\tiny Zero forcing, $\nu = 0.1$}
  \includegraphics[scale=0.7]{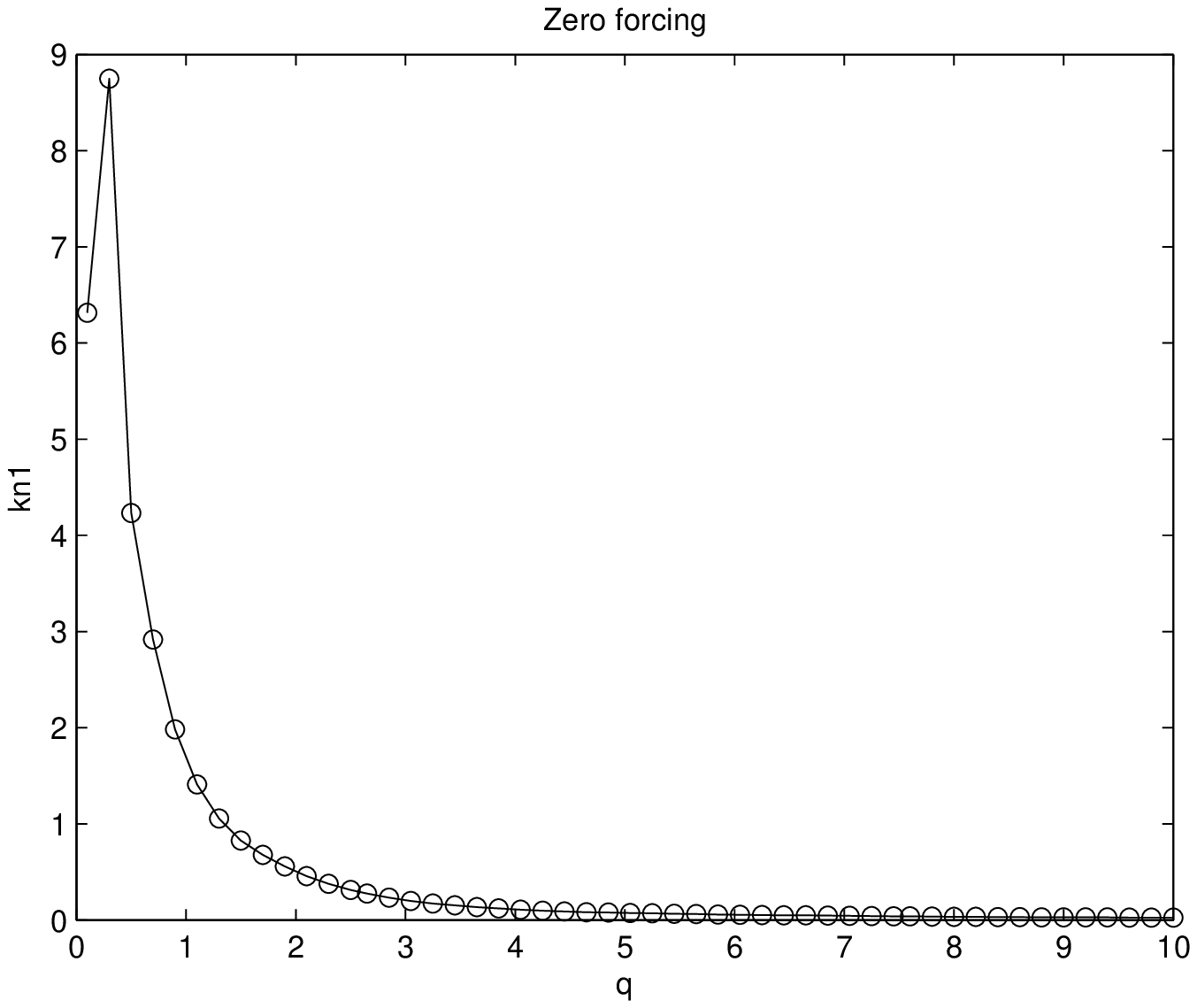}
} \caption{For zero forcing and $\nu = 0.1$: (a). The original
(unaccelerated) equation, $\| \hat{U}_{\delta}^{(64)}(\cdot,p) \|_{l^{1}}$
vs. $p$. (b). Accelerated equation with $n = 2$, $\|
\hat{U}_{\delta}^{(64)}(\cdot,q) \|_{l^{1}}$ vs. $q$.}
\label{fig.kn.kida.acc.0f}
\end{figure}

Fig.\ref{fig.log.kn.kida.acc.0f} shows the plot of $\log \|
\hat{U}_{\delta}^{(64)}(\cdot,q) \|_{l^{1}}$ vs. $q^{1/3}$. Note that $\|
\hat{U}(\cdot,q) \|_{l^{1}} \sim c_{1} e^{-c_{2} q^{1/3}}$ for large $q$,
where $c_{2} = (0.3)^{2/3} 2^{-5/3} 3 \approx 0.42$.
\begin{figure}[h]
\centering \subfigure[]{
  \psfrag{q13}{\tiny $q^{1/3}$}
  \psfrag{log(kn1)}{\tiny $\log \| \hat{U}_{\delta}^{(64)}(\cdot,q) \|_{l^{1}}$}
  \psfrag{Zero forcing}{\tiny Zero forcing, $\nu = 0.1$}
  \includegraphics[scale=0.7]{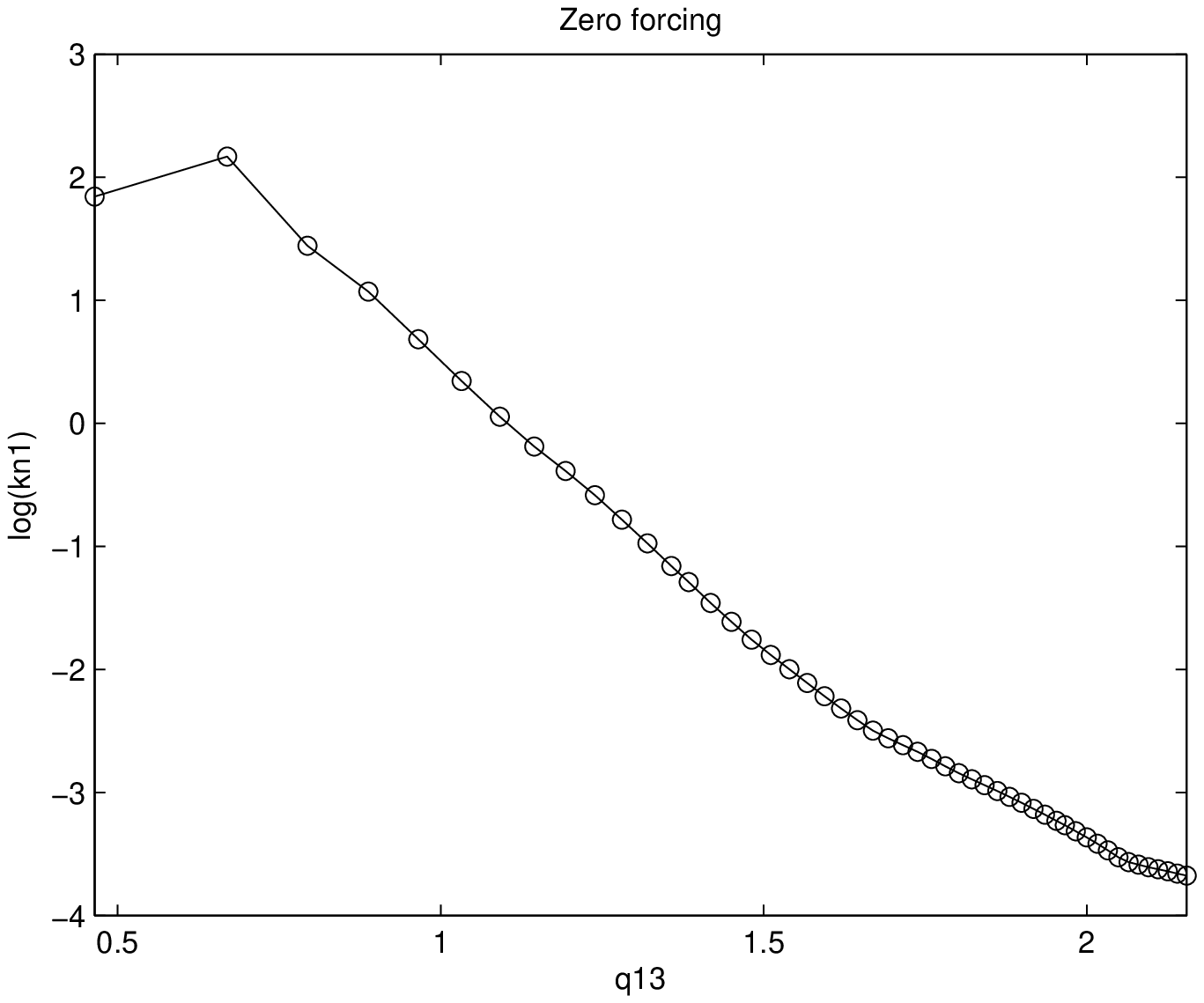}
} \subfigure[]{
  \psfrag{q13}{\tiny $q^{1/3}$}
  \psfrag{dlog(kn1)}{\tiny $\Delta_{-} \Bigl[ \log \| \hat{U}_{\delta}^{(64)}
    (\cdot,s^{3}) \|_{l^{1}} \Bigr] / \Delta s$}
  \psfrag{Zero forcing}{\tiny Zero forcing, $\nu = 0.1$}
  \includegraphics[scale=0.7]{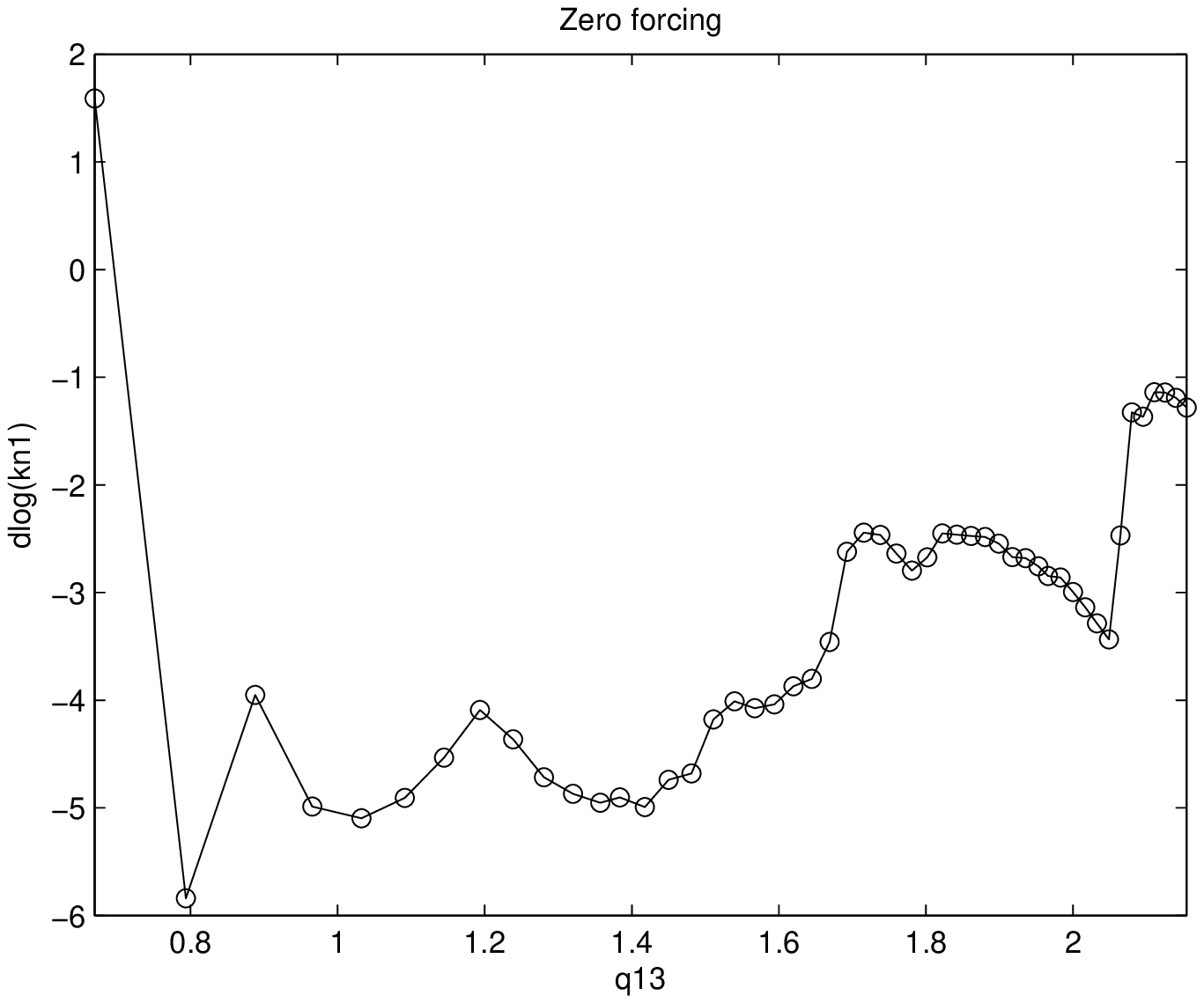}
} \caption{Asymptotic behavior of $\| \hat{U}_{\delta}^{(64)}(\cdot,q)
\|_{l^{1}}$. (a). $\log \| \hat{U}_{\delta}^{(64)}(\cdot,q) \|_{l^{1}}$
vs. $q^{1/3}$. (b) $\Delta_{-} \Bigl[ \log \| \hat{U}_{\delta}^{(64)}
(\cdot,s^{3}) \|_{l^{1}} \Bigr] / \Delta s$ vs. $s$, where $s = q^{1/3}$
and $\Delta_{-}$ is the backward difference operator in $s$.}
\label{fig.log.kn.kida.acc.0f}
\end{figure}

\subsection{Longer Time Existence}
We next computed the constants in estimate \eqref{ext3}. By taking $q_{0}
= 10$ and $\alpha_{0} = 30$, we obtained
$$
b \approx 0,\qquad \epsilon \approx 1.1403,\qquad \epsilon_{1} \approx
13.6921.
$$
This implies the existence of the solution for $\alpha \geq 32.7564$,
which corresponds to an interval of existence $(0,\alpha^{-1/2}) =
(0,0.1747)$.

We compare with a classical estimate of the existence time. The formula
$$
T_{cl} = \frac{1}{c_m\| D^m v_0 \|_{L^2}}
$$
(where $c_m$ is known) was optimized in the range $m>5/2$, giving a
maximal value $T_{cl} \approx 0.01$ at $m \approx 3.2$, about 17 time
shorter than the time obtained from the integral equation.

Furthermore, considerable optimization of code is expected to allow
numerical calculation over much larger $q$-interval.

\section{Appendix}

\subsection{Derivation of the integral equation and of its properties}
\label{A1}

\subsubsection*{The integral equation}
We start with the Fourier transformed equation \eqref{eqn.ns.acc.fu}:
\begin{multline}\label{eqn.ns.acc.fu}
\hat{u}_{t} + \nu |k|^{2} \hat{u}  = -ik_{j} P_{k}[\hat{v}_{0,j} \hat{*}
\hat{u} + \hat{u}_{j} \hat{*} \hat{v}_{0} + \hat{u}_{j} \hat{*} \hat{u}] +
\hat{v}_{1}(k) \\
=: -ik_{j} \hat{h}_{j} + \hat{v}_{1}(k) =: \hat{r} + \hat{v}_{1}(k), \\
\hat{u}(k,0) = 0,
\end{multline}
where
$$
\hat{v}_{1}(k) = \hat{f} (k) - \nu |k|^{2} \hat{v}_{0} - ik_{j}
P_{k}[\hat{v}_{0,j} \hat{*} \hat{v}_{0}].
$$
For $n > 1$, look for a solution in the form
\begin{equation}\label{eqn.ns.acc.int.fu}
\hat{u}(k,t) = \int_{0}^{\infty} \hat{U}(k,q) e^{-q/t^{n}}\,dq
\end{equation}
where
\begin{multline}\label{eqn.ns.acc.int.fr}
\hat{r}(k,t)  = -ik_{j} \hat{h}_{j}(k,t) = -ik_{j} \int_{0}^{\infty}
\hat{H}_{j}(k,q) e^{-q/t^{n}}\,dq \\
=: \int_{0}^{\infty} \hat{R}(k,q) e^{-q/t^{n}}\,dq.
\end{multline}
Inversion of the left side of (\ref{eqn.ns.acc.fu}) and the change of
variable $\tau = t^{-n}$ yield
\begin{multline}\label{A10.88.1}
\hat{u}(k,t)  = \int_{0}^{t} e^{-\nu |k|^{2} (t-s)} \hat{r}(k,s)\,ds +
\int_{0}^{t} e^{-\nu |k|^{2} (t-s)} \hat{v}_{1}(k)\,ds \\
= \int_{0}^{1} t e^{-\nu |k|^{2} t(1-s)} \hat{r}(k,ts)\,ds +
\frac{\hat{v}_{1}(k)}{\nu |k|^{2}} \Bigl( 1-e^{-\nu |k|^{2} t} \Bigr) \\
= \int_{0}^{1} \tau^{-1/n} e^{-\nu |k|^{2} \tau^{-1/n}(1-s)}
\int_{0}^{\infty} \hat{R}(k,q') e^{-q' s^{-n} \tau}\,dq'\,ds +
\frac{\hat{v}_{1}(k)}{\nu |k|^{2}} \Bigl( 1-e^{-\nu |k|^{2} \tau^{-1/n}}
\Bigr) \\
=: I(k,\tau) + J(k,\tau).
\end{multline}
Inverse Laplace transform (formal for now) of $I$ and $J$ yield:
\begin{multline}\label{Gdefine}
\frac{1}{2\pi i} \int_{c-i\infty}^{c+i\infty} I(k,\tau) e^{q\tau}\,d\tau
\\
= \int_{0}^{\infty} \hat{R}(k,q') \int_{0}^{1} \bigg\{ \frac{1}{2\pi i}
\int_{c-i\infty}^{c+i\infty} \tau^{-1/n} e^{-\nu |k|^{2}
\tau^{-1/n}(1-s)+(q-q' s^{-n}) \tau}\,d\tau \bigg\} \,ds\,dq' \\
= \int_{0}^{\infty} \hat{R}(k,q') \int_{0}^{1} (q-q's^{-n})^{1/n-1}
\bigg\{ \frac{1}{2\pi i} \int_{c-i\infty}^{c+i\infty} \zeta^{-1/n}
e^{\zeta-\mu\zeta^{-1/n}}\, d\zeta \bigg\}\,ds\,dq' \\
=: \int_0^\infty \hat{R} (k, q') \mathcal{G} (q, q'; k) dq',
\end{multline}
where
$$
\zeta = (q-q's^{-n}) \tau, \qquad \mu = \nu |k|^{2} (1-s)
(q-q's^{-n})^{1/n},
$$
while
\begin{multline}\label{J123}
\frac{1}{2\pi i} \int_{c-i\infty}^{c+i\infty} J(k,\tau) e^{q\tau}\,d\tau =
\frac{\hat{v}_{1}(k)}{\nu |k|^{2}} \bigg\{ \frac{1}{2\pi i}
\int_{c-i\infty}^{c+i\infty} e^{q\tau} \Bigl( 1-e^{-\nu |k|^{2}
\tau^{-1/n}} \Bigr)\,d\tau \bigg\} \\
= \frac{\hat{v}_{1}(k)}{\nu |k|^{2} q} \bigg\{ \frac{1}{2\pi i}
\int_{c-i\infty}^{c+i\infty} \Bigl( e^{\tilde{\zeta}} -
e^{\tilde{\zeta}-\tilde{\mu}\tilde{\zeta}^{-1/n}} \Bigr)\,d\tilde{\zeta}
\bigg\} =: \hat{U}^{(0)}(k,q),
\end{multline}
where
$$
\tilde{\zeta} = q\tau,\qquad \tilde{\mu} = \nu |k|^{2} q^{1/n}.
$$
The Bromwich contour is homotopic to a contour $C$ from $\infty e^{-i\pi}$
to the left of the origin, ending at $\infty e^{i\pi}$ encircling the
origin, and we finally obtain the integral equation:
$$
\hat{U}(k,q) = \int_{0}^{q} \mathcal{G}(q,q';k) \Bigl[\! -\!ik_{j}
\hat{H}_{j}(k,q') \Bigr]\,dq' + \hat{U}^{(0)}(k,q),
$$
where
$$
\hat{H}_{j}(k,q) = P_{k} \biggl[ \hat{v}_{0,j} \hat{*} \hat{U} +
\hat{U}_{j} \hat{*} \hat{v}_{0} + \hat{U}_{j} \substack{* \\ *} \hat{U}
\biggr](k,q).
$$
Rescaling the integration variable, $s \to s \gamma^{1/n}$, the kernel in
(\ref{Gdefine}) becomes
\begin{multline}\label{Gdefineac}
\mathcal{G}(q,q';k)  = q^{1/n-1} \gamma^{1/n} \int_{1}^{\gamma^{-1/n}}
(1-s^{-n})^{1/n-1} F(\mu)\,ds \\
= \frac{\gamma^{1/n}}{\nu^{1/2} |k| q^{1-1/(2n)}} \int_{1}^{\gamma^{-1/n}}
(1-s^{-n})^{1/(2n)-1} (1-s\gamma^{1/n})^{-1/2} \mu^{1/2} F(\mu)\,ds,
\end{multline}
where
$$
\gamma = \frac{q'}{q},\qquad \mu = \nu |k|^{2} q^{1/n} (1-s\gamma^{1/n})
(1-s^{-n})^{1/n},
$$
and
$$
F(\mu) = \frac{1}{2\pi i} \int_{C} \zeta^{-1/n} e^{\zeta-\mu
\zeta^{-1/n}}\,d\zeta.
$$
Furthermore, from (\ref{J123}) we have
\begin{equation}\label{defU}
\hat{U}^{(0)}(k,q) = \frac{\hat{v}_{1}(k)}{\nu |k|^{2} q} G(\nu |k|^{2}
q^{1/n}),
\end{equation}
where
$$
G(\mu) = -\frac{1}{2\pi i} \int_{C} e^{\zeta-\mu
\zeta^{-1/n}}\,d\zeta,\qquad G(0) = 0.
$$

\subsubsection*{Power series representations of $F$ and $G$}
To show that $F$ is entire, we start with the definition
\begin{equation}\label{defF}
F(\mu) = \frac{1}{2\pi i} \int_{C} \zeta^{-1/n} e^{\zeta} e^{-\mu
\zeta^{-1/n}}\,d\zeta
\end{equation}
and expand $e^{-\mu\zeta^{-1/n}}$ into power series of $\zeta^{-1/n}$ to
obtain
$$
F(\mu) = \frac{1}{2\pi i} \sum_{j=0}^{\infty} \frac{(-1)^{j}}{j!} \mu^{j}
\int_{C} \zeta^{-(j+1)/n} e^{\zeta} \,d\zeta,
$$
where the interchange of order of summation and integration is justified
by the absolute convergence of the series. {F}rom the integral
representation of the Gamma function (see \cite{Abramowitz}) we get
$$
\int_{C} \zeta^{-(j+1)/n} e^{\zeta}\,d\zeta = 2i \sin \biggl(
\frac{j+1}{n}\, \pi \biggr)\, \Gamma \biggl( 1-\frac{j+1}{n} \biggr) =
\frac{2\pi i}{\Gamma((j+1)/n)},
$$
(where in the last step we have used the identity ${\sin(\pi z)}
\Gamma(1-z)\Gamma(z) = {\pi}$) and thus $F$ has the power series
representation
$$
F(\mu) = \sum_{j=0}^{\infty} F_{j} \mu^{j},\qquad \mbox{where } F_{j} =
\frac{(-1)^{j}}{j!\, \Gamma((j+1)/n)}.
$$
Similarly, $G$ is an entire function and has the power series
representation
$$
G(\mu) = \sum_{j=1}^{\infty} G_{j} \mu^{j},\qquad \mbox{where } G_{j} =
-\frac{(-1)^{j}}{j!\, \Gamma(j/n)}.
$$

\subsubsection{The Asymptotics of $F$ and $G$ for  $n \ge 2$ and large
$\mu>0$}\label{A12} Elementary contour deformation and estimates at $0$
show that
$$
F(\mu) = \frac{1}{2\pi i} (I_{1} - \bar{I}_{1}) = \frac{1}{\pi} \Im I_{1},
$$
where
\begin{multline}\label{defI1}
I_1 (\mu) =  \int_{0}^{\infty} r^{-1/n} e^{i\pi/n} \exp\Bigl[ -r-\mu
r^{-1/n} e^{i\pi/n} \Bigr]\,dr \\
= n \mu^{1-2/(n+1)} e^{i\pi/n} \int_{0}^{\infty} s^{n-2} \exp\Bigl[
-\mu^{n/(n+1)} (s^{n} + e^{i\pi/n} s^{-1}) \Bigr]\,ds \\
= n \mu^{1-2/(n+1)} e^{2i\pi/(n+1)} \int_{0}^{\infty} x^{n-2} \exp\Bigl[
-w \varphi(x) \Bigr]\,dx,
\end{multline}
where
$$
w = \mu^{n/(n+1)} e^{i\pi/(n+1)},\qquad \varphi(x) = x^{n} + \frac{1}{x}.
$$
Similarly,
$$
\bar{I}_1 = n \mu^{1-2/(n+1)} e^{-2i\pi/(n+1)} \int_{0}^{\infty} x^{n-2}
\exp\Bigl[ -{\bar{w}} \varphi(x) \Bigr]\,dx.
$$

We now use the Laplace method to obtain the complete asymptotic expansion
of $I_{1}$ for large $w$ with $\arg w \in \left( -\frac{\pi}{2},
\frac{\pi}{2} \right)$ or $\arg \mu \in \left( -\frac{(n+3)\pi}{2n},
\frac{(n-1)\pi}{2n} \right)$. We then show that $I_1$ solves a linear
differential equation. It will follow, from standard results on
asymptotics in ODEs, that the expansion is valid in a wider complex
sector. First, it is easily seen that the only solution to the equation
$$
\varphi'(x) = nx^{n-1} - \frac{1}{x^{2}} = 0
$$
on the positive real axis is $x=x_{0} = n^{-1/(n+1)}$. If we introduce a
new variable
\begin{equation}\label{eqn.xi.x}
\xi = \varphi(x),
\end{equation}
then clearly $\xi$ decreases monotonically from $x = 0^{+}$ to $x =
x_{0}$, where it attains the minimum value
$$
\xi_{0} = \varphi(x_{0}) = n^{-n/(n+1)} (n+1).
$$
We denote this branch of $\varphi^{-1}$ by $x_{1}(\xi)$. Further, as $x$
increases beyond $x = x_{0}$ up to $\infty$, $\xi$ increases from
$\xi_{0}$ to $\infty$. We denote this branch of $\varphi^{-1}$ by
$x_{2}(\xi)$. It follows that
$$
I_{1} = n \mu^{1-2/(n+1)} e^{2i\pi/(n+1)} \Biggl[ -\int_{\xi_{0}}^{\infty}
\frac{x_{1}^{n-2}(\xi)\, e^{-w\xi}}{nx_{1}^{n-1}(\xi) -
x_{1}^{-2}(\xi)}\,d\xi + \int_{\xi_{0}}^{\infty} \frac{x_{2}^{n-2}(\xi)\,
e^{-w\xi}}{nx_{2}^{n-1}(\xi) - x_{2}^{-2}(\xi)}\,d\xi \Biggr].
$$
To find an expansion of $x_{i}(\xi),\ i=1,2$, we note that
$$
\xi - \xi_{0} = \varphi(x) - \varphi(x_{0}) = \sum_{j=2}^{\infty}
\varphi^{(j)}(x_{0}) \frac{(x-x_{0})^{j}}{j!},
$$
and thus
$$
(\xi-\xi_{0}) - \sum_{j=3}^{\infty} \varphi^{(j)}(x_{0})
\frac{(x-x_{0})^{j}}{j!} = \frac{1}{2} \varphi''(x_{0}) (x-x_{0})^{2},
$$
or
\begin{equation}\label{eqnx}
x_{\pm} = x_{0} \pm \sqrt{\frac{2}{\varphi''(x_{0})} \biggl[ (\xi-\xi_{0})
- \sum_{j=3}^{\infty} \varphi^{(j)}(x_{0}) \frac{(x-x_{0})^{j}}{j!}
\biggr]},
\end{equation}
where $x_{-} = x_{1}$ and $x_{+} = x_{2}$. By (\ref{eqnx}) we have
$$
\frac{x_{i}^{n-2}(\xi)}{nx_{i}^{n-1}(\xi)-x_{i}^{-2}(\xi)} =
\sum_{j=-1}^{\infty} b_{j}^{[i]} (\xi-\xi_{0})^{j/2}.
$$
Watson's lemma then implies that
\begin{multline}\label{asympI1}
I_{1}  \sim n \mu^{1-2/(n+1)} e^{2i\pi/(n+1)} e^{-\xi_{0} w}
\sum_{j=-1}^{\infty} \int_{0}^{\infty} \Bigl( b_{j}^{[2]} - b_{j}^{[1]}
\Bigr) \eta^{j/2} e^{-w\eta}\,d\eta\qquad (\eta = \xi-\xi_{0}) \\
\sim n \mu^{1-2/(n+1)} e^{2i\pi/(n+1)} e^{-\xi_{0} w} \sum_{j=-1}^{\infty}
\Bigl( b_{j}^{[2]} - b_{j}^{[1]} \Bigr)\, \Gamma \biggl( 1+\frac{j}{2}
\biggr)\, w^{-1-j/2}.
\end{multline}
We see that
$$
b_{j}^{[2]} - b_{j}^{[1]} =
\begin{cases}
  0 & j \mbox{ even} \\
  2b_{j}^{[2]} & j \mbox{ odd}
\end{cases}.
$$
Similar analysis for $\bar{I}_1$ gives
\begin{multline}\label{asymphatI1}
\bar{I}_{1} \sim n \mu^{1-2/(n+1)} e^{-2i\pi/(n+1)} e^{-\xi_{0} {\bar w}}
\sum_{j=-1}^{\infty} \int_{0}^{\infty} \Bigl( b_{j}^{[2]} - b_{j}^{[1]}
\Bigr) \eta^{j/2} e^{-\bar{w} \eta}\,d\eta \\
\sim n \mu^{1-2/(n+1)} e^{-2i\pi/(n+1)} e^{-\xi_{0} {\bar w}}
\sum_{j=-1}^{\infty} \Bigl( b_{j}^{[2]} - b_{j}^{[1]} \Bigr)\, \Gamma
\biggl( 1+\frac{j}{2} \biggr)\, {\bar w}^{-1-j/2}.
\end{multline}

With $\xi_{0} w$ replaced by $z$, we finally obtain for $\mu$ large and
positive
\begin{multline}\label{eqFasymp}
F(\mu)  = \frac{1}{\pi} \Im I_{1} \\
\sim \frac{n}{\pi} \Im \Bigg\{ \mu^{(n-2)/[2(n+1)]} e^{3i\pi/[2(n+1)]}
e^{-z} \sum_{m=0}^{\infty} 2b_{2m-1}^{[2]} \Gamma \biggl( m+\frac{1}{2}
\biggr)\, \xi_{0}^{m} z^{-m} \Bigg\},
\end{multline}
where
\begin{equation}\label{eqn.xi0.z}
\xi_{0} = n^{-n/(n+1)} (n+1),\qquad z = \xi_{0} \mu^{n/(n+1)}
e^{i\pi/(n+1)}.
\end{equation}
A similar analysis shows that
$$
G(\mu) \sim -\frac{n}{\pi} \Im \Bigg\{ \mu^{n/[2(n+1)]} e^{i\pi/[2(n+1)]}
e^{-z} \sum_{m=0}^{\infty} 2d_{2m-1}^{[2]} \Gamma \biggl( m+\frac{1}{2}
\biggr)\, \xi_{0}^{m} z^{-m} \Bigg\},
$$
where $z,\ \xi_{0}$ are given by \eqref{eqn.xi0.z} and $d_{j}^{[i]}$ are
coefficients of the expansion
$$
\frac{x_{i}^{n-1}(\xi)}{nx_{i}^{n-1}(\xi)-x_{i}^{-2}(\xi)} =
\sum_{j=-1}^{\infty} d_{j}^{[i]} (\xi-\xi_{0})^{j/2}.
$$

To obtain the leading asymptotics of $F$ and $G$, we note that
$$
\frac{x_{2}^{n-2}}{nx_{2}^{n-1}-x_{2}^{-2}} =
\frac{x_{2}^{n-2}}{\varphi'(x_{2})} =
\frac{x_{0}^{n-2}}{\varphi''(x_{0})(x_{2}-x_{0})} + O(1) =
\frac{x_{0}^{n-2}}{\sqrt{2 \varphi''(x_{0})}} (\xi-\xi_{0})^{-1/2} + O(1).
$$
It follows that
$$
b_{-1}^{[2]} = \frac{x_{0}^{n-2}}{\sqrt{2 \varphi''(x_{0})}} =
\frac{1}{\sqrt{2}} n^{3/[2(n+1)]-1} (n+1)^{-1/2}\qquad (\mbox{where }
\varphi''(x_{0}) = n^{3/(n+1)} (n+1)).
$$
Similarly
$$
d_{-1}^{[2]} = \frac{x_{0}^{n-1}}{\sqrt{2 \varphi''(x_{0})}} =
\frac{1}{\sqrt{2}} n^{1/[2(n+1)]-1} (n+1)^{-1/2}.
$$
As a result, we have to the leading order,
\begin{multline*}
F(\mu)  \sim \sqrt{\frac{2}{\pi}}\, n^{3/[2(n+1)]} (n+1)^{-1/2} \Im \Big\{
\mu^{(n-2)/[2(n+1)]} e^{3i\pi/[2(n+1)]} e^{-z} \Big\}, \\
G(\mu) \sim -\sqrt{\frac{2}{\pi}}\, n^{1/[2(n+1)]} (n+1)^{-1/2} \Im \Big\{
\mu^{n/[2(n+1)]} e^{i\pi/[2(n+1)]} e^{-z} \Big\}.
\end{multline*}

\subsubsection*{Differential equations for $F$ and $G$ for $n\in\NN$ and
extended asymptotics} To derive a differential equation satisfied by $F$,
we differentiate (\ref{defF}) $n$ times in $\mu$ (justified by dominated
convergence)
$$
F^{(n)}(\mu) = \frac{(-1)^{n}}{2\pi i} \int_{C} \zeta^{-1/n-1}
e^{\zeta-\mu \zeta^{-1/n}}\,d\zeta.
$$
Integrating by parts once, we get
\begin{multline*}
F^{(n)}(\mu) = \frac{n}{2\pi i} (-1)^{n} \int_{C} \zeta^{-1/n}
e^{\zeta-\mu \zeta^{-1/n}} \biggl( 1 + \frac{\mu}{n} \zeta^{-1/n-1}
\biggr)\,d\zeta \\
= (-1)^{n} n F(\mu) - \mu F^{(n+1)}(\mu),
\end{multline*}
so the differential equation satisfied by $F$ is
$$
\mu F^{(n+1)} + F^{(n)} - (-1)^{n} n F = 0.
$$
Since $G' = F$, the differential equation satisfied by $G$ is
$$
\mu G^{(n+2)} + G^{(n+1)} - (-1)^{n} n G' = 0.
$$
Integrating once and using $G(0) =0$, we obtain
\begin{equation}\label{Gdiff}
\mu G^{(n+1)} - (-1)^{n} n G = 0.
\end{equation}
We can make the same argument for
\begin{equation}\label{defI2}
I_2 (\mu) \equiv \int_{\infty e^{-i\pi}}^{0^+} e^{\zeta - \mu
\zeta^{-1/n}} d\zeta -1
\end{equation}
or for
\begin{equation}\label{defhatI2}
\bar{I}_2 (\mu) \equiv \int_{\infty e^{i\pi}}^{0^+} e^{\zeta - \mu
\zeta^{-1/n}} d\zeta -1.
\end{equation}
It is to be noted that $G(\mu) = \frac{1}{2\pi i} \left [ I_2 (\mu) -
{\bar I}_2 (\mu) \right ]$, while $I_2^\prime (\mu) = I_1 (\mu)$ and
${\bar I}_2^\prime (\mu) = {\bar I}_1 (\mu)$.

Equation (\ref{Gdiff}) has $(n+1)$ independent solutions with the
following asymptotic behavior for large $\mu$ (see \cite{Wasow}):
\begin{equation}\label{genWKB0}
\mu^{n/[2(n+1)]} \exp \left [ - z e^{-i 2 \pi j/(n+1)} \right ] ;\ \ z:=
\xi_0 e^{i \pi/(n+1)} \mu^{n/(n+1)},\ j=0,1,\dotsc,n.
\end{equation}
Thus, there is only one solution with the asymptotic behavior
$$
-\sqrt{\frac{2}{\pi}}\, n^{1/[2(n+1)]} (n+1)^{-1/2} \mu^{n/[2(n+1)]} \exp
\left [ - z \right ]~~{\rm for }~~\arg z = 0
$$
(all solutions independent from it are larger). Since $I_2 (\mu)$ has this
asymptotics in particular for $\arg \mu = - \frac{\pi}{n}$, corresponding
to $\arg z = 0$ as discussed already, $I_2$ is the only solution of
(\ref{Gdiff}) satisfying
\begin{equation}\label{genWKB}
I_2 (\mu) \sim -\sqrt{\frac{2}{\pi}}\, n^{1/[2(n+1)]} (n+1)^{-1/2}
\mu^{n/[2(n+1)]} \exp \left [ -z \right ]~~{\rm for }~~\arg \mu = -
\frac{\pi}{n}.
\end{equation}

As we rotate around in the counter-clockwise direction starting from $\arg
z = 0$ in the complex $z$ (or complex $\mu$) plane, the classical
asymptotics of $I_2$ can only change at antistokes lines. The first
antistokes line is $\arg z = \frac{\pi}{2} + \frac{2\pi}{n+1}$,
corresponding to $\arg \mu = \frac{(n+3)\pi}{2n}$.

Similarly, in a clockwise direction, the first antistokes line is $\arg z
= -\frac{\pi}{2} - \frac{2\pi}{n+1}$, {\it i.e.} $\arg \mu =
-\frac{(n+7)\pi}{2n}$.

Therefore, for $\arg \mu$ $\in$ $\left ( -\frac{(n+7)\pi}{2n},
\frac{(n+3)\pi}{2n} \right )$ the asymptotic expansion $I_2$ is the same.

{F}rom the symmetry between ${\bar I}_2$ and $I_2$, it follows that
\begin{equation}\label{genWKBhat}
{\bar I}_2 (\mu) \sim -\sqrt{\frac{2}{\pi}}\, n^{1/[2(n+1)]} (n+1)^{-1/2}
\mu^{n/[2(n+1)]} \exp \left [ - \bar{z} \right ]
\end{equation}
for $\arg \mu \in \left ( -\frac{(n+3)\pi}{2n}, \frac{(n+7)\pi}{2n} \right
)$. Since $G(\mu) = \frac{1}{2\pi i} \left [ I_2 (\mu) - {\bar I}_2 (\mu)
\right ]$, noting that $I_2 (\mu)$ is dominant for $\arg \mu \in \left (
0, \frac{(n+3)\pi}{2n} \right ) $, it follows that in this range of $\arg
\mu$, $G(\mu) \sim -\frac{i}{2\pi} I_2 (\mu)$. while for $\arg \mu \in
\left ( -\frac{(n+3)\pi}{2n}, 0 \right ) $, since ${\bar I}_2$ is
dominant, $G(\mu) \sim \frac{i}{2\pi} {\bar I}_2 (\mu)$. Lemma~\ref{lemG}
follows.

\subsection{Instantaneous smoothing}
The following result shows that the solution ${\hat v} (k, t)$ obtained
from ${\hat U} (k, q)$  corresponds to a classical solution of
(\ref{nseq0}) for $t \in (0, T]$, {\it i.e.} there is instantaneous
smoothing due to viscous effects. This is a known result (See for instance
\cite{Bertozzi}), but we include it for completeness.

\begin{Lemma}\label{instsmooth}
Assume ${\hat v}_0, {\hat f} \in l^1 (\mathbb{Z}^3)$, where ${\hat v}_0
(0) =0 = {\hat f} (0)$. Assume further that (\ref{nseq0}) has a solution
${\hat v} (k, t) $ with $ \| {\hat v} ( \cdot, t) \|_{l^1} < \infty $ for
$t \in [0, T]$. Then $v (x, t) = \mathcal{F}^{-1} \left [ {\hat v} (\cdot,
t) \right ] (x)$ is a classical solution of  (\ref{nseq0}) for $t \in
\left ( 0, T \right ]$.
\end{Lemma}

\begin{proof}
It suffices to show $|k|^2 {\hat v} (\cdot, t) \in l^1$ for $t \in (0, T]
$ since this implies $v \in {C}^2$ and usual arguments imply that $v$
satisfies (\ref{nseq0}).

Consider the time interval $[\epsilon, T]$ for $\epsilon \ge 0$, $T <
\alpha^{-1/n}$. Define
$$
{\hat w}_\epsilon (k) = \sup_{\epsilon \le t \le T} |{\hat v}| (k, t).
$$
Since $|{\hat v} (k, t)| \le \int_0^\infty | {\hat U} (k, q)| e^{-\alpha
q} dq$, ${\hat w}_0$ (or ${\hat w}_\epsilon$) satisfies
$$
\| {\hat w}_0 \|_{l^1} \le \int_0^\infty \| {\hat U} (\cdot, q) \|_{l^1}
e^{-\alpha q} dq.
$$
On $[\epsilon, T]$ for $\epsilon > 0$, (\ref{nseq}) implies
\begin{equation}\label{intvk}
{\hat v} (k, t) = -i k_j \int_0^t e^{-\nu |k|^2 (t-\tau)} P_k \left (
{\hat v}_j {\hat *} {\hat v}  \right ) (k, \tau) d\tau + {\hat v}_0
e^{-\nu |k|^2 t} + \frac{{\hat f}}{\nu |k|^2} \left (1-e^{-\nu |k|^2 t}
\right ).
\end{equation}
Therefore,
$$
|k| |{\hat v}| (k, t) \le 2 \left \{ {\hat w}_0 {\hat *} {\hat w}_0 \right
\} \int_0^t |k|^2 e^{-\nu |k|^2 (t-\tau)} d\tau  + |k| {\hat v}_0 e^{-\nu
|k|^2 t} + \frac{|{\hat f}|}{\nu |k|} \left (1-e^{-\nu |k|^2 t} \right ).
$$
Since $\int_0^t \nu |k|^2 e^{-\nu |k|^2 (t-\tau)} d\tau \le 1$, it follows
that
\begin{equation}\label{khatw}
|k| {\hat w}_{\epsilon/2} \le \frac{2}{\nu} \left \{ {\hat w}_0 {\hat *}
{\hat w}_0 \right \} + \sqrt{\frac{2}{\nu \epsilon}} \left ( \sup_{\gamma
> 0} \gamma e^{-\gamma^2} \right ) | {\hat v}_0 | + \biggl| \frac{{\hat
f}}{\nu |k|} \biggr|.
\end{equation}
Using now the bounds on ${\hat w}_0$ we get
$$
\| |k| {\hat w}_{\epsilon/2} \|_{l^1} \le \frac{2}{\nu} \left \{ \| {\hat
U} (\cdot, q) \|^{\alpha}_1 \right \}^2  + \frac{C}{\epsilon^{1/2}
\nu^{1/2}} \| {\hat v}_0 \|_{l^1} + \nu^{-1} \biggl\| \frac{{\hat f}}{|k|}
\biggr\|_{l^1}.
$$
The evolution of ${\hat v}$ is autonomous in time, and thus, for $t \in
\left [ \frac{\epsilon}{2}, T \right ]$ we have
\begin{multline}\label{ml1}
{\hat v} (k, t) = -i \int_{\epsilon/2}^t e^{-\nu |k|^2 (t-\tau)} P_k \left
( {\hat v}_j {\hat *} [k_j {\hat v} ] \right ) (k, \tau) d\tau \\
+ {\hat v} (k, \epsilon/2) e^{-\nu |k|^2 (t-\epsilon/2)} + {\hat f} (k)
\frac{1-e^{-\nu |k|^2 (t-\epsilon/2 )} }{\nu |k|^2},
\end{multline}
where we used the divergence condition $k \cdot {\hat v} (k, t) =0$.
Multiplying (\ref{ml1}) by $|k|^2$ and using (\ref{khatw}), it follows
that for $t \in [ \epsilon, T]$ we have
\begin{multline*}
|k|^2 |{\hat v} (k, t)| \le 2 {\hat w}_{\epsilon/2} {\hat *} \left [ |k|
{\hat w}_{\epsilon/2} \right ] \int_{\epsilon/2}^t |k|^2 e^{-\nu |k|^2
(t-\tau)} d\tau \\
+ \frac{1}{\nu (t-\epsilon/2)} \left ( \sup_{\gamma > 0} \gamma
e^{-\gamma} \right ) | {\hat v} (k, \epsilon/2) | +  \frac{| {\hat f}
|}{\nu},
\end{multline*}
implying that
$$
\| |k|^2 {\hat w}_\epsilon \|_{l^1} \le \frac{2}{\nu} \| {\hat
w}_{\epsilon/2} \|_{l^1} \| |k| {\hat w}_{\epsilon/2} \|_{l^1} +
\frac{C}{\epsilon \nu} \| {\hat w}_{\epsilon/2} \|_{l^1}  + \frac{ \|
{\hat f} \|_{l^1} }{\nu}.
$$
Since $\epsilon > 0$ is arbitrary, it follows that $|k|^2 {\hat v} (\cdot,
t) \in l^1$ for $t \in (0, T]$.
\end{proof}

\subsection{Estimate of $T_c$ beyond which Leray's weak solution becomes
classical} It is known that (\ref{nseq}) is equivalent to the integral
equation
\begin{multline}\label{NS1}
{\hat v} (k, t) = \int_0^t e^{-\nu |k|^2 (t-\tau) } P_k \left [ - i k_j
{\hat v}_j {\hat *} {\hat v} \right ] (k, \tau)\,d\tau + e^{-\nu |k|^2 t}
{\hat v}_0
\\
\equiv \mathcal{F} \left \{ \mathcal{N} \left [ v \right ] (\cdot, t)
\right \} (k).
\end{multline}
Applying $\mathcal{F}^{-1}$ in $k$ to (\ref{NS1}), it follows that
\begin{equation}\label{NS1p}
v (x, t) = e^{\nu t \Delta} v_0 - \int_0^t e^{\nu (t-\tau) \Delta}
\mathcal{P} \left [ (v \cdot \nabla) v \right ] \equiv \mathcal{N} \left [
v \right ] (x, t).
\end{equation}

We first determine the value of $\epsilon$ such that, if $\| v_0 \|_{H^1}
\le \epsilon$, then classical solutions $v (\cdot, t)$ to Navier-Stokes
exist for all time. The argument holds for real $t$ as well as in
$$
{\tilde S}_{\tilde \delta} := \Bigl\{ t: \arg t \in ( - {\tilde \delta},
{\tilde \delta} ) \Bigr\},
$$
where $ 0 < {\tilde \delta} < \frac{\pi}{2} $. Sectorial existence of
analytic solution in $t$ with exponential decay for large $|t|$ was useful
in proving Theorem \ref{Thm02}. We denote by $\mathcal{A}_t $ the class of
functions analytic in $t$ for $t \in {\tilde S}_{\tilde \delta}$ for $0 <
|t| < T$.

We consider the space of functions
$$
X \equiv \left \{ \mathcal{A}_{t} H^1_x \right \} \cap \left \{ L_{|t|}^2
H^2_x \right \}:= \Big ( \mathcal{A}_{t} \otimes H^1(\mathbb{T}^3 [0,
2\pi])\Bigr)\,\,\cap \Bigl( L^2 \left [ e^{i \phi} (0,T) \right ] \otimes
H^2(\mathbb{T}^3 [0, 2\pi])\Bigr),
$$
where $t = |t| e^{i \phi} $, $|\phi| < {\tilde \delta}$, and the weighted
norm
$$
\| v \|_{X} = \sup_{t \in {\tilde S}_{\tilde \delta}, 0 < |t| < T} \| e^{
\frac{3}{4} \nu t} v (\cdot, t)\|_{H^1_x} + \sup_{|\phi| < {\tilde
\delta}} \left \{ \int_0^T \| e^{\frac{3}{4} \nu t} v (\cdot, |t| e^{i
\phi}) \|_{H^2_x}^2 d|t| \right \}^{1/2}.
$$
Note that
$$
\| f \|_{H^1_x} = \left ( \sum_{k} (1+|k|^2) |{\hat f} (k)|^2 \right
)^{1/2},\qquad \| f \|_{H^2_x} = \left ( \sum_{k} (1+|k|^4) |{\hat f} (k)
|^2 \right )^{1/2},
$$
and ${\hat f}$ is the Fourier-Transform of $f$.

The arguments below are an adaptation of classical arguments, see
\cite{Tao}. We introduce exponential weights in time, allowing for
estimates independent of $T$, and extend the analysis to a complex sector.

\begin{Lemma}\label{lemu0}
For $v_0 \in H^1_x $, with zero average over $\mathbb{T}^3 [0, 2\pi]$ we
have
$$
\| e^{\nu t \Delta} v_0 \|_X \le c_1 \| v_0 \|_{H^1_x},
$$
where $c_1 = \left ( 1 + \sqrt{ \frac{2}{\nu \cos {\tilde \delta} } }
\right )$.
\end{Lemma}
\begin{proof}
First, take $f = v_0$ and $t \in [0, T]$. Note that zero average implies
${\hat f}(0) = 0$; so we only need to consider $|k| \ge 1$.
\begin{equation}
\lvert e^{\frac{3}{2} \nu t} \rvert \| e^{\nu t \Delta} f \|^2_{H^1_x} \le
\sum_{k \ne 0} (1+|k|^2) e^{-2 \nu (|k|^2-3/4) t} |{\hat f}_k |^2  \le
\sum_{k \ne 0} (1+|k|^2) |{\hat f}_k|^2 = \| f \|_{H^1_x}^2.
\end{equation}
Also, note that
\begin{multline}\label{bbound}
\int_0^T \| e^{\frac{3}{4} \nu t} e^{\nu t \Delta} f \|_{H^2_x}^2 dt \le
\sum_{k \ne 0 } (1+|k|^4) |{\hat f}_k |^2 \left ( \int_0^T e^{-\nu (2
|k|^2 -\frac{3}{2}) t} dt \right ) \\
\le \sum_{k\ne 0} \frac{1+|k|^4}{\nu (2 |k|^2 -\frac{3}{2})} |{\hat f}_k
|^2 \le \frac{2}{\nu} \| f \|^2_{H^1_x}.
\end{multline}
If $t \in {\tilde S}_{\tilde \delta}$, we integrate along the ray $|t|
e^{i \phi}$. It is clear all the steps go through when $\nu$ is replaced
by $\nu \cos \phi$. A bound, uniform in $ {\tilde S}_{\tilde \delta}$, is
obtained by replacing $\frac{2}{\nu}$ in (\ref{bbound}) by $\frac{2}{\nu
\cos {\tilde \delta}}$. The result follows.
\end{proof}

\begin{Lemma}\label{lemuf}
If $e^{\frac{3}{4} \nu t} F \in L_{|t|}^2 L_x^2 $ uniformly in  $\phi \in
(-{\tilde \delta}, {\tilde \delta} )$, then
\begin{equation}
\biggl\| \int_0^t e^{\nu (t-\tau) \Delta} F (x, \tau) d\tau \biggr\|_{X}
\le c_2 \sup_{|\phi| < {\tilde \delta}} \| e^{\frac{3}{4} \nu t} F \|_{
L_{|t|}^2 L_x^2 },
\end{equation}
with
$$
c_2 = \left ( \frac{2 \sqrt{2}}{\sqrt{\nu \cos {\tilde \delta} }} +
\frac{4 \sqrt{2}}{\nu \cos {\tilde \delta} } \right ).
$$
\end{Lemma}
\begin{proof}
We first show this for $t \in [0, T]$. The function
$$
v (x, t) = \int_0^t e^{\nu (t-\tau) \Delta} F (x, \tau) d\tau
$$
satisfies
\begin{equation}\label{A.veq}
v_t - \nu \Delta v = F,\qquad v(x, 0)=0.
\end{equation}
Multiplying (\ref{A.veq}) by $v^*$, the conjugate of $v$, integrating over
$x \in \mathbb{T}^3 [0, 2\pi]$ and combining with the equation for $v^*$
we obtain
\begin{equation}\label{A.veq1}
\frac{d}{dt} \| v (\cdot, t) \|^2_{L^2_x}  + 2 \nu \| D v (\cdot, t)
\|^2_{L^2_x}  \le \frac{4}{\nu} \| F (\cdot, t) \|^2_{L^2_x} +
\frac{\nu}{4} \| v (\cdot, t) \|^2_{L^2_x}.
\end{equation}
Similarly, taking the gradient in $x$ of (\ref{A.veq}), taking the dot
product with $\nabla v^*$ and combining with the equation satisfied by
$\nabla v^*$, we obtain
\begin{equation*}
\frac{d}{dt} \| D v (\cdot, t) \|^2_{L^2_x } + 2 \nu \| D^2 v (\cdot, t)
\|_{L^2_x}^2  = \int_{\mathbb{T}^3} (D F) \cdot (Dv^*) dx +
\int_{\mathbb{T}^3} (D F^*) \cdot (Dv ) dx.
\end{equation*}
Integration by parts and Cauchy's inequality give
\begin{equation}\label{A.veq2}
\frac{d}{dt} \| D v (\cdot, t) \|_{L^2_x }^2 + 2 \nu \| D^2 v (\cdot, t)
\|_{L^2_x}^2  \le \frac{4}{\nu} \| F (\cdot, t) \|^2_{L^2_x} +
\frac{\nu}{4} \| \Delta v (\cdot, t) \|^2_{L^2_x}.
\end{equation}
Combining (\ref{A.veq1}) and (\ref{A.veq2}) and using Poincar\'e's
inequality, we have
\begin{equation}\label{eq1UF}
\frac{d}{dt} \| v (\cdot, t) \|_{H^1_x}^2 + \frac{3}{2} \nu \| v (\cdot,
t) \|_{H^1_x}^2 + \frac{\nu}{4} \| D v (\cdot, t) \|^2_{H^1_x} \le
\frac{8}{\nu} \| F (\cdot, t) \|^2_{L^2_x}.
\end{equation}
Therefore, using (\ref{eq1UF}) and the fact that $v(x, 0)=0$,
$$
\| e^{\frac{3}{4} \nu t} v (\cdot, t) \|_{H^1_x}^2 \le \frac{8}{\nu}
\int_0^t \| e^{\frac{3}{4} \nu \tau} F (\cdot, \tau) \|^2_{L^2_x} d\tau.
$$
Hence,
\begin{equation}\label{A.veq3}
\sup_{t \in [0, T]} \| e^{\frac{3}{4} \nu t} v (\cdot, t) \|_{H^{1}_x} \le
\frac{2\sqrt{2}}{ \sqrt{\nu}} \| e^{\frac{3}{4} \nu t} F \|_{L^2_{|t|}
L^2_x}.
\end{equation}
Integration of (\ref{eq1UF}), using $v(x,0)=0$, gives
$$
\int_0^t \| e^{\frac{3}{4} \nu \tau} v (\cdot, \tau) \|^2_{H^2_x} d\tau
\le \frac{32}{\nu^2} \int_0^t \| e^{\frac{3}{4} \nu \tau} F (\cdot, \tau)
\|^2_{L^2_x} d\tau.
$$
Therefore, for $t \in [0, T]$, we obtain
\begin{equation}\label{A.veq4}
\left [ \int_0^t \| e^{\frac{3}{4} \nu \tau} v (\cdot, \tau) \|^2_{H^2_x}
d\tau \right ]^{1/2} \le \frac{4 \sqrt{2}}{\nu} \left [ \int_0^t \|
e^{\frac{3}{4} \nu \tau} F (\cdot, \tau) \|^2_{L^2_x} d\tau \right
]^{1/2}.
\end{equation}
Now (\ref{A.veq3}) and (\ref{A.veq4}) together imply
\begin{multline}\label{ineq1}
\sup_{t \in [0, T] } \biggl \{ \sum_{k\ne 0} (1+|k|^2) \biggl |
e^{\frac{3}{4} \nu t} t \int_0^1 e^{-\nu |k|^2 t (1-s)} {\hat F} (k, t s)
ds \biggr |^2 \biggr \}^{1/2} \\
+ \biggl \{ \int_0^T d|t| \sum_{k\ne 0} (1+|k|^4) \biggl | e^{\frac{3}{4}
\nu t} t \int_0^1 {\hat F} (k, t s) e^{-\nu |k|^2 t (1-s)} ds \biggr |^2
\biggr \}^{1/2} \\
\le \biggl ( \frac{2 \sqrt{2}}{\sqrt{\nu}} + \frac{4 \sqrt{2}}{\nu} \biggr
) \biggl \{ \int_0^T \sum_{k \ne 0} |e^{\frac{3}{4} \nu t} {\hat F}|^2 (k,
t) |dt| \biggr \}^{1/2},
\end{multline}
and replacing $t \in [0, T]$ by $t \in e^{i \phi} [0, T] \in {\tilde
S}_{\tilde \delta}$ is equivalent to replacing $\nu $ by $\nu \cos {\tilde
\delta}$.
\end{proof}

\begin{Lemma}\label{lemFest}
If $F = - \mathcal{P} \left [ v \cdot \nabla v \right ]$, then for $v \in
X$, and $t \in e^{i \phi} [0, T] \subset {\tilde S}_{\tilde \delta}$,
$$
\sup_{|\phi| < {\tilde \delta}} \| e^{\frac{3}{4} \nu t} F \|_{L^2_{|t|}
L^2_x} \le c_3 \| v \|_{X}^2,
$$
where $c_3 = \frac{c_4^{3/2}}{ (3 \nu \cos {\tilde \delta} )^{1/4} }$ for
$t \in {\tilde S}_{\tilde \delta}$, and $c_4$ is the Sobolev constant
bounding  $\| \cdot \|_{L^6}$ by $\| \cdot \|_{H^1} $ (see for instance
\cite{adam}, page 75).
\end{Lemma}
\begin{proof}
First consider $t \in [0, T]$. H\"{o}lder's inequality implies
$$
\| e^{\frac{3}{4} \nu t} F \|^2_{L^2_{|t|} L^2_x} \le \left [ \int_0^T |
e^{- 3 \nu \tau}| d|\tau| \right ]^{1/2} \left [ \int_0^T \|
e^{\frac{3}{2} \nu \tau} | F (\cdot, \tau) | \|^4_{L^2_x} d|\tau| \right
]^{1/2}.
$$
Hence,
$$
\| e^{\frac{3}{4} \nu t} F \|_{L^2_{|t|} L^2_x} \le \frac{1}{(3
\nu)^{1/4}} \| e^{\frac{3}{2} \nu t} F \|_{L^4_{|t|} L^2_x}.
$$
If we replace $t \in [0, T]$ by $t \in {\tilde S}_{\tilde \delta}$ in this
argument, the effect is simply that  $\frac{1}{(3 \nu)^{1/4}}$ gets
replaced by $ \frac{1}{(3 \nu \cos {\tilde \delta} )^{1/4}}$.

For nonnegative $u$, $w$, repeated use of H\"{o}lder's inequality gives
\begin{multline*}
\int_{\TT^3} w^2 u^2 dx \le \left ( \int_{\TT^3} w^6 dx \right )^{1/3}
\left ( \int_{\TT^3} u^3 dx \right )^{2/3} \\
\le \left \{ \int_{\TT^3} w^6 dx \right \}^{1/3} \left \{ \int_{\TT^3}
u^{2} dx \right \}^{1/2} \left \{ \int_{\TT^3} u^{6} dx \right \}^{1/6}
\le \| w \|_{L_x^6}^2 \| u \|_{L_x^2} \| u \|_{L_x^6}.
\end{multline*}
Therefore, it follows that
$$
\|e^{\frac{3}{2} \nu t} F (\cdot, t) \|_{L^2_x} \le \| e^{\frac{3}{2} \nu
t} |v (\cdot, t)| |\nabla v (\cdot, t)| \|_{L^2_x} \le \| e^{\frac{3}{4}
\nu t} v \|_{L_x^6} \| e^{\frac{3}{4} \nu t} \nabla  v \|_{L_x^2}^{1/2} \|
e^{\frac{3}{4} \nu t} \nabla v \|_{L_x^6}^{1/2},
$$
and
$$
\| e^{\frac{3}{2} \nu t} F \|_{L^4_{|t|} L^2_x} \le \| e^{\frac{3}{4} \nu
t} v \|_{L^\infty_{|t|} L_x^6 } \| e^{\frac{3}{4} \nu t} \nabla v
\|^{1/2}_{L^{\infty}_{|t|} L_x^2} \| e^{\frac{3}{4} \nu t} \nabla v
\|^{1/2}_{L_{|t|}^2 L_x^6}.
$$
Using Sobolev inequalities, we have
$$
\| v (\cdot, t) \|_{L_x^6} \le c_4 \| v (\cdot, t) \|_{H^1_x},
$$
$$
\| D v (\cdot, t) \|_{L_x^6} \le c_4 \| D v (\cdot, t) \|_{H^1_x}.
$$
Thus
$$
\| e^{\frac{3}{2} \nu t} F \|_{L^4_{|t|} L^2_x} \le c_4^{3/2} \|
e^{\frac{3}{4} \nu t} v \|^{3/2}_{L^\infty_{|t|} H^1_x} \| e^{\frac{3}{4}
\nu t} D v \|^{1/2}_{L_{|t|}^2 H^1_x} \le c_4^{3/2} \| v \|_X^2.
$$
Therefore,
$$
\| e^{\frac{3}{4} \nu t} F \|_{L^2_{|t|} L^2_x} \le \frac{c_4^{3/2} }{(3
\nu \cos {\tilde \delta} )^{1/4}} \| v \|_{X}^2.
$$
Since the right hand side is independent of $\phi$, taking the supremum of
the left side over $\phi$ for $|\phi| < {\tilde \delta}$, the Lemma
follows.
\end{proof}

\begin{Lemma}\label{lemN}
The operator $\mathcal{N}$ defined in (\ref{NS1p}) satisfies the following
estimate:
\begin{multline*}
\| \mathcal{N} [ v ] \|_X \le c_1 \| v_0 \|_{H^1_x} + c_2 c_3 \| v \|_X
^2, \\
\| \mathcal{N} [ v^{(1)} ] - \mathcal{N} [ v^{(2)} ] \|_X \le c_2 c_3
\left ( \| v^{(1)} \|_{X} + \| v^{(2)} \|_X \right ) \| v^{(1)} - v^{(2)}
\|_X.
\end{multline*}
\end{Lemma}
\begin{proof}
Note that
$$
\mathcal{N} \left [ v \right ] = e^{\nu t \Delta} v_0 + \int_0^t e^{\nu
(t-\tau)\Delta} F (\cdot, \tau) d\tau,
$$
where $F = -\mathcal{P} \left [ v \cdot \nabla v \right ]$. By Lemmas
\ref{lemu0}, \ref{lemuf} and \ref{lemFest} it follows that
$$
\| \mathcal{N} \left [ v \right ] \|_X  \le c_1 \| v_0 \|_{H^1_x} + c_2
c_3 \| v \|_X^2.
$$
For the second part, we note that
$$
v^{(1)} \cdot \nabla v^{(1)} - v^{(2)} \cdot \nabla v^{(2)} = \Bigl (
v^{(1)} - v^{(2)} \Bigr ) \cdot \nabla v^{(1)} + v^{(2)} \cdot \Bigl (
\nabla v^{(1)} - \nabla v^{(2)} \Bigr ).
$$
Using Lemmas \ref{lemu0}, \ref{lemuf} and \ref{lemFest} again, we obtain
the desired estimate.
\end{proof}

\begin{Lemma}\label{lemepsilon}
If
$$
\| v_0 \|_{H^1_x} < {\hat \epsilon} \equiv \frac{1}{4 c_1 c_2 c_3} =
\frac{ 3^{1/4} \nu^{7/4} [ \cos {\tilde \delta} ]^{7/4} }{ 8 \sqrt{2}\,
c_4^{3/2} (\sqrt{\nu \cos {\tilde \delta}} + \sqrt{2} ) (2+ \sqrt{\nu \cos
{\tilde \delta}}) },
$$
$v (x, t)$ exists in $X$ for any $T$. $v (\cdot, t)$ is analytic in $t \in
{\tilde S}_{\tilde \delta}$ and decays exponentially in that sector as
$|t| \to \infty$, with
$$
\| v  (\cdot, t) \|_{H^1_x}  < 2 c_1 {\hat \epsilon} e^{-\frac{3}{4} \nu
\Re t}.
$$
Further, this solution is smooth in $x$. If
$$
\| v_0 \|_{H^1_x} < \epsilon_0 \equiv \frac{3^{1/4} \nu^{7/4}}{8
\sqrt{2}\, c_4^{3/2} (\sqrt{\nu} + \sqrt{2} ) (2 + \sqrt{\nu} ) },
$$
then $v (x,t)$ is a classical solution for all $t \in \mathbb{R}^+$.
\end{Lemma}
\begin{proof}
If $\| v_0 \|_{H^1_x} < {\hat \epsilon}$, Lemma \ref{lemN} implies that
the operator $\mathcal{N}$ (defined in Lemma \ref{lemN}) is contractive
and hence a solution to Navier-Stokes equation exists in $X$. Since the
estimates are uniform in $t$, it follows that this solution exists for all
$t \in {\tilde S}_{\tilde \delta}$. Known results (or Theorem \ref{Thm01}
above) imply that if the initial data is in $H^1_x$, then the solution
becomes smooth (in fact, analytic for periodic data, \cite{FoiasTem})
instantly, and thus it is a classical solution when $t > 0$. Analyticity
and decay in $t$ follows from  the definition of $X$, the arbitrariness in
the choice of $T$ and the observation that $\mathcal{N}$ in Lemma
\ref{lemN} is contractive in a ball of radius $2 c_1 \|v_0 \|_{H^1_x}$.
Further, by taking the $\lim_{\tilde \delta \to 0^+} {\hat \epsilon} =
\epsilon_0$, we obtain the less restrictive condition on $\| v_0
\|_{H^1_x}$ that ensures existence of classical solution only for $t \in
\mathbb{R}^+$.
\end{proof}

\begin{Lemma}\label{lemH2}
If $\| v_0 \|_{H^2_x} \le \epsilon_2$ for sufficiently small $\epsilon_2$,
$$
\| v (\cdot, t) \|_{H^2_x} \le 2 c_1 \| v_0 \|_{H^2_x} e^{-\frac{3}{4} \nu
\Re t}
$$
for any $t \in {\tilde S}_{\tilde \delta}$.
\end{Lemma}
\begin{proof}
This is similar to the proof of  Lemma \ref{lemepsilon} with $X$ replaced
by
$$
X \equiv \left \{ \mathcal{A}_{t} H^2_x \right \} \cap \left \{ L_{|t|}^2
H^3_x \right \}:= \Big ( \mathcal{A}_{t} \otimes H^2(\mathbb{T}^3 [0,
2\pi])\Bigr)\,\,\cap \Bigl( L^2 \left [ e^{i \phi} (0,T) \right ] \otimes
H^3(\mathbb{T}^3 [0, 2\pi]) \Bigr),
$$
for $|\phi| < {\tilde \delta}$.
\end{proof}

\begin{Theorem}\label{Testimate}
A weak solution to (\ref{nseq0}) becomes classical when $t > T_c$, where
$$
T_c = \frac{256 E c_4^3 (\sqrt{\nu} + \sqrt{2} )^2 (2+\sqrt{\nu})^2 }{
3^{1/2} \nu^{9/2}}.
$$
This solution is analytic in $t$ for $(t - T_{c,a}) \in {\tilde S}_{\tilde
\delta}$, where
$$
T_{c,a} = \frac{256 E c_4^3 (\sqrt{\nu \cos {\tilde \delta}} + \sqrt{2}
)^2 (2+\sqrt{\nu \cos {\tilde \delta} })^2}{ 3^{1/2} \nu [ \nu \cos
{\tilde \delta} ]^{7/2} }.
$$
Further, for any constant $C$, there exists $T_2$ so that for $ (t-T_2)
\in {\tilde S}_{\tilde \delta}$,
$$
\| {\hat v} (\cdot, t) \|_{l^1} < C \exp \left [ -\frac{3}{4} \nu \Re \{
t-T_{2} \}  \right ].
$$
\end{Theorem}

\begin{proof}
Leray's energy estimate implies
$$
\| \nabla v \|_{L^2_{|t|} L^2_x} \le \sqrt{\frac{E}{\nu}},
$$
where $E = \frac{1}{2} \| v_0 \|^2_{L^2_x}$. From a standard pigeon-hole
argument, it follows that there exists $T_1 \in (0, T]$ so that
$$
\| \nabla v (\cdot, T_1) \|_{L^2_x} \le \sqrt{\frac{E}{\nu T}}.
$$
Therefore, Poincar\'e's inequality implies
$$
\| v (\cdot, T_1) \|_{H^1_x} \le \sqrt{\frac{2 E}{\nu T}}.
$$
This means there exists some $T_1 \in [0, T_c]$, where
$$
T_c = \frac{256 E c_4^3 (\sqrt{\nu} + \sqrt{2} )^2 (2+\sqrt{\nu})^2 }{
3^{1/2} \nu^{9/2}}
$$
for which
$$
\| v (\cdot, T_1) \|_{H^1_x} < \frac{ 3^{1/4} \nu^{7/4}}{8 \sqrt{2}
c_4^{3/2} (\sqrt{\nu} + \sqrt{2}) (2+\sqrt{\nu})}.
$$
Replacing $t$ by $t-T_1$ in Lemma \ref{lemepsilon}, we see that the
solution is classical and smooth for $t - T_1 \in \mathbb{R}^+$, therefore
necessarily for $t > T_c$.

Further, from these  arguments, it is clear that there exists a $T_{1,a}
\in  [ 0, T_{c,a} ]$ so that
$$
\| v (\cdot, T_{1,a}) \|_{H^1_x} \le \frac{ 3^{1/4} [\nu \cos {\tilde
\delta}]^{7/4}}{8 \sqrt{2} c_4^{3/2} (\sqrt{\nu \cos {\tilde \delta} } +
\sqrt{2} ) (2+\sqrt{\nu \cos {\tilde \delta}})}.
$$
Replacing $t$ by $t-T_{1,a}$ in Lemma (\ref{lemepsilon}), we see that the
classical solution is analytic in $t - T_{1, a} \in {\tilde S}_{\tilde
\delta}$ (which includes the region $t - T_{c,a} \in {\tilde S}_{\tilde
\delta}$).

Further, since for $t > T_1$ we have
$$
\int_{T_1}^\infty \| v (\cdot, t) \|_{H^2_x}^2 dt \le \sup_{T > T_1} \| v
\|_X^2 \le (2 c_1 \epsilon_0)^{2},
$$
it follows from a pigeon-hole argument that given $\epsilon_2$, there
exists a $T_2 > T_1$ such that
$$
\| v (\cdot, T_2) \|_{H^2_x} < \epsilon_2.
$$
From Lemma \ref{lemH2}, it follows that $v$ exists for $t - T_2 \in
{\tilde S}_{\tilde \delta}$ and
$$
\| v (\cdot, t) \|_{H^2_x} < 2 c_1 \epsilon_2 e^{-\frac{3}{4} \nu \Re
(t-T_2)}.
$$
The last part of the theorem follows from (recall ${\hat v}(0) = 0$)
$$
\| {\hat v} (\cdot, t) \|_{l^1} \le c_5 \| |k|^2 {\hat v} (\cdot, t)
\|_{l^2} \le c_5 \| v (\cdot, t) \|_{H^2_x}.
$$
\end{proof}

\begin{Remark}{
\rm The decay rate $e^{-\frac{3}{4} \nu t}$ for $\| {\hat v} (\cdot, t)
\|_{l^1}$ is not sharp. A more refined argument can be given, to estimate
away the nonlinear terms and obtain a $e^{-\nu t}$ decay.}
\end{Remark}

\section{Acknowledgments.}
This work was supported in part by the National Science Foundation
(DMS-0406193, DMS-0601226, DMS-0600369 to OC and DMS-0405837, DMS-0733778
to S.T). We are grateful to P. Constantin for giving us useful references
and to Alexey Cheskidov for pointing out that estimates of the time beyond
which weak Leray solutions becomes classical are easy to obtain.

\vfill \eject
\end{document}